\newif\ifARXIV
\ARXIVtrue

\ifARXIV
\documentclass{article}
\else
\documentclass[sn-mathphys,Numbered]{sn-jnl}%
\fi

\ifARXIV
\usepackage{libertinus}
\usepackage{authblk}
\usepackage{url}
\usepackage{etoolbox} 

\fi

\usepackage[utf8]{inputenc} %
\usepackage{multicol}
\usepackage{amsmath,amssymb,amsfonts}
\usepackage{amsthm}
\usepackage{mathrsfs}
\usepackage[title]{appendix}
\usepackage{xcolor}
\usepackage{textcomp}
\usepackage{manyfoot}
\usepackage{booktabs}
\usepackage{algorithmicx}
\usepackage{algpseudocode}
\usepackage{listings}
\usepackage{hyperref}
\usepackage[capitalise, noabbrev]{cleveref}
\usepackage{fullpage}
\usepackage{graphics}
\usepackage{graphicx}
\graphicspath{{./images}}
\usepackage{float}
\usepackage[inkscapelatex=false]{svg}
\usepackage{url}
\usepackage{nicefrac}
\usepackage{microtype}
\usepackage{stmaryrd}
\usepackage{pst-node}
\usepackage{tikz-cd}
\usepackage{dsfont}
\usepackage{tipa}
\usepackage{tikz}
\usetikzlibrary{cd}
\usepackage[noend]{algorithm2e}
\RestyleAlgo{ruled}
\usepackage{verbatim}
\usepackage{subcaption}
\usepackage{xspace}
\usepackage{enumitem}
\setlist[itemize]{left=.1cm, itemsep=5pt} %
\setlist[enumerate]{ left=.1cm} %

\newcommand{\svg}[3]{
\begin{figure}[h]
\centering
\includegraphics[width=#2\textwidth]{#1}
\caption{#3}
\end{figure}
}

\usepackage[utf8]{inputenc}
\usepackage[sort&compress]{natbib}
\setcitestyle{numbers,open={(},close={)}} %
\usepackage{doi}

\newcommand{\field}{k}

\newcommand{\norm}[1]{\left\lVert #1 \right\rVert}

\newcommand{\pairing}{{\rm Low}}
\newcommand{\reduced}{{\rm bd}}

\newcommand{\fibered}[1]{ {\mathcal F}{\mathcal B}(#1) }
\newcommand{\barcode}[1]{ {\mathcal B}\left(#1\right) }
\newcommand{\conv}[1]{ {\rm conv}(#1) }
\newcommand{\dir}[1]{ {\rm dir}(#1) }
\newcommand{\codir}[1]{ {\rm codir}(#1) }
\newcommand{\induced}[1]{\field^{#1}}

\newcommand{\simpcomp}{S}
\newcommand{\matching}{\sigma}
\newcommand{\summand}{M}
\newcommand{\interval}{I}
\newcommand{\R}{\mathbb{R}}
\newcommand{\Rbb}{\mathbb{R}}
\newcommand{\Mbb}{\mathbb{M}}
\newcommand{\Ibb}{\mathbb{I}}

\newcommand{\N}{\mathbb{N}}
\newcommand{\Z}{\mathbb{Z}}

\newcommand{\M}{{\mathbb M}}

\newcommand{\disti}{{d_{\rm I}}}
\newcommand{\distb}{{d_{\rm B}}}

\newcommand{\leqn}{\leq_n}
\newcommand{\geqn}{\geq_n}
\newcommand{\sln}{<_n}
\newcommand{\sgn}{>_n}

\newcommand{\vine}{\mathrm{vine}}

\newcommand{\upperb}[1]{ U[#1] }
\newcommand{\lowerb}[1]{ L[#1] }
\newcommand{\birthcp}[1]{ C_B(#1) }
\newcommand{\deathcp}[1]{ C_D(#1) }
\newcommand{\birthLcp}[1]{ C_B^L(#1) }
\newcommand{\deathLcp}[1]{ C_D^L(#1) }
\newcommand{\supp}[1]{ {\rm supp}(#1) }
\newcommand{\dimension}[1]{ {\rm dim}(#1) }
\newcommand{\rectangle}[2]{ R_{#1, #2} }
\newcommand{\interior}[1]{{\kern0pt#1}^{\mathrm{o}}}
\newcommand\restr[2]{{  \left.\kern-\nulldelimiterspace #1 \vphantom{\big|} \right|_{#2}  }}
\DeclareRobustCommand*{\ora}{\overrightarrow}
\newcommand{\birthptapprox}{b_l}

\newcommand{\RIVET}{\texttt{Rivet}}
\newcommand{\MMA}{\texttt{MMA}}

\newcommand{\mMMA}{\mathtt{MMA}}

\newcommand{\bound}{}
\newcommand{\twobound}{2}

\newcommand{\mmaout}{candidate}
\newcommand{\mmaoutcompat}{approximate}

\newcommand{\compacityassumption}{compactly characterizes}
\newcommand{\Lan}{\ensuremath{\mathrm{Lan}_n}\,}

\newcommand*{\vcenteredhbox}[1]{\begingroup
	\setbox0=\hbox{#1}\parbox{\wd0}{\box0}\endgroup}

\newcommand{\bluerectangletwo}{\vcenteredhbox{\includesvg[width=2em]{symbols/solo_blue_rectangle2.svg}\xspace}}
\newcommand{\bluerectangletwotilde}{\vcenteredhbox{\includesvg[width=2em]{symbols/solo_blue_rectangle2tilde.svg}}}
\newcommand{\redz}{\vcenteredhbox{\includesvg[width=2em]{symbols/solo_red_z.svg}\xspace}}
\newcommand{\redztilde}{\vcenteredhbox{\includesvg[width=2em]{symbols/solo_red_ztilde.svg}\xspace}}
\newcommand{\ordonly}[0]{\ensuremath{\mathrm{ord}}}
\newcommand{\ord}[1]{\ensuremath{\ordonly\left(#1\right)}}

\newtheorem{theorem}{Theorem}%
\newtheorem{proposition}[theorem]{Proposition}%

\newtheorem{example}{Example}%
\newtheorem{remark}{Remark}%

\newtheorem{definition}{Definition}%

\newtheorem{lemma}[theorem]{Lemma}

\usepackage[loadshadowlibrary]{todonotes}

\newcommand{\rebuttal}[1]{{ #1}}

\ifARXIV{}
\title{Multi-parameter Module Approximation: an efficient and interpretable invariant for multi-parameter persistence modules with guarantees}
\date{}
\else
\title[\rebuttal{Multi-parameter Module Approximation: an efficient and interpretable invariant for multi-parameter persistence modules with guarantees}]{Multi-parameter Module Approximation: an efficient and interpretable invariant for multi-parameter persistence modules with guarantees}
\fi

\ifARXIV{}
\author[1,3]{David Loiseaux\footnote{Corresponding author}}
\author[1]{Mathieu Carrière}
\author[2]{Andrew J. Blumberg}
\affil[1]{{\small\texttt{surname.name@inria.fr},      DataShape, Centre Inria d'Université Côte d'Azur, France}}
\affil[2]{{\small\texttt{surname.name@columbia.edu}, Department of Mathematics, Columbia University, USA}}
\affil[3]{{\small LIX, CNRS, École Polytechnique, IP Paris,  France}}
\else
\author[1]{David Loiseaux\footnote{Corresponding author}}
\author[2]{Mathieu Carrière}
\author[3]{Andrew J. Blumberg}
\affil[1]{\texttt{david.loiseaux@inria.fr},      DataShape, Centre Inria d'Université Côte d'Azur, France}
\affil[2]{\texttt{mathieu.carriere@inria.fr},    DataShape, Centre Inria d'Université Côte d'Azur, France}
\affil[3]{\texttt{andrew.blumberg@columbia.edu}, Department of Mathematics, Columbia University, USA}
\fi

\begin{document}
\maketitle

\begin{abstract}

Topological data analysis (TDA) is a rapidly growing area of data
science, whose
most common descriptor is %
persistent homology, which tracks the topological changes in growing
families of subsets of the data set itself, called filtrations, and
encodes them in an algebraic object, called a persistence module.  The
algorithmic and theoretical properties of persistence modules are now
well understood in the single-parameter case, that is, when there is
only one filtration (e.g., feature scale) to study.  In contrast, much
less is known in the multi-parameter case, where several filtrations
(e.g., scale and density) are used simultaneously. 
Since multi-parameter persistence modules usually encode information that is invisible to their single-parameter counterparts, %
it is critical to build
tractable 
proxies for them, \rebuttal{ideally with some theoretical robustness guarantees}. %

In this article, we introduce a new parameterized family of topological 
\rebuttal{descriptors},
taking the form of {\em \rebuttal{\mmaout{}} decompositions}, for multi-parameter persistence modules, \rebuttal{and we a identify a subfamily of these descriptors, that we call {\em \mmaoutcompat{} decompositions}, that}
are controllable approximations, \rebuttal{in the sense that they preserve}
diagonal barcodes. 
Then, we introduce \MMA{} (Multipersistence Module Approximation): an algorithm \rebuttal{based on matching functions} for computing %
instances of \rebuttal{\mmaout{} decompositions with some precision parameter $\delta > 0$.} %
By design, \MMA{} can handle an arbitrary number of filtrations,
and has bounded complexity and running time. 
\rebuttal{Moreover, we prove the robustess of \MMA{}: when computed with so-called {\em compatible} matching functions, we show that \MMA{} produces \mmaoutcompat{} decompositions (and we prove that such matching functions exist for $n=2$ filtrations). Next,} 
\rebuttal{we restrict the focus on}
modules that can be decomposed into interval summands. In that case,
\rebuttal{compatible matching functions always exist, and we show that, for small enough $\delta$, the \mmaoutcompat{} decompositions obtained with such compatible matching functions by \MMA{} have an
approximation error %
(in terms of the standard interleaving and bottleneck distances) that is bounded by $\delta$, and that reaches zero for an even smaller, positive precision $\delta_{\rm exact}$.}
 Finally, we present empirical evidence validating that \MMA{} has state-of-the-art performance and running time on several data sets.
\end{abstract}

\ifARXIV{}
\hfill{{\footnotesize\textbf{Keywords:} Multiparameter Persistent Homology, Persistence Modules, Interval Modules, Approximation Methods, Convergence Analysis}}
\else
\fi

\section{Introduction}

Topological Data Analysis (TDA)~\cite{edelsbrunnerComputationalTopologyIntroduction2010,
oudotPersistenceTheoryQuiver2015}
is a new and rapidly developing area of data science that has seen a lot of interest due to its success in various applications,
ranging from
bioinformatics~\cite{Rabadan2019} to material science~\cite{Buchet2018}. 
The main computational tool of TDA is {\em persistent homology} (PH).  
Whereas homology is a qualitative descriptor of the shape of a topological space $\simpcomp$,
the core idea of PH is to capture how the homology groups change 
when computed on a {\em filtration} of $\simpcomp$. 
A filtration is a family $\{\simpcomp_x\subseteq \simpcomp\}_{x\in A}$ %
of subspaces of $\simpcomp$ indexed over a partially ordered set (poset) $A$,
that is {\em nested w.r.t. inclusion}, i.e., it satisfies
$\simpcomp_x\subseteq \simpcomp_y$ for any $x\leq y$. 
Then, 
the functoriality of homology and these inclusions induce morphisms between the corresponding 
homology groups $H_*(\simpcomp_x) \rightarrow H_*(\simpcomp_y)$ for each pair
$x\leq y$, which allows to detect the differences in homology when going from
index $x$ to index $y$.
One of the most common ways to produce such filtrations %
is to study the {\em sublevel sets} of a continuous {\em filter} function $f:\simpcomp\rightarrow \R^n$, 
defined with $\simpcomp_x=f^{-1}(\{x'\in \R^n : x'\leqn x\})$; %
the partial order on the poset $ \R^n$ (denoted by $\leqn$) is defined, for $x,y\in\R^n$, as $x \leqn y$ if and only if  $x_i \leq y_i$ for every dimension $i$. 

\paragraph*{Single-parameter PH.} 
When $\mathcal A$ is totally ordered, e.g., when $\mathcal A\subseteq \R$,
then applying the homology functor $H_*(-;\field)$ for a field $\field$ to a (single-parameter) filtration 
results in a sequence of vector spaces 
connected by linear maps, called a {\em single-parameter persistence module}.  This situation has been studied extensively
in the TDA literature~\cite{carlssonTopologyData2009, Chazal2016, edelsbrunnerComputationalTopologyIntroduction2010, oudotPersistenceTheoryQuiver2015}. 
Notably, one can show that such 
persistence modules can always be decomposed into a direct sum of simple {\em interval summands}: $\M\simeq \bigoplus_{i\in \mathcal I} \rebuttal{\field^{\interval(b_i,d_i)}}$, where
each interval summand \rebuttal{$\field^{\interval(b_i,d_i)}$}
intuitively represents the lifetime \rebuttal{$\interval(b_i,d_i)$} of a topological \rebuttal{feature}, i.e., $b_i$ is the appearance time (birth) and $d_i$ is the disappearance time (death) of  a
topological \rebuttal{feature}, that is detected by 
homology as the index increases. Moreover, single-parameter persistence modules can be efficiently represented in a compact descriptor 
called the {\em persistence barcode}, and several representation methods, such as Euclidean embeddings and kernels for machine learning classifiers, have been
proposed for such barcodes in the literature~\cite{bubenikStatisticalTopologicalData2015, adamsPersistenceImagesStable2017, reininghausStableMultiscaleKernel2015, carriereSlicedWassersteinKernel2017, carrierePersLayNeuralNetwork2020}.  As a consequence, most applications of 
TDA use single-parameter persistence modules, and often use the sublevel sets of, e.g., the data set scale,
as the corresponding single-parameter filtration.

\paragraph*{Multi-parameter PH.} 
However, many data sets come with not just one, but multiple, possibly intertwined,
salient filtrations.  For example, 
image data typically has both a spatial filtration and an intensity filtration, and arbitrary point cloud 
data can be filtered both by feature scale and density.  
Unfortunately, in general, the resulting {\em multi-parameter persistence modules}, obtained by applying the 
homology functor to a filtration indexed over $\R^n$~\cite{botnanIntroductionMultiparameterPersistence2023}, are much 
less tractable; in contrast to the single-parameter case, there is no decomposition theorem that can break down any module into 
a direct sum of simple indicator summands (e.g., interval modules).
Instead, there is now a rich literature on general decompositions into arbitrarily complicated summands~\cite{deyGeneralizedPersistenceAlgorithm2022, deyDecomposingMultiparameterPersistence2025} and their associated minimal presentations~\cite{fugacciChunkReductionMultiParameter2019, kerberFastMinimalPresentations2021, lesnickComputingMinimalPresentations2022}, 
on the theoretical study of a few restricted cases (such as \rebuttal{some specific} $n=2$-parameter filtrations or \rebuttal{exact p.f.d. (\cref{def:pfd}) $2$-parameter persistence modules}) where simple 
decompositions can be obtained~\cite{cochoyDecompositionExactPFD2020,
		botnanRectangleDecomposable2ParameterPersistence2022, 
        asashibaIntervalDecomposability2D2022,
        botnanLocalCharacterizationsDecomposability2023,
        leboviciLocalCharacterizationBlockdecomposability2024},
\rebuttal{and on simpler representations of multi-parameter persistence modules, such as the Euler characteristic}~\cite{hacquardEulerCharacteristicTools2024}, \rebuttal{the fibered barcode}~\cite{corbetKernelMultiparameterPersistent2019, vipondMultiparameterPersistenceLandscapes2020}, \rebuttal{the signed barcode}~\cite{botnanSignedBarcodesMultiparameter2024, loiseauxStableVectorizationMultiparameter2023}, \rebuttal{and the (generalized) rank invariant and persistence diagram}~\cite{kimGeneralizedPersistenceDiagrams2021, xinGRIL2parameterPersistence2023}. 
It has thus become crucial to define general topological \rebuttal{descriptors} for multi-parameter persistence modules that
are meaningful, visually interpretable, and easily computable.

\paragraph*{Contributions.} In this article, we introduce 
new \rebuttal{descriptors} for multi-parameter persistence modules (\rebuttal{ computable from} multi-parameter filtrations of  simplicial complexes\footnote{\rebuttal{Or even solely from module presentations, see~\cref{rem:free_pres}}.}) along with a new algorithm that we call \MMA{} (Multipersistence Module Approximation) for their practical computations.

Before going into our detailed contributions, we
provide a gist of the strategy used to define our descriptors with our
new algorithm \MMA{}.
For simplicity, let us start with an interval decomposable module $\Mbb$
(see \cref{def:interval_module_multi})
that is a direct sum of %
two interval summands %
$\Mbb = \field^{\bluerectangletwo} \oplus \field^{\redz}$.
See \cref{fig:mma_intuition}.

\begin{enumerate}
	\item We fix a grid of diagonal lines $L$ spaced by $\delta > 0$, and compute the %
	      barcodes %
	      associated to ${(\restr \Mbb l)}_{l \in L}$. 
    \item Using some matching function $\sigma$, we match bars of consecutive barcodes together. %
	\item We estimate an interval decomposable module
        $\tilde \Mbb = \field^{\bluerectangletwotilde} \oplus \field^{\redztilde}$ from these matched
	      barcodes. By design, $\tilde \Mbb$ has the same barcodes along $L$ than $\Mbb$.
\end{enumerate}
\begin{figure}[H]
	\centering
	\foreach \n in {0,2,4,5}{
			{\includesvg[width=.24\textwidth]{mma_schema/ht\n}}%
		}
	\caption[\MMA{} algorithm intuition]{The different steps of \MMA{} for computing a candidate decomposition of the module ${\Mbb = k^{\bluerectangletwo} \oplus k^{\redz}}$.
	}\label{fig:mma_intuition}
\end{figure}

Our contributions are five-fold:
\begin{enumerate}
    \item {\bf We introduce a new 
    family of topological \rebuttal{descriptors} for \rebuttal{finitely presented} multi-parameter persistence modules} (Definition~\ref{def:candidate}), 
    taking the form of {\em \mmaout{} decompositions} 
    $\tilde {\Mbb}_\delta=\bigoplus_{i\in\tilde{\mathcal I}}  \rebuttal{\field^{\tilde\interval_i}}$. These \mmaout{} decompositions are parameterized by a precision parameter $\delta > 0$, 
	and each $\rebuttal{\field^{\tilde I_i}}$ in these \mmaout{} decompositions %
    is an interval summand in $\R^n$.

    \item {\bf Then, we introduce our method \MMA{} (Multi-parameter persistence Module Approximation, Algorithm~\ref{algo:approx})
    for computing instances of such \mmaout{} decompositions}. \rebuttal{Our method is crucially  based on so-called {\em matching functions}}, and, \rebuttal{using 
    any matching function whose complexity is linear w.r.t. $N$}, 
    has
    running time 
    $$O\left(N^3 + \frac{1}{\delta^{n-1}}(N + n\cdot 2^{n-1})\right),$$ 
    where $N$ is the number of simplices and $n$ is the number of filtrations.
    See Figure~\ref{fig:mma_intro}. Note that \MMA{} does not require the input module $\Mbb$ to be interval decomposable in order to run. 

    \item %
    \rebuttal{{\bf We show that \MMA{} is a good approximation when computed with so-called {\em compatible} matching functions}}, i.e.,
    	that the \mmaout{} decompositions produced by \MMA{} are, in this case, {\em \mmaoutcompat{}
        decompositions}:
    	they preserve the (single-parameter) persistence barcodes associated to diagonal slices of the multi-parameter filtration (Proposition~\ref{prop:tildeIcandidate}):
	\begin{equation*}
		\textnormal{for all diagonal line } l\subseteq \mathbb R^n, \quad
		\distb \left(\barcode{\restr{\Mbb}{l}},
		\barcode{\restr{\tilde\Mbb^{\MMA{}}_\delta}{l}} \right)\leq
		2\delta.
	\end{equation*}
    	  We also show that, upon carefully choosing matching functions, the \mmaoutcompat{} decompositions produced by \MMA{} are also stable w.r.t. the input data (Proposition~\ref{prop:mma_stab}):
    	$$\disti(\tilde \Mbb^{\MMA{}}_\delta(f),\tilde \Mbb^{\MMA{}}_\delta(g)) \leq \distb(\tilde \Mbb^{\MMA{}}_\delta(f),\tilde \Mbb^{\MMA{}}_\delta(g)) \leq \norm{f-g}_\infty +\delta,$$
    	where $\tilde \Mbb^{\MMA{}}_\delta(f)$ stands for the candidate decomposition that is induced by the sublevel sets of $f$ and computed with \MMA{} (and similarly for $g$).

    \item %
    \rebuttal{{\bf Then, we restrict the focus to interval decomposable modules.}} 
    In that case, compatible matching functions always exist, and, under generic assumptions and small enough $\delta$, {\bf we prove that the interleaving and bottleneck distances between the \mmaoutcompat{} decompositions 
    produced by \MMA{} %
    and the underlying persistence module are upper bounded} (Proposition~\ref{prop:approx}):
    $$\disti(\Mbb,\tilde \Mbb^{\MMA{}}_\delta) \leq \distb(\Mbb,\tilde \Mbb^{\MMA{}}_\delta) \leq \delta.$$
    Moreover, we also show that, %
    when $\delta \leq \delta_{\rm exact}$, where $\delta_{\rm exact}$ is  a constant that depends only on the multi-parameter filtration values,
    the \mmaoutcompat{} decompositions %
    produced by \MMA{} %
    recover the underlying persistence module exactly (\cref{prop:exact_recovery}):
    $$\disti(\Mbb,\tilde \Mbb^{\MMA{}}_\delta) = \distb(\Mbb,\tilde \Mbb^{\MMA{}}_\delta) = 0.$$
    
    \item %
    {\bf We perform numerical experiments} that showcase the performance of \MMA{} and exhibit the trade-off between computation time and approximation error in Section~\ref{sec:expe}.
\end{enumerate}

\begin{figure}[H]
    \centering
    \includegraphics[width=.8\textwidth]{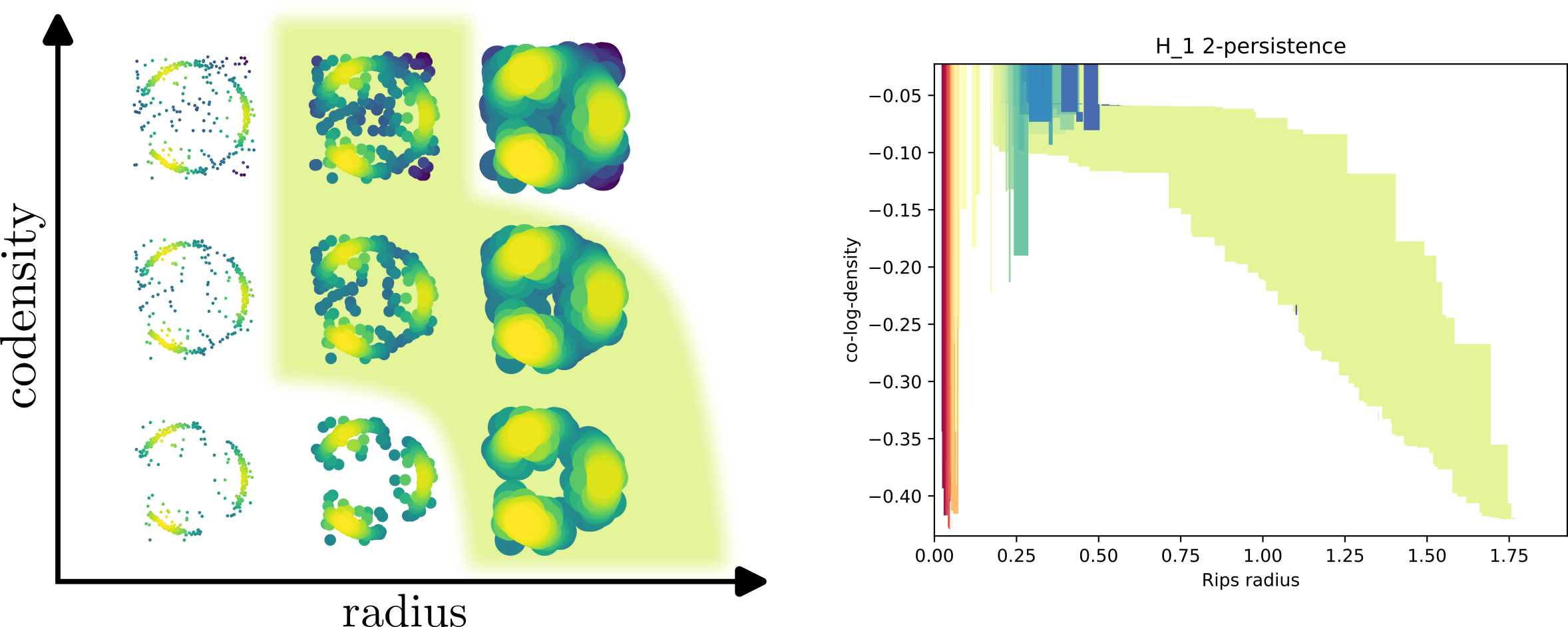}
    \caption{Example of \mmaout{} decomposition
    computed by \MMA{} on a point cloud filtered by both growing balls around
    the points (also called {\em \v Cech filtration}) and using the sublevel
    sets of codensity (or, equivalently, the superlevel sets of density), in
    homology dimension 1. One can see that there is a large lightgreen summand
    in the \mmaout{} decomposition on the right that corresponds to the cycle
    formed by the points amid outliers, which is also highlighted in lightgreen
    in the multi-parameter filtration on the left.}
    \label{fig:mma_intro}
\end{figure}

\paragraph*{Related work.} 
Computing approximate decompositions of multi-parameter persistence modules
with simple summands or interval summands has been studied in a few other
works.
For instance, in~\cite{deyRectangularApproximationStability2021}, the authors
provide an algorithm that computes optimal approximations with rectangle
summands of interval decomposable, $2$-parameter persistence modules in order
to lower bound their interleaving distance.
In~\cite{asashibaApproximation2DPersistence2023}, the authors provide a method
that associates to every $2$-parameter persistence module (and that was
recently generalized to any $n$-parameter persistence
module~\cite{asashibaIntervalReplacementsPersistence2024}) a pair of interval
decomposable persistence modules obtained by computing the M\"obius inversion
of the so-called {\em compressed multiplicity}, in a spirit similar to the
computation of generalized persistence diagrams and signed barcodes. 
Focusing on homology dimension 0 and $n=2$ {with the Vietoris-Rips
bifiltration on augmented metric spaces}, the authors in~\cite{caiElderRuleStaircodesAugmentedMetric2021}
have proposed the {\em elder-rule staircode}, an interval decomposable
persistence module who is known to recover the rank invariant and even the true
interval decomposition of the module (if it exists) under some assumptions. 
Focusing rather on stability (as general decompositions of modules are known to
be highly instable),
in~\cite{bjerkevikStabilizingDecompositionMultiparameter2025} the author
provides new decompositions for persistence modules based on {\em pruning}
(although with arbitrarily complicated summands) that enjoy better robustness
guarantees.

More closely related to our approach,
both~\cite{carriereMultiparameterPersistenceImage2020}
and~\cite{russoldGraphcodeLearningMultiparameter2024,
kerberRepresentingTwoparameterPersistence2025} provide methods to compute
descriptors for $2$-parameter persistence modules using matching functions
between persistence barcodes computed from a sorted family of $1$-dimensional
slices, namely the {\em multi-parameter persistence image}, and the {\em
graphcode}, respectively.  
While the multi-parameter persistence image encodes a decomposition constructed
by matching consecutive barcodes with the {\em vineyards}
algorithm~\cite{cohen-steinerVinesVineyardsUpdating2006} and computing the
summand boundaries with the endpoints of matched bars (which does {\em not}
guarantee that the resulting summand is an interval, or even that the resulting
module is close to the input), the graphcode is an abstract graph whose nodes
represent the bars of consecutive barcodes, and whose edges are built based on
matching functions representing canonical inclusions of representative cycles,
after fixing some cycle bases for every slice. Note that these matchings are
{\em not} one-to-one: a representative cycle in one barcode might be included
in several others in the next barcode. 

Our method differs from the previous ones in three key aspects: first, note
that, except for~\cite{asashibaIntervalReplacementsPersistence2024,
bjerkevikStabilizingDecompositionMultiparameter2025}, all other approaches are
designed for $n=2$ filtrations, while our method \MMA{} can handle an arbitrary
number $n$ of filtrations (albeit without theoretical guarantees if $n > 2$)
and works in any homology dimension. Second, we designed \MMA{} so that it has
{\em parameterized} complexity: while the other approaches might be costly to
compute as they rely on, e.g., Möbius inversions on large non-grid posets, the
running time for \MMA{} is controlled by the user through the choice of
$\delta$ and of the matching function. Finally, and most importantly, our
intention with the interval decomposable module produced by \MMA{} was to
provide a descriptor that is:

\begin{enumerate}[label=(\textit{\roman*})]
    \item {}\label{enum:intro_interpretable} {\em interpretable} in the same
	way than  persistence barcodes: its summands should correspond to
	lifetimes (in $\R^n$) of some homologous representative cycles, and
    \item {}\label{enum:intro_stable} {\em as stable as possible}: as it is
	impossible to provide interval decompositions that are always consistent with
	the rank invariant of the module (see~\cite[Section
	10.2.3]{lesnickNotesMultiparameterPersistence2023}), and as we still
	want to encode more information than the pointwise dimension of the
	module (i.e., the Hilbert function, which already requires a cubic
	complexity to compute in the two parameter case
	\cite{clauseMetaDiagrams2ParameterPersistence2023}), we seek for
	decompositions that
	preserve the persistence barcodes of diagonal slices of
	the module.
\end{enumerate}

We prove in this article that our method \MMA{} produces interval
decompositions that satisfy both
\cref{enum:intro_interpretable,enum:intro_stable}. On the other hand, other
approaches are either less
interpretable~\cite{deyRectangularApproximationStability2021,
asashibaApproximation2DPersistence2023,
asashibaIntervalReplacementsPersistence2024,
bjerkevikStabilizingDecompositionMultiparameter2025,
russoldGraphcodeLearningMultiparameter2024,
kerberRepresentingTwoparameterPersistence2025}, thus not satisfying
\ref{enum:intro_interpretable}, or not stable
enough~\cite{carriereMultiparameterPersistenceImage2020,
caiElderRuleStaircodesAugmentedMetric2021}, thus not satisfying
\ref{enum:intro_stable}.

More precisely, concerning \ref{enum:intro_stable}, our method can be seen as a
generalization of the decompositions provided
in~\cite{caiElderRuleStaircodesAugmentedMetric2021}, as we are {\em not}
restricted to homology dimension $0$, and as a continuation of the
decompositions provided in~\cite{carriereMultiparameterPersistenceImage2020},
as $(a)$ we generalize it to {\em compatible matching functions} instead of
relying solely on the vineyards algorithm, $(b)$ we guarantee that the produced
decompositions only contain {\em interval summands} instead of producing
summands supported on arbitrarily complicated shapes, and $(c)$ we prove that
the produced decomposition {\em preserve the diagonal barcodes}, and recover
the true decompositions exactly when the input module is itself interval
decomposable. 

Finally, one might argue that the interpretability property that we ask in
\ref{enum:intro_interpretable} also induces some degree of arbitrariness, as
there are, in general, several ways to choose representative cycles and their
lifetimes (precisely because there are no good barcodes in multi-paramerer
persistent homology, again see~\cite[Section
10.2.3]{lesnickNotesMultiparameterPersistence2023}). This is perfectly true,
and illustrates the trade-off between interpretability and computability that
is often encountered in this field; note for instance that
graphcodes~\cite{russoldGraphcodeLearningMultiparameter2024,
kerberRepresentingTwoparameterPersistence2025} also rely on arbitrary choices
of cycles bases. However, notice that $(a)$ using such interval decompositions
has nonetheless already proved to be useful in computational
topology~\cite{blaserCoreBifiltration2025} and data
science~\cite{loiseauxFrameworkFastStable2023}, and $(b)$ 
the {\em collection} of interval decompositions produced by \MMA{}
obtained by ranging over all compatible matching functions is itself a
topological {\em invariant} of the module. %
 \\

\paragraph*{Outline.} \rebuttal{ \cref{sec:background} provides a concise review of multi-parameter persistence modules and related notions. In~\cref{sec:candidates}, we present our new descriptors for multi-parameter persistence modules, as well as the \MMA{} algorithm for computing them. In~\cref{sec:approx}, we present some approximation properties satisfied by \MMA{}. Then, we provide stronger robustness guarantees for interval decomposable modules in~\cref{sec:exact_algo}, and we discuss the design of matching functions in~\cref{sec:matching_mma}.}
Finally, we illustrate the performances of \MMA{} in Section~\ref{sec:expe}.

\section{Background}
\label{sec:background}

In this section, we recall the basics of %
multi-parameter persistent homology and persistence modules.
This section only contains the necessary background and
notations, and can be skipped if the reader is already familiar with persistence theory. A more complete treatment of persistence modules
can be found in~\cite{Chazal2016, deyComputationalTopologyData2022, oudotPersistenceTheoryQuiver2015, botnanIntroductionMultiparameterPersistence2023}.

\paragraph*{Notations.}
We first introduce a few notations: we let $(e_1,\dots,e_n)$ be the canonical basis of $\R^n$,
and,
given a set $A\subseteq \R^n$,
we let $\conv{A}$ denote the
convex hull of $A$. Moreover, given a hyperplane $H\subseteq \R^n$ and its two associated
vectors $a_H,b_H\in\R^n$ which satisfy $H=b_H+\{x\in\R^n: \langle x,a_H \rangle=0\}$,
we call $a_H$ the {\em codirection} of $H$.
When $a_H$ is a vector in the canonical basis of $\R^n$, i.e., there exists $i\in \llbracket 1,n\rrbracket$ such that
$a_H=e_i$, we slightly abuse notation and also call $i$ the codirection of $H$.
\rebuttal{Finally, given a point $x\in \R^n$, we let $\left< x \right<$ (resp.
	$\left> x \right>$) denote the {\em upset} (resp. {\em downset}) of $x$:
$\left< x \right< := \left\{ y \in\R^n: y\geqn x \right\}$ (resp. $\left> x \right> := \left\{ y \in\R^n: y \leqn x \right\}$). }

\paragraph*{Multi-parameter persistence modules.}
In their most general form, multi-parameter persistence modules \cite{carlssonTheoryMultidimensionalPersistence2009} are nothing but
$\field$-vector spaces (where $\field$ denotes a field) indexed by $\R^n$ and
connected by linear maps.\\

\begin{definition}[Multi-parameter persistence module]\label{def:multipers_module}
	An ($n$-)multi-parameter persistence module $\Mbb$ is a family of vector spaces indexed over $\R^n$: $\Mbb=\{M_x\}_{x\in\R^n}$,
	equipped with linear transformations $\{\varphi_x^y: M_x\rightarrow M_y\}_{
		x,y\in\R^n,x\leqn y
	}$, that are called the {\em transition maps} of $\Mbb$,
	and
	that satisfy
	$\varphi_x^z=\varphi_y^z\circ\varphi_x^y$ \rebuttal{and $\varphi^x_x = \mathrm{id}$} for any $x\leqn y\leqn z$. \rebuttal{We sometimes let $\Mbb(x\leqn y):=\varphi_x^y$ denote these transition maps.}

	A {\em morphism} between two multi-parameter persistence modules $\Mbb,\Mbb'$
	with transition maps $\varphi^\cdot_\cdot$ and $\psi^\cdot_\cdot$ respectively, is a collection of linear maps $f=\{f_x:M_x\rightarrow M'_x\}_{x\in\R^n}$,
	that commutes with transitions maps, i.e., one has
	$f_y\circ \varphi^y_x=\psi^y_x\circ f_x$, for all $x\leqn y$. \\
\end{definition}

\rebuttal{
	Multi-parameter persistence modules are often assumed to satisfy finiteness
	assumptions, such as being
	\emph{pointwise finite dimensional} (\cref{def:pfd}) or
	\emph{finitely presentable} (\cref{def:fp}).
	See \cite{botnanIntroductionMultiparameterPersistence2023} for more details. \\

	\begin{definition}[Pointwise finite dimensional module]\label{def:pfd}
		Let $\Mbb$ be an $n$-parameter persistence module.
		We say that
		$\Mbb$ is \emph{pointwise finite dimensional (p.f.d.)}, if for any $x\in \R^n$, one has $\dim \Mbb _x <\infty$. \\
	\end{definition}
}

In this article, all multi-parameter persistence modules come from applying the {\em homology functor} $H_*$ on
a {\em multi-parameter filtration} of a
\rebuttal{topological space}
$\simpcomp$, that is, on
a family $\{\simpcomp_x\}_{x\in\R^n}$ of subsets of $\simpcomp$ indexed over $\R^n$ such that $x\leqn y \Rightarrow \simpcomp_x\subseteq \simpcomp_y$.
In other words, we study modules of the form $\Mbb:=\{H_*(\simpcomp_x)\}_{x\in\R^n}$,
where the linear maps $H_*(\simpcomp_x)\to H_*(\simpcomp_y)$ are induced by the canonical inclusions $\simpcomp_x\subseteq \simpcomp_y$ (when $x\leqn y$).
There are many interesting multi-parameter filtrations in data science;  one of
the most common one (with $n=2$) comes from filtering by feature scale and
density.  This allows to detect the topological structures (encoded in the
homology groups) of point clouds in the face of noise and outliers~\cite{blumbergStability2parameterPersistent2021, carriereMultiparameterPersistenceImage2020}. \\

\begin{definition}
	The {\em direct sum} of two multi-parameter persistence modules $\Mbb$ and
	$\Mbb'$, written as $\Mbb\oplus\Mbb'$, is the module $\Mbb''$ with vector spaces $ \{M''_x\}_{x\in\R^n}$ and transition maps $(\varphi'')_\cdot^\cdot$, defined as
	$M''_x = M_x\oplus M'_x$ for all $x\in\R^n$, and $(\varphi'')_\cdot^\cdot = \varphi_\cdot^\cdot \oplus (\varphi')_\cdot^\cdot$,
	where $ \{M_x\}_{x\in\R^n}$ (resp. $ \{M'_x\}_{x\in\R^n}$) and $\varphi_\cdot^\cdot$ (resp. $(\varphi')_\cdot^\cdot$) are the vector spaces and transition maps of $\Mbb$ (resp. $\Mbb'$) respectively.

	A multi-parameter persistence module $\Mbb$ such that there are no non-trivial modules $A$ and $B$ such that $\Mbb\simeq A \oplus B$ is called {\em indecomposable}. \\
\end{definition}

Note that while multi-parameter persistence modules can always be decomposed
into indecomposable summands \rebuttal{if they are p.f.d.}
(see~\cite{deyGeneralizedPersistenceAlgorithm2022,
deyDecomposingMultiparameterPersistence2025} for corresponding algorithms),
these summands can be arbitrarily complicated, and the resulting decomposition cannot really be used as an intuitive and simple invariant of the module.

\paragraph*{Distances between modules.}
Multi-parameter persistence modules can be compared with the standard {\em interleaving distance}~\cite{lesnickTheoryInterleavingDistance2015}. \\

\begin{definition}[Interleaving distance]\label{def:interleaving_distance}
	Given $ \varepsilon>0$, two multi-parameter persistence
	modules $\Mbb$ and $\Mbb'$ are {\em $\boldsymbol\varepsilon$-interleaved} if there exist
	two morphisms $f \colon \Mbb \to \Mbb'( \boldsymbol\varepsilon)$ and $g\colon \Mbb' \to
	\Mbb(\boldsymbol \varepsilon)$ such that
	$
	g_{\cdot+\boldsymbol \varepsilon} \circ f_\cdot = \varphi_{\cdot}^{\cdot + 2
	\boldsymbol \varepsilon} \text{ and } f_{\cdot+\boldsymbol\varepsilon}\circ g_\cdot =
	\psi_{\cdot}^{\cdot + 2 \boldsymbol\varepsilon},
	$
	where $\Mbb(\boldsymbol\varepsilon)$ is the {\em shifted module} $\{M_{x+\boldsymbol\varepsilon}\}_{x\in\R^n}$, $\boldsymbol\varepsilon = ( \varepsilon, \dots, \varepsilon)\in\R^n$,
	and $\varphi$ and $ \psi$ are the transition maps of $\Mbb$ and $\Mbb'$ respectively.

	The {\em interleaving distance} between two multi-parameter persistence modules $\Mbb$ and $\Mbb'$ is then defined as
	$
	\disti (\Mbb,\Mbb') := \inf \left\{  \varepsilon \ge 0 : \Mbb \text{ and } \Mbb'
	\text{ are } \boldsymbol\varepsilon\text{-interleaved}\right\}
	$. \\
\end{definition}

The main property of this distance is that it is {\em stable} for multi-parameter filtrations that are
obtained from the sublevel sets of functions. More precisely, given two continuous functions $f,g:\simpcomp\to\R^n$ defined on a \rebuttal{topological space}
$\simpcomp$
let $\Mbb(f),\Mbb(g)$ denote the multi-parameter persistence modules obtained from the corresponding multi-parameter filtrations
$\{\simpcomp^f_x:=\{s \in \simpcomp : f(s)\leqn x\}\}_{x\in\R^n}$ and
$\{\simpcomp_x^g:=\{s \in \simpcomp : g(s)\leqn x\}\}_{x\in\R^n}$.
Then, one has~\cite[Theorem 5.3]{lesnickTheoryInterleavingDistance2015}:
\begin{equation}\label{eq:stability}
	\disti(\Mbb(f),\Mbb(g))\leq \norm{f-g}_\infty.
\end{equation}

Another usual distance is the {\em bottleneck distance}~\cite[Section 2.3]{botnanAlgebraicStabilityZigzag2018}.
Intuitively, it relies on decompositions of the modules into direct sums of indecomposable summands,
and is defined as the largest interleaving
distance between summands that are matched under some matching. \\

\begin{definition}[Bottleneck distance] \label{def:bottleneck_distance}
	Given two multisets $A$ and $B$, $ \mu \colon A\not\to B $ is called a
	\rebuttal{{\em partial bijection}} if there exist $A' \subseteq A$ and $B'\subseteq B$ such that $\mu\colon A'\to B'$ is a bijection. The
	subset $A':=\mathrm{ coim}( \mu)$ (resp. $B':= \mathrm{ im}(\mu)$) is called the {\em coimage} (resp. {\em image}) of $\mu$.

	Let $\Mbb \cong \bigoplus_{i\in \mathcal{I}}\summand_i$ and $\Mbb' \cong \bigoplus_{j\in \mathcal{J}}\summand'_j$ be two multi-parameter persistence modules. %
	Given $\varepsilon\geq 0$, the modules $\Mbb$ and $\Mbb'$ are
	$\boldsymbol\varepsilon$-\emph{matched} if there exists a \rebuttal{partial bijection} $\mu \colon \mathcal{I}\not\to \mathcal{J}$ such that
	$M_i$ and $M'_{\mu(i)}$ are $ \boldsymbol\varepsilon$-interleaved for all
	$i\in\mathrm{ coim}(\mu)$, and $M_i$ (resp. $M'_j$) is
	$\boldsymbol\varepsilon$-interleaved with the null module {\bf 0} for all $i\in\mathcal{I}\backslash \mathrm{ coim}(\mu)$ (resp. $j\in \mathcal{J}\backslash \mathrm{ im}(\mu)$).

	The {\em bottleneck distance} %
	between two multi-parameter persistence modules $\Mbb$ and $\Mbb'$ is then defined as
	$
	\distb(\Mbb,\Mbb') := \inf \left\{ \varepsilon\ge 0 \,:\, \Mbb \text{ and } \Mbb' \text{ are  }\boldsymbol \varepsilon \text{-matched}\right\}. \\
	$
\end{definition}

Since a matching between the decompositions of two multi-parameter persistence modules induces an interleaving between
the modules themselves, it follows that $\disti\leq\distb$.
Note also that $\distb$ can actually be arbitrarily larger than $\disti$, as showcased in~\cite[Section 9]{botnanAlgebraicStabilityZigzag2018}.

\paragraph*{Interval modules.} %
\rebuttal{Now, we define a particular subfamily of multi-parameter persistence modules, the so-called {\em interval modules}.} Intuitively, they are modules that are trivial,
except on a subset of $\R^n$ called an {\em interval}. \\

\begin{definition}[Interval] \label{def:poset_interval}
	A subset $I$ of $\R^n$ is called an \emph{interval} %
	if it satisfies: %
	\begin{itemize}
		\item {\em (convexity)} if $p,q\in I$ and $p\leqn r \leqn q$ then $r\in I$, and
		\item {\em (connectivity)} if $p,q \in I$, then there exists a finite sequence $r_1,r_2,\dots,r_m \in I,$ for some $ m\in \N$,
			such that $p \sim r_1 \sim r_2 \sim \dots \sim r_m \sim q$, where $\sim$ can be either $\leqn$ or $\geqn$. \\
	\end{itemize}
\end{definition}

\begin{definition}[Interval module, indicator module]\label{def:interval_module_multi}
	An {\em interval module} $\Ibb$ is a multi-parameter persistence module such that:
	\begin{enumerate}
		\item $\Ibb$ is a thin module, i.e., $\forall  x \in \R^n, \, \dimension{\Ibb_x} \leq 1$,
		\item whose support $\supp{\Ibb} := \{x \in \R^n : \dimension{\Ibb_x} = 1\}$ is an interval of $\Rbb^n$,
		\item and whose transition maps are identity maps, i.e., $ \forall x\le_n y \in \supp{\Ibb}, \, \Ibb(x\le y) = \mathrm{id}$.
	\end{enumerate}

	Moreover, given an interval $I$, we let $\field^I$ denote the corresponding interval module with support $I$.
	Finally, a module $\Ibb$ that is a direct sum of interval modules $\Ibb =
	k^{I_1}\oplus\dots\oplus k^{I_m}$, $m\in\N^*$, whose supports have empty
	pairwise (closed) intersections, i.e., such that $I_p\cap I_q = \varnothing$ for all $1\leq p,q \leq m$, $p\neq q$, is called an \emph{indicator module}, and denoted by $k^U$,
	where $U:=\cup_{i=1}^m I_i$ is the union of their supports.\\
\end{definition}

Finally, we define a specific type of interval modules, those whose support is
equal to a union of rectangles. We call these modules {\em discretely
presented}---the \mmaout{} decompositions computed by our algorithm \MMA{} (Section~\ref{sec:candidates}) are actually made up of such modules. \\

\begin{definition}[Discretely presented interval module]\label{def:discretely_presented}
	An interval module $\Ibb = k^I$ is \emph{discretely presented} if its
	support $I$ is a locally finite union of rectangles in $\overline\R^n$,
	and whose boundary is an $(n-1)$-submanifold of $\R^n$.
	More precisely,
	there exist two locally finite families of points, the {\em birth} and
	{\em death critical points} of $I$, denoted by $\birthcp{I}$ and
	$\deathcp{I}$ respectively, such that:
	\begin{equation}\label{eq:discrete_inter}
		I:=\bigcup_{c\in \birthcp{I}}\bigcup_{c'\in \deathcp{I}} \rectangle{c}{c'},
	\end{equation}
	where $R_{c,c'}:=\{x\in(\R\cup\{\pm\infty\})^n : c\leqn x {\rebuttal <_{n}} c'\}$ is the rectangle %
	with %
	{\em corners} $c$ and $c'$. \\ %
\end{definition}

\rebuttal{
	Discretely presented interval modules can be obtained through stronger assumptions, such as
	being a \emph{finitely presentable} interval module. See \cite{botnanIntroductionMultiparameterPersistence2023} for more details. \\

	\begin{definition}[Finitely presentable persistence module.]\label{def:fp}
		Let $\Mbb$ be an $n$-parameter persistence module. One says that
		$\Mbb$
		is \emph{finitely presentable (f.p.)}
		if it is isomorphic to the cokernel of a persistence morphism $f\colon R\to G$, where $G,R$ are interval decomposable modules of the form:
		\begin{equation}
			G = \bigoplus_{1\le i \le n_G} k^{ \left< x_i \right<}
			\quad \textnormal{ and }\quad
			R = \bigoplus_{1\le i \le n_R} k^{ \left< y_i \right<},
		\end{equation}
		for some finite sets $\beta_0 := (x_i)_{1\le i \le n_G}$ and $\beta_1 :=(y_i)_{1\le i\le n_R}$, with $n_G,n_r\in\N$.
		When $n_G$ and $n_R$ are minimal, the sets $\beta_0(\Mbb)$ and
		$\beta_1(\Mbb)$ are unique
		and called the degree 0 and degree 1 \emph{graded Betti numbers} of $\Mbb$, respectively.  \\
	\end{definition}

	The graded Betti numbers characterize the position of topological events in a multi-parameter filtration. %
	In particular, f.p. modules can be restricted to \emph{large enough} compact
	sets without loosing information, if such sets contain the graded Betti numbers, see \cref{rmk:compacity_assumption}. Next, we define {\em restrictions of modules}. \\
}

\rebuttal{
	\begin{definition}[Restrictions and slices]\label{remark:restrictions_and_slices}
		If  $Q\subseteq \R^n$, then an $n$-parameter persistence module $\Mbb=\{M_x\}_{x\in\R^n}$ induces a persistence module $\restr \Mbb Q$ indexed over $Q$, defined by:
		\begin{equation*}
			\forall x\leqn y \in Q, \quad \left(\restr \Mbb {Q}\right)_x := M_x \textnormal{
			and } \restr \Mbb Q (x \leqn y) := \Mbb(x\leqn y).
		\end{equation*}
	\end{definition}

	In particular, if $\Mbb$ is an $n$-parameter persistence module, then given a
	line $l\subseteq \R^n$ with positive slope, the persistence module $\restr \Mbb l$ can be seen as a usual 1-parameter persistence module \emph{up to some parametrization of $l$}.
	\\

	\begin{example}[Restriction to a line]\label{ex:restr_line_mp}
		Let $\Mbb=\{M_x\}_{x\in\R^n}$ be an $n$-parameter persistence module, and $l\subseteq \R^n$ be a line
		parametrized by $t\in \R\mapsto at+b$ with $a\in (\R_+)^n
		\backslash \left\{ 0 \right\} $.
		The restricted module $\restr \Mbb l$ given by:
		\begin{equation}
			\left( \restr \Mbb l \right) _t := M_{at+b} \quad \textnormal{and }\quad
			\left( \restr \Mbb l \right)(s\le t) := \Mbb(as+b \le at+b),
		\end{equation}
		is a $1$-parameter module, called the {\em restriction, or slice, of $\Mbb$ along $l$}.  \\
	\end{example}

	The opposite operation is called a (left) Kan extension, that we only define in
	our setup. \\

	\begin{definition}[Kan Extension]\label{def:kan_extension}
		Let $A = A_1\times \cdots \times A_n \subseteq \R^n$ be a poset of $\R^n$ obtained as a
		product of subsets of $\R$, and $\Mbb=\{M_x\}_{x\in A}$ be a multi-parameter persistence module indexed over $A$.
		The \emph{(left) Kan extensions} of $\Mbb$  is the
		$n$-parameter persistence module (indexed over $\R^n$)
		defined as follows:
		for any $x\leqn y \in \R^n$,
		\begin{equation}
			\left( \mathrm{Lan}_n \Mbb \right)_x = M_{\lfloor x\rfloor_A}
			\quad
			\textnormal{and}
			\quad
			\mathrm{Lan}_n \Mbb (x\leqn y) = \Mbb(\lfloor x\rfloor_A\leqn \lfloor y\rfloor_A),
		\end{equation}
		where $\lfloor x \rfloor_A := \max \left\{ g\in A \,:\, g \le x \right\} $, with the conventions $\max (\varnothing) = -\boldsymbol \infty$, and $\Mbb_{-\boldsymbol\infty} = \boldsymbol 0$.
		\\
	\end{definition}

	\begin{remark}\label{rmk:compacity_assumption}
		If the graded Betti numbers of $\Mbb$ are included in a rectangle subset $K = [a_1,b_1]\times \cdots \times [a_n,b_n]$ of $\R^n$,
		then the restriction $\restr \Mbb K$ contains as much information as the original module $\Mbb$; more formally, we have $\mathrm {Lan}_n \restr \Mbb K \simeq \Mbb$.
		Notice that, in practice, finding such a compact set $K$ for
		a given persistence module $\Mbb$
		does not require computing a minimal presentation of $\Mbb$
		(\cref{def:fp}). Indeed,
		if $\Mbb$ is obtained from applying the homology functor on a finite multi-parameter filtration $F$, i.e., $\Mbb \simeq H_*(F)$, then
		any rectangle $R\subseteq \R^n$ containing
		the filtration values of the simplices of $F$ will also contain
		the graded Betti numbers.
		\\
	\end{remark}

}

Note that when only one filtration is given, single-parameter \rebuttal{p.f.d.} persistence
modules (\cref{def:pfd}) always decompose into interval modules:
$\Mbb\cong\bigoplus_{i \in
\mathcal I} \induced{[b_i,d_i)}$~\cite[Theorem 2.8]{Chazal2016}. In that case,
they are frequently represented as the collection of the supports (in $\R$) of
their summands, also called {\em persistence barcode}
$\barcode{\Mbb}:=\{[b_i,d_i)\}_{i \in \mathcal I}$.

\paragraph*{Boundaries and facets of intervals}\label{app:defin}
Now, we recall the definition of {\em upper- and lower-boundaries} of interval modules, as well as their so-called {\em facets},
which are convenient characterizations of the interval supports. \\

\begin{definition}[Upper- and lower-boundaries]\label{def:upper_lower_boundaries}
	Given an interval $I\subset \R^n$, its \emph{upper-boundary} $\upperb{I}$ and \emph{lower-boundary} $\lowerb{I}$ are defined as:
	\begin{equation*}
		\lowerb{I}:= \left\{ x \in \bar I: \forall y\in\R^n,\, y \sln x \Rightarrow y \not\in I  \right\},\ \
		\upperb{I}:= \left\{ x \in \bar I: \forall y\in\R^n,\, y \sgn x \Rightarrow y \not\in I  \right\}.
	\end{equation*}
	Moreover, the boundary of ${I}$ can be decomposed with $\partial {I}=\lowerb{I} \cup \upperb{I}$.
	See Figure~\ref{fig:liui} for an illustration. \\
\end{definition}

\svg{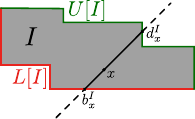}{0.3}{\label{fig:liui} Lower- and upper-boundaries of
an interval in $\R^2$ (Definition~\ref{def:upper_lower_boundaries}); and birthpoints and deathpoints $b_x^I$ and $d_x^I$  (Definition~\ref{def:birthpoints_deathpoints}) of a point $x \in \R^2$.}

When interval modules are discretely presented (see Definition~\ref{def:discretely_presented}), their lower- and upper-boundaries are made of flat parts, which are the faces
of the corresponding rectangles forming the interval. Hence, we call {\em facets}
the subsets of the lower- and upper-boundaries that are included in some hyperplanes of $\R^n$. \\ %

\begin{definition}[Facet]\label{def:facet}
	A \emph{lower} (resp. \emph{upper}) \emph{facet} of an interval $I\subset \R^n$ is an $(n-1)$-submanifold of $\partial \supp{I}$ written as
	$
	\left\{ x \in \R^n \,:\, x_i = c\right\} \cap \lowerb{I}$ (resp. $\left\{ x \in \R^n \,:\, x_i = c\right\} \cap \upperb{I}$)
	for some $c \in \R$ and some dimension $i \in \llbracket 1, n \rrbracket$ that is called the facet \emph{codirection}.
	In particular, the upper- and lower-boundaries of a discretely presented interval module is a locally finite union of facets.
\end{definition}

\paragraph*{Fibered barcode.} The {\em fibered barcode}~\cite{cerriBettiNumbersMultidimensional2013,lesnickInteractiveVisualization2D2015} is a centerpiece of our \MMA{} algorithm,
and is defined, given a p.f.d. multi-parameter persistence module $\Mbb$,
as a map that takes as input a line (or segment) $l$ in $\R^n$,
and outputs the persistence barcode
associated to the single-parameter persistence module obtained by restricting $\Mbb$ along $l$.
We formalize these concepts in the next definition.
\\

\begin{definition}[\rebuttal{Diagonal} fibered barcode]\label{def:barcode}
	Let $\Mbb$ %
	be a \rebuttal{p.f.d.} multi-parameter persistence module.
	Given a \rebuttal{set  $L$ of diagonal lines in $\mathbb{R}^n$} (i.e., lines
	with direction vector $[1,\dots,1]\in\R^n$), we let the {\em $L$-fibered barcode} (or {\em fibered barcode} for short when $L$ is clear) be
	the family of barcodes associated to restrictions of the module along lines in
	$L$, i.e.,
	$\fibered{\Mbb}_L = \rebuttal{\left( \barcode {\restr \Mbb l}  \right)_{l\in L}}$. \\
\end{definition}

\begin{remark}\label{rmk:barcode_indicator}
	Recall from~\cite[Theorem 2.8]{Chazal2016} %
	that
	$\supp{\restr{\Mbb}{l}}$ %
	is a \rebuttal{multi}set of bars called persistence barcode:
	$\supp{\restr \Mbb l} = \barcode{\restr \Mbb l} = \{[b_i,d_i)\}_{i\in \mathcal I(l)}$, where $\mathcal I (l)$
	is an index set that depends on $\Mbb$ and $l$.
	Moreover, when $\Mbb
	= \bigoplus_{i\in \mathcal  I} \induced{\interval_i}$ is decomposable into
	interval modules, there are as many bars in the barcode as there are interval
	summands intersecting the line $l$: $|\barcode{\restr \Mbb l}| =
	|\{\interval_i : \interval_i\cap l\neq\varnothing\}|$. \\ %
\end{remark}

It is also useful to characterize the fibered barcode with {\em endpoints} of lines. \\

\begin{definition}[Birthpoint, Deathpoint]\label{def:birthpoints_deathpoints}
	Given a positive line $l$ (that is, a line whose direction vector $u$ is in $\R^n_+ \setminus \left\{ 0  \right\}$)
	and an interval $I\subseteq \R^n$, the \emph{birthpoint} (resp. \emph{deathpoint}) of $I$ along $l$ is:
	\begin{equation*}
		b^I_l := \inf \, l \cap I , \quad \text{resp. } d^I_l := \sup \, l \cap I.\footnote{
			Note that this is well-defined and finite as the restriction of the poset $(\R^n,\leqn)$ to a positive line is totally ordered.
		}
	\end{equation*}
	If the direction $u \in \R_+^n$ of the line is given by the context, and $x \in \R^n$, we will also let $b^I_x := b_{l_x}^I$ (resp. $d^I_x := d_{l_x}^I$)
	denote the \emph{birthpoint} (resp. \emph{deathpoint}) associated to $I$ and $x$, where $l_x$ is the line crossing $x$, with direction vector $u$.
	See \cref{fig:liui}.\\
\end{definition}

\begin{remark}[Slicing an interval decomposable module]\label{rem:module_slice}
	Using birthpoints and deathpoints, the $L$-fibered barcode of an interval
	decomposable multi-parameter persistence module
	$\Mbb = \bigoplus_{i\in \mathcal  I} \induced{\interval_i}$
	can be written as:
	\begin{equation}
		\fibered{\Mbb}_L = \left( \barcode{\restr \Mbb l}\right)_{l\in L}=\left(\{[b^{\interval_i}_l,d^{\interval_i}_l)\}_{i\in\mathcal I}\right)_{l\in L}.
	\end{equation}
\end{remark}

\begin{remark}
	Notice that in \cref{rem:module_slice}, the persistence barcodes $\barcode{\restr \Mbb l} =
	\{[b^{\interval_i}_l,d^{\interval_i}_l)\}_{i\in\mathcal I}$
	can be seen as multisets of segments $[b_i,d_i)$ in $\R^n\cup \left\{
	\boldsymbol\infty{} \right\} $.
	In particular, the diagonal line of a given segment $[b_i,d_i)$ can be
	recovered from, for instance, the birthpoint $b_i$, and hence, without
	loosing any information (except for lines with trivial barcodes), we will consider the following identification:
	\begin{equation}
		\fibered{\Mbb}_L = \bigcup_{l\in L} \barcode{\restr \Mbb l}
		= \left\{ \barcode{\restr \Mbb l} : l \in L \right\}.
	\end{equation}
\end{remark}

\paragraph*{Geometry and stability of diagonal barcodes.}

\rebuttal{We now present two simple yet fundamental results on diagonal barcodes. The first one characterizes rectangles formed by endpoints.} \\

\begin{lemma}\label{rem:flat_rect}
	Let $l_1,l_2$ be two diagonal lines and $\rebuttal{\field^{I}}$ be a \rebuttal{f.p.}
	interval module such that the barcodes
	$\barcode{\restr{\rebuttal{\field^{I}}}{l_1}}$ and
	$\barcode{\restr{\rebuttal{\field^{I}}}{l_2}}$ are not empty.
	Let $\barcode{\restr{\rebuttal{\field^{I}}}{l_1}}=[b_1,d_1)$ and
	$\barcode{\restr{\rebuttal{\field^{I}}}{l_2}}=[b_2,d_2)$.
	Then,
	the rectangles $R_{b_1,b_2}$, $R_{b_2,b_1}$, $R_{d_1,d_2}$ and
	$R_{d_2,d_1}$ are flat,
	that is, they either have null volume, or their corners are not comparable.
\end{lemma}

\begin{proof}
	This lemma is a simple consequence of the persistence module definition:
	if	$b_1$ and $b_2$ were comparable (as in Figure~\ref{fig:flat_rect}), then the rectangle
	$\rectangle{b_1}{b_2}$ %
	would not be trivial, and
	$b_2$ %
	would not be a birthpoint %
	since it would be possible to find a smaller birthpoint $\tilde b_2 \leqn b_2$ %
	w.r.t. the partial order of $\R^n$ along the diagonal
	line passing through $b_2$. %
	A similar argument holds for $d_1$ and $d_2$.
	See Figure~\ref{fig:flat_rect}.

	\svg{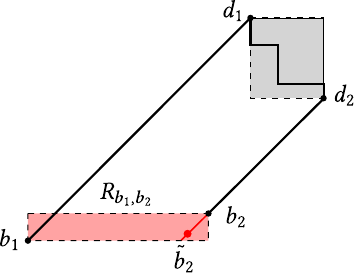}{0.4}{\label{fig:flat_rect} Two bars $[b_1,d_1)$ and $[b_2,d_2)$ of an interval module.
	}
\end{proof}

We now show that endpoints of bars in barcodes associated to lines that are close should also be close.
In other words, bars of the fibered barcode that are associated to lines that are close to each other
must have similar length,
as stated in the lemma below; see also~\cite[Lemma 2]{landiRankInvariantStability2018}. \\

\begin{lemma}\label{lemma:compatible_close}
	Let $\rebuttal{k^I}$ be a \rebuttal{f.p.} interval module, let $l_1,l_2\subseteq\R^n$ be two diagonal lines and
	let $\ora u\in \R^n$ be a positive or negative vector
	such that $l_2 = l_1 + \ora u$.
	Then, the following properties hold:
	\begin{enumerate}[label=\textnormal{(\textit{\roman*})}]
		\item If the barcode
			$\barcode{\restr{\rebuttal{\field^I}}{l_1}}=\{[b^I_{l_1}, d^I_{l_1})\}$
			is not empty and satisfies $\norm{d^I_{l_1} - b^I_{l_1}}_\infty >
			\norm{\ora u}_\infty$, then the barcode
			$\barcode{\restr{\rebuttal{\field^I}}{l_2}}$ is not empty as well, and
		\item	If the barcodes $\barcode{\restr{\rebuttal{\field^I}}{l_1}}$ and
			$\barcode{\restr{\rebuttal{\field^I}}{l_2}}$ are not empty,
			then one has
			\begin{equation*}
				\norm{ d^I_{l_1}-d^I_{l_2}}_{\infty} \le \norm{ \ora u}_\infty \text{ and } \norm{ b^I_{l_1}-b^I_{l_2}}_\infty \le \norm{ \ora u}_\infty,
			\end{equation*}
			where we used the conventions
			$(+\boldsymbol\infty) - (+\boldsymbol\infty) = (-\boldsymbol\infty) - (-\boldsymbol\infty)  = 0 $.
	\end{enumerate}
\end{lemma}

\begin{proof}

	\noindent {\bf Item $(i)$.} Since $|(d^I_{l_1})_i - (b^I_{l_1})_i| =
	\norm{d^I_{l_1} - b^I_{l_1}}_\infty > \norm{\ora u }_\infty$ for any index
	$i\in\llbracket 1,n \rrbracket$, it follows that $b^I_{l_1}\leqn
	b^I_{l_1}+\ora u \leqn d^I_{l_1}$.\footnote{We assume here that
	$\ora u$ is positive. It should be replaced by $d_{l_1}^I - \ora u$ if it is negative.} Thus $b^I_{l_1}+\ora u$ must belong to ${I}$ since $I$ is
	an interval. Hence, since $b^I_{l_1}+\ora u \in l_2$, one has
	$\barcode{\restr{\rebuttal{\induced I}}{l_2}} = {I}\cap l_2 \neq\varnothing$.

	\noindent {\bf Item $(ii)$.} If one of the endpoints is infinite, the result
	holds trivially as the other endpoint has to be infinite too, so we now assume that the endpoints of the bars are all finite.
	Without loss of generality, %
	assume that $l_2 = l_1 + \ora u$ where $\ora u$ is positive.
	Now, since both $d_{l_2}^I$ and $d_{l_1}^I + \ora u$ belong to $l_2$, they are comparable, so one has either $d_{l_2}^I \sgn d_{l_1}^I + \ora u$ or $d_{l_2}^I \leqn d_{l_1}^I + \ora u$.
	However, the first possibility would lead to  $d_{l_2}^I \sgn d_{l_1}^I + \ora u \sgn d_{l_1}^I$, hence $d_{l_1}^I$ and $d_{l_2}^I$ would be (strictly) comparable in $\R^n$,
	which contradicts Lemma~\ref{rem:flat_rect}.
	Thus, one must have $d_{l_2}^I \leqn d_{l_1}^I + \ora u$.
	Furthermore, and using the exact same arguments, $d^I_{l_2} - \ora u + \norm{ \ora u}_{\infty} \cdot \boldsymbol 1$ is on $l_1$, and one must have
	$d_{l_2}^I - \ora u +\norm{ \ora u}_{\infty} \cdot \boldsymbol 1 \geqn d_{l_1}^I$. %
	Finally, by combining the two previous inequalities, one has:
	\begin{equation*}
		d_{l_1}^I - \norm{ \ora u}_{\infty} \cdot \textbf{1}
		\leqn d_{l_1}^I + \ora u - \norm{ \ora u}_{\infty} \cdot \textbf{1}
		\leqn
		d_{l_2}^I
		\leqn
		d_{l_1}^I + \ora u
		\leqn
		d_{l_1}^I + \norm{ \ora u}_{\infty} \cdot \textbf{1},
	\end{equation*}
	which leads to the result for deathpoints. The proof extends straightforwardly to birthpoints. %
\end{proof}

\section{\rebuttal{Computing \mmaout{} decompositions with the \MMA{} algorithm}}
\label{sec:candidates}

In this section,  we present our 
family of \rebuttal{descriptors} for multi-parameter persistence modules,
defined as {\em \mmaout{} decompositions} into interval summands 
\rebuttal{and we identify a specific subfamily that we call {\em \mmaoutcompat{} decompositions} (\cref{def:candidate}), in Section~\ref{subsec:approx}.}
Then, we show how practical computations of instances of such \mmaout{} decompositions can be done with our \MMA{} algorithm in Sections~\ref{subsec:motivation} and~\ref{subsec:Algorithms}.

\subsection{Candidate and \mmaoutcompat{} decompositions} \label{subsec:approx}
Our candidate decompositions depend on
{\em $\delta$-grids of lines}, that we now define. \\

\begin{definition}[$\delta$-grid of lines]\label{def:grid}
	Let $K\subset\R^n$ be a compact set and $\delta > 0$. 
	The $\delta$-grid of lines associated to $K$, denoted as $L_\delta(K)$, is a family of diagonal lines evenly sampled in $K$:
	$$L_\delta(K):=\{l_{\delta\cdot u}: u \in \Z^n \text{ and } l_{\delta\cdot
	u}\cap K \neq \varnothing\},$$
	where $l_{\delta\cdot u}:=\delta \cdot
	u+\boldsymbol{e}_\Delta\rebuttal{\mathbb R}$ is the diagonal line with
	direction vector $\boldsymbol{e}_\Delta=[1,\dots,1]^T\in\R^n$ passing through
	$\delta\cdot u$. \\
\end{definition}

Several new definitions can be introduced from grids of lines, which will turn useful either in the definition of our \MMA{} algorithm, or in the corresponding theoretical proofs. \\

\begin{definition}[$\delta$-regularly distributed lines filling a compact set]\label{def:regularly_distributed}
	Let $L$ be a set of diagonal lines in $\R^n$ and $K\subseteq \R^n$ be a compact set.
	Then, we say that:
	\begin{enumerate}
		\item Two diagonal lines $l, l' \in L$ are $\delta$-\emph{consecutive} (or {\em consecutive} when $\delta$ is clear) if there exists ${\ora u} \in \left\{ 0,1\right\}^n \backslash \left\{ \mathbf{0},\mathbf{1}\right\}$
		such that $l' = l \pm \delta\cdot {\ora u}$. %
		\item Two diagonal lines $l, l' \in L$ are
		$\delta$-\emph{comparable} if there exists a positive or negative vector $\ora u\in\R^n$ 
		with $\norm{\ora u}_\infty \le \delta$ such
		that $l' = l + \ora u$.
		If $\ora u$ is positive (resp. negative), we write $l'\ge l$ (resp. $l'\le l$).
		\item $L$ is $\delta$-\emph{regularly distributed} if, for any pair of lines $(l,l')\in L$, there exists a sequence of $\delta$-consecutive lines $\{l_1, \dots, l_k\}$ in $L$ such that $l = l_1$ and $l'=l_k$.
		\item For a given line $l$ in a $\delta$-regularly distributed family of lines $L$, we call
		$L_l:=L\cap\{l+\delta\cdot \ora u: \ora u\in \{0,1\}^{n-1}\times \{0\}\}$
		the {\em $L$-surrounding set} of $l$. In particular, one has $|L_l| \le 2^{n-1}$.
		\item $L$ $\delta$-\emph{fills} $K$ (or {\em fills} $K$ when $\delta$ is clear)
		if any point of $K$ is at distance at most $ \delta/2$ from some line in $L$.
		In other words, $K$ is included in the offset $L^{\delta/2}$. \\ %
	\end{enumerate}
\end{definition}

Our candidate decompositions \rebuttal{of a given multi-parameter persistence module $\Mbb$} are,
roughly speaking, interval decomposable modules with fibered barcodes \rebuttal{containing the one of}
$\Mbb$ on a $\delta$-grid of lines. %
\rebuttal{
	Before going into the definition of our estimator \mmaout, we introduce a
	compactness assumption, that directly follows \cref{rmk:compacity_assumption}.\\

	\begin{definition}[Module compactness]
		We say that a rectangle ${K = [a_1, b_1] \times \cdots\times [a_n,b_n]\subseteq \R^n}$
		{\em \compacityassumption{}} 
		an $n$-parameter persistence module $\Mbb$
		if restricting $\Mbb$ to $K$ preserve information, or, more formally, if
		$\Lan \restr \Mbb K \simeq \Mbb$.
		\\
	\end{definition}
}
\begin{definition}[Candidate and \mmaoutcompat{} decompositions]\label{def:candidate}
Let $\Mbb$
	be a \rebuttal{f.p.} $n$-parameter persistence module. Let $K$ be a compact set that  
	\rebuttal{\compacityassumption{} $\Mbb$,}
and $L:=L_\delta(K^\delta)$ be the $\delta$-grid of lines of the offset $K^\delta=\{x\in\R^n : d_\infty(x,K) \leq \delta\}$, where $d_\infty$ stands for the $\norm{\cdot}_\infty$ distance.
A %
multi-parameter persistence module $\tilde\Mbb_\delta$
is called a $\delta$-\emph{\mmaout{} decomposition} of $\Mbb$ if:
\begin{enumerate}
	\item[$(i)$] $\tilde\Mbb_\delta$ is interval decomposable: $\tilde\Mbb_\delta=\bigoplus_{i \in {\tilde{\mathcal{I}}}} \rebuttal{k^{\tilde \interval_i}} $, and
	\item[$(ii)$] \rebuttal{$\barcode{\restr{\Mbb}{l}}\subseteq\barcode{\restr{\tilde \Mbb_\delta}{l}}$ for any $l\in L$, i.e., the $L$-fibered barcode of $\Mbb$, seen as a multiset of segments in $\R^n$, is included in the one of $\tilde\Mbb_\delta$.}
\end{enumerate}
\rebuttal{Clearly, a candidate decomposition can be a rough descriptor of
	$\Mbb$, as the bars in its fibered barcode can be arbitrarily large. Hence,
	we identify a more stable subfamily of \mmaout{} decompositions:}	
\begin{enumerate}
	\item[$(iii)$] If $\distb\left(\barcode{\restr{\Mbb}{l}},\barcode{\restr{\tilde \Mbb_\delta}{l}}\right)\leq 2\delta$ for {\em any} diagonal line $l$ (not only those that belong to $L$), then $\tilde \Mbb_\delta$ is called an {\em \mmaoutcompat{}} decomposition of $\Mbb$. \\
\end{enumerate} 
\end{definition}

\rebuttal{The reason we focus on preserving the diagonal fibered barcodes (instead of controlling, e.g., the rank invariant or the interleaving distance to $\Mbb$) is because of the impossibility for general multi-parameter persistence modules to build a decomposition into indicator modules that is consistent with the rank invariant (see~\cite[Section 10.2.3]{lesnickNotesMultiparameterPersistence2023}). Note however that this is still stronger than preserving the Hilbert function, i.e., the pointwise dimension of the module.} \\

\rebuttal{
	\begin{remark}[Non-diagonal lines]
		Extending the definition of our \mmaout{} decompositions to grids of non-diagonal lines (i.e., with direction vector different than $\boldsymbol{e}_\Delta$) is straightforward, and is completely equivalent to rescaling %
		the filtrations.
		Using such non-diagonal grids will however produce less stable descriptors, as the interleaving distance (\cref{def:interleaving_distance}) is based on the diagonal direction. \\
	\end{remark}
}

\begin{remark}\label{rmk:grid}
	One can check that the $\delta$-grid of lines of associated to $K^\delta$ used in
	Definition~\ref{def:candidate}
	is $\delta$-regularly distributed and $\delta$-fills $K$. \\
\end{remark}

\rebuttal{
Finally, we introduce the definition of {\em matching functions}. Such functions play a key role in our \MMA{} algorithm for computing \mmaout{} decompositions. \\

	\begin{definition}[Matching function]\label{def:barcode_matching}
		Let $\Mbb$ be a f.p. $n$-parameter persistence module, and $l, l'\subseteq \R^n$ be two positive lines. %
		A map $\matching$ between the persistence barcodes:
		\begin{equation*}
			\matching\colon
			\barcode{\restr{\Mbb}{l}}
			\rightarrow
			\barcode{\restr{\Mbb}{l'}}
			\cup \left\{ \varnothing\right\}
		\end{equation*}
		is called an ($\Mbb$-)\emph{matching function} between $l$ and $l'$ if the restriction of $\matching$ to $\matching^{-1}(\barcode{\restr{\Mbb}{l'}})$ is injective. %
		In other words, $\matching$ is a partial bijection (\cref{def:bottleneck_distance}) between the two barcodes, seen as multisets of intervals. \\
	\end{definition}

	\begin{definition}[Induced matching functions]
		\label{ex:exact_matching}
		If $\Mbb = \bigoplus_{i\in \mathcal I}k^{I_i}$ is a f.p. interval decomposable
		module, 
		then, for any positive line $l$,
		the bars of any barcode $\barcode{\restr \Mbb {l}} \cong \bigoplus_{ i \in \mathcal{I}} \restr
		{k^{{I_i}}} l$ can be indexed using $\mathcal{I}$ (by also counting empty bars).
		Thus, given any two positive lines $l_1,l_2$, one can match the bars $\restr {k^{I_i}}
		{l_1}$ and $\restr {k^{{I_i}}} {l_2}$ together so that matched bars
		correspond to the same underlying interval summand of $\Mbb$.
		In that case, the corresponding matching function $\sigma_{\Mbb}$ is referred to as \emph{induced from $\Mbb$}.
	\end{definition}
	
}

\subsection{Motivation for the \MMA{} algorithm} \label{subsec:motivation}

Our \MMA{} algorithm can be roughly described as a method that constructs interval summands based on 
families of bars (coming from the fibered barcode) that have been matched together using some 
matching function. 
The goal of this section is to frame the general question of practically computing \mmaout{} decompositions of a multi-parameter persistence module
from its fibered barcode and a matching function. There are many ways of doing so, but the most natural ones are not necessarily 
the easiest computable ones.
For the sake of simplicity, let us leave the problem of finding proper matching functions aside for now (which we will discuss in more details in Section~\ref{sec:matching_mma}), and assume that the underlying module is a single 
interval module $\Mbb=\Ibb$. %
Since interval modules are characterized by their supports, the goal is to recover $\supp{\Ibb}$.
Moreover, if $\Ibb$ is discretely presented, 
only the 
facets and critical points (i.e., points where several facets intersect) of $\supp{\Ibb}$ %
need to be captured or approximated. %
There are many different ways, for a given interval module $\Ibb$, to define candidate critical points, that we call {\em corners}, 
using the endpoints of 
its fibered barcode, e.g., 
by using the minimum and maximum of consecutive endpoint coordinates. 
Hence, it is natural
to find a \mmaout{} decomposition
(or candidate interval in this case, since there is just one interval summand) 
$\tilde \Ibb$ with
{\em model selection}, i.e., by minimizing
some penalty cost 
$ \mathrm{pen}\colon S \to \mathbb R_+$,
where $S$ is the set of discretely presented interval modules having the same 
fibered barcode than $\Ibb$,
or a subset thereof. See Figure~\ref{fig:candidate_argmin} for examples of sets $S$ and corresponding candidate intervals.
This penalty would forbid, e.g., overly complicated intervals that have lots of corners.
For instance, 
minimizing the penalty:
\begin{equation}\label{eq:penalty}
	\mathrm{pen}: \tilde \Ibb \mapsto  \# \text{corners of } \supp{\tilde \Ibb},
\end{equation} 
would provide a sparse approximation of $\Ibb$.
Actually,
when one assumes 
that the underlying interval module $\Ibb$ %
is discretely presented 
with facets that are large enough with respect to the family of lines $L$ of the fibered barcode, 
the target $\Ibb$ minimizes penalty~(\ref{eq:penalty}). %
Indeed, 
as all the facets of $\Ibb$ are detected by some endpoints of the fibered barcode by assumption, any candidate approximation $\tilde \Ibb$ of $\Ibb$ %
has at least the same number of facets than $\Ibb$, i.e., $\mathrm{pen}(\Ibb)\leq \mathrm{pen}(\tilde \Ibb)$ for any \mmaout{} $\tilde \Ibb$.

\svg{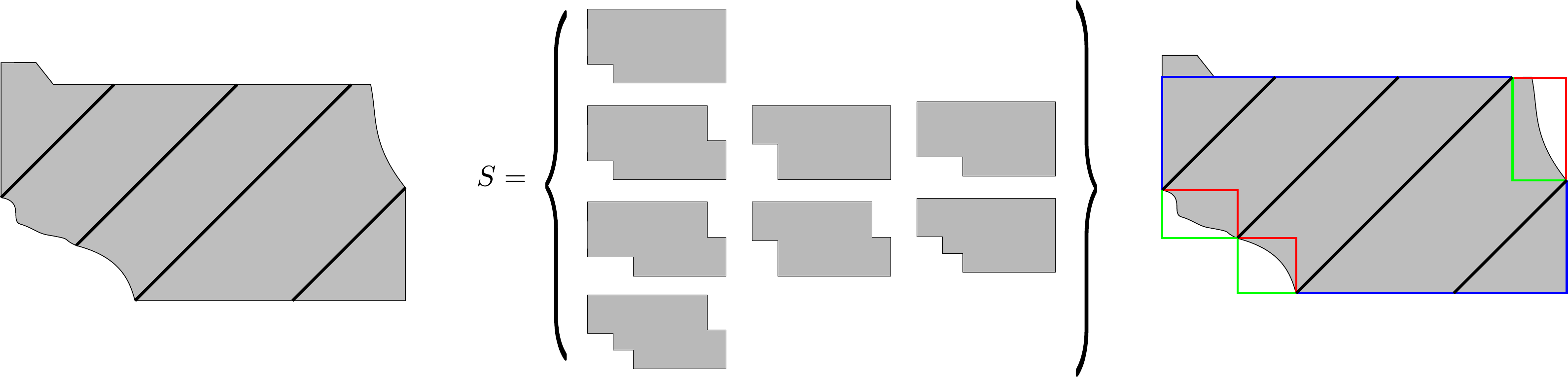}{0.9}{
\label{fig:candidate_argmin}
Example of \mmaout{} decompositions for a $2$-interval module $\Ibb$ with support in $\R^2$. 
\textbf{(Left)} Given the $L$-fibered barcode of $\Ibb$, where $L$ is the family of the four black lines, we want to approximate $\Ibb$ with an element of $S$, i.e.,
an interval module with the same fibered barcode. 
\textbf{(Middle)} When one further constrains the set $S$ %
by asking to have at most one corner between two consecutive endpoints of the fibered barcode, the whole set $S$ can be computed explicitly. 
\textbf{(Right)} The set $S$ can also be described as the set of intervals which have to go through the blue path, and which can arbitrarily
choose between the red or green path at three different locations. Hence, the cardinality of $S$ is $2^3$.
}

For interval modules, $S$ is generally a set of cardinal $c^d$, where $c$ is the number of candidate corners between 
birthpoints or deathpoints, and $d$ is the number of corners.  %
For instance, in Figure~\ref{fig:candidate_argmin}, one has $n=2$, $c=2$ and $d=3$. Unfortunately, $c$ is of the order of $2^{n-1}$, and thus grows exponentially with the dimension $n$, and $d$ is difficult to control in practice, since 
it heavily depends on the number of lines in the fibered barcode and the regularity of the underlying interval module $\Ibb$. 
Minimizing a penalty over $S$ 
is thus practical only for low dimension $n$ and small number of lines in the fibered barcode. Hence, our algorithm \MMA{} presented
in Section~\ref{subsec:Algorithms} does not use penalty minimization,
but is rather defined with natural and simple corner choices. \\

\begin{remark} 
Note also that there are cases when the corner choices are canonical. For instance, any $2$-persistence module $\Mbb$ with transition maps $\varphi_\cdot^\cdot$ that are  
{\em weakly exact}, i.e., that satisfy, for any
$x \le y$:
\begin{equation*}
	\mathrm{ im}\left( \varphi_{x}^y\right)  = \mathrm{ im}\left( \varphi_{(y_1,x_2)}^y\right) \cap \mathrm{ im}\left( \varphi_{(x_1,y_2)}^y\right) \text{    and    }
	\ker\left( \varphi_{x}^y\right) = \ker\left( \varphi_x^{(y_1,x_2)}\right) + \ker\left( \varphi_x^{(x_1,y_2)}\right),
\end{equation*}
is rectangle decomposable~\cite{botnanRectangleDecomposable2ParameterPersistence2022}.
Hence,
a canonical approximation of a summand $\Ibb$ of $\Mbb$ is given by the interval module whose support is the rectangle with corners $(\min_l (b^I_l)_1, \min_l (b^I_l)_2)$ and $(\max_l (b^I_l)_1, \max_l (b^I_l)_2)$,
where $l$ goes through the family of lines $L$ of the fibered barcode.
\end{remark}

\subsection{The \MMA{} algorithm for computing candidate decompositions}\label{subsec:Algorithms}

In this section, we introduce \MMA: a fast algorithm for computing $\delta$-\mmaout{} decompositions. The pseudo-code for \MMA{} is provided in Algorithm~\ref{algo:approx}.
Roughly speaking, given a \rebuttal{f.p. $n$-parameter persistence module $\Mbb$, an approximation parameter $\delta >0$, a $\delta$-grid of lines $L=L_\delta(K^\delta)$ where $K$
\compacityassumption{} $\Mbb$}, and a matching function $\sigma$
(see Section~\ref{sec:matching_mma} for a discussion about how to find such matching functions), Algorithm~\ref{algo:approx} works in three steps: \\

\begin{itemize}[left=1.5cm]
	\item[Step 1:] compute the $L$-fibered barcode of $\Mbb$,
	\item[Step 2:] match together bars
		using the matching function $\sigma$, %
	\item[Step 3:] for each summand, use the endpoints of the corresponding bars to compute estimates of the critical points, using Algorithm~\ref{algo:approx_inter}. \\
\end{itemize}

Step 1 can be performed using any persistent homology software (such as, e.g., \texttt{Gudhi}, \texttt{Ripser}, \texttt{Phat}, etc),
or with \texttt{Rivet}~\cite{lesnickInteractiveVisualization2D2015} when $n=2$. %
Our code is part of the \texttt{multipers} library \cite{loiseauxMultipersMultiparameterPersistence2024}, and can be found at
\url{https://github.com/DavidLapous/multipers}. Moreover, it uses the vineyard algorithm~\cite{cohen-steinerVinesVineyardsUpdating2006}, which allows us to run Steps 1 and 2 jointly
(see Section~\ref{subsec:vineyards}).

\vspace{3mm}
\begin{algorithm}[H]\label{algo:approx}
	\caption{\MMA: Multi-parameter persistence Module Approximation.}
	\textbf{Input 1:} A \rebuttal{f.p. $n$-parameter persistence module $\M$
	and a compact $K\subset \R^n$ that \compacityassumption{} $\Mbb$}, \\
	\textbf{Input 2:} $\delta$-grid of evenly spaced diagonal lines $L=L_\delta(K^\delta)$ \\
	\textbf{Input 3:} Matching function $\matching$ \\ %
	\textbf{Output:} Candidate decomposition $\tilde \Mbb^{\MMA{}}_\delta$ \\
	Compute $\fibered{\Mbb}_L$, %
	i.e., the $L$-fibered barcode of $\Mbb$;\\
	$S\leftarrow$ []; \textcolor{gray}{\footnotesize{\# $S$ is the set of interval summands of the output \mmaout{} decomposition, intialized as the empty set}}
	\\
	\For{$l \in L$}{
		\For{$[b_l^\Mbb,d_l^\Mbb]\in\barcode{\restr{\Mbb}{l}}$}{
			\textcolor{gray}{\footnotesize{\# Check whether it is in the image of the input matching}}\\
			\If{$\exists B \in S$ and $[b,d]\in B$ s.t. $[b_l^{\Mbb},d_l^{\Mbb}]=\matching([b,d])$}{
				$B$.append($[b_l^{\Mbb},d_l^{\Mbb}]$); \textcolor{gray}{\footnotesize{\# If it is, attach the bar to the corresponding summand}}
			}
			\textcolor{gray}{\footnotesize{\# Otherwise initialize a new summand with the bar}}\\
			\Else{
				Add $B:=[[b_l^\Mbb,d_l^\Mbb]]$ to $S$;
			}
		}
	}
	\textcolor{gray}{\footnotesize{\# For each summand in $S$ characterized by a set of bars, build an approximate interval summand}}\\ %
	\textbf{Return} $\tilde \Mbb^{\MMA{}}_\delta:= \bigoplus_{B\in S}$  \textsc{ApproximateInterval}$(B)$; %
\end{algorithm}

\vspace{3mm}

We now describe the algorithm \textsc{ApproximateInterval}, which is used at the end of Algorithm~\ref{algo:approx}.
Its pseudo-code is given in Algorithm~\ref{algo:approx_inter},
and is defined in two steps:\\

\begin{enumerate}
	\item first, we {\em label} birthpoints and deathpoints to identify facets with \textsc{LabelEndpoints}
		(Algorithm~\ref{algo:label_endpoints}),
	\item then, we use these labels to compute candidate critical points with \textsc{ComputeCorners} (Algorithm~\ref{algo:generate_corners}). \\
\end{enumerate}

\begin{algorithm}[H]\label{algo:approx_inter}
	\caption{\textsc{ApproximateInterval}}
	\textbf{Input:} Set of bars $B=\{[b_l,d_l)\}_{l\in L_B}$, where $L_B\subseteq L$ \\ %
	\textbf{Output:} Discretely presented interval module $\rebuttal{\induced{\tilde I(B)}}$ \\ %
	\textsc{labs} $\leftarrow$ \textsc{LabelEndpoints}$(B)$; \\
	$\birthLcp{B},\deathLcp{B} \leftarrow$ \textsc{ComputeCorners}$(B,$ \textsc{labs}$)$;\\
	$\tilde I(B)$ $\leftarrow$ $\bigcup_{c\in \birthLcp{B}}\bigcup_{c'\in \deathLcp{B}} R_{c,c'}$;\\
	{\bf Return} \rebuttal{$\induced{\tilde I(B)}$};
\end{algorithm}

\vspace{3mm}
We first describe \textsc{LabelEndpoints}.
The core idea of this algorithm, whose pseudo-code is given in
Algorithm~\ref{algo:label_endpoints}, is, for a given bar in $I$ associated to
a line $l\in L$, to look at the corresponding surrounding set $L_l$ %
(see item (4) in Definition~\ref{def:regularly_distributed}).
If there exists a hyperplane $H$ such that all endpoints in this surrounding set
belong to $H$, we identify $H$ as a facet, and we label the bar with the codirection of $H$. \\

\begin{algorithm}[H]\label{algo:label_endpoints}
	\caption{\textsc{LabelEndpoints}}
	\textbf{Input:} Set of bars $B=\{[b_l,d_l)\}_{l\in L_B}$, where $L_B\subseteq L$ \\ %
	\textbf{Output:} List \textsc{labs} of labels for each endpoint in $B$ \\
	\textsc{labs}$(b_l)\leftarrow []$ for all $l\in L_B$;\\
	\For{$l \in L_B$}{
		\If{$\exists i \in \llbracket 1,n\rrbracket $ and $c_i\in \R,$ such that $\forall l' \in L_l, \,(b_{l'})_i = c_i$}{
			Add $(i,c_i)$ to \textsc{labs}$(b_{l'})$ for all $l'\in L_l$;
		}
	}
	{\bf Return} \textsc{labs};
\end{algorithm}

\vspace{3mm}
Note that endpoints can have zero or more than one label. For instance, an endpoint
that belongs to the intersection of several facets might have multiple labels.
However, if several labels are identified, they must be associated to
different dimensions. See Figure~\ref{fig:label_algo} for examples of label assignments
when the underlying interval module has rectangle support.

\svg{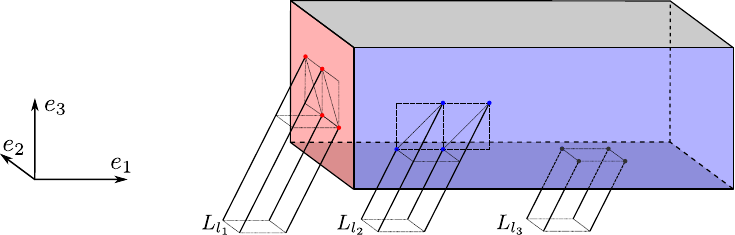}{0.8}{
	\label{fig:label_algo}
	Example of birthpoint labelling for an interval module $I$ with rectangle support
	with three surrounding sets of lines $L_{l_1}$, $L_{l_2}$, $L_{l_3}$
	associated to three lines $l_1,l_2,l_3$. The labels of $l_1,l_2,l_3$ that are identified
correspond to the red, blue and grey colored facets of $I$ respectively.}

Finally, we describe \textsc{ComputeCorners}.
The core idea of the algorithm, whose pseudo-code is given in Algorithm~\ref{algo:generate_corners}, is to
use the labels identified by \textsc{LabelEndpoints} to compute {\em corners},  or critical point estimates, in the following way:
if all birthpoints (resp. deathpoints) in a surrounding set have at least one associated facet,
i.e., have a non-empty list of labels, then a candidate corner
can be defined using the minimum (resp. maximum) of all birthpoints (resp. deathpoints) coordinates.
We only present the pseudo-code for birthpoints %
since the code for deathpoints is symmetric and can be obtained by replacing minimum by maximum and $-\infty$ by $+\infty$.
\rebuttal{Note that the correctness of \MMA{} follows directly from how these
	corners are computed: it is clear from Algorithm~\ref{algo:generate_corners}
	that, for any bar in the fibered barcode, the algorithm will produce a corner
that is lower (resp. larger) w.r.t. the partial order $\leqn$ than the birthpoint (resp. deathpoint) of the bar (excluding the trivial case of corners with infinite coordinates).}

\begin{algorithm}[h]
	\caption{\textsc{ComputeCorners} \label{algo:generate_corners}}
	\textbf{Input 1:} Set of bars $B=\{[b_l,d_l]\}_{l\in L_B}$, where $L_B\subseteq L$ \\
	\textbf{Input 2:} List \textsc{labs} of labels for each endpoint in $B$ \\ %
	\textbf{Output:} List of birth corners $C_B$ \\ %
	$C_B\leftarrow []$; \\ %
	\For{$l \in L_B$}{
		$B_{L_l}\leftarrow \{b_{l'}:l'\in L_l\cap L_B\}$; \textcolor{gray}{\footnotesize{\# Note that $B_{L_l}\subseteq B$ by construction}} \\
		\textcolor{gray}{\footnotesize{\# Check whether all birthpoints in the surrounding set belong to $K$}} \\ %
		\If{ $B_{L_l} \subseteq K$}{
			\textcolor{gray}{\footnotesize{\# Compute birth corner if all the birthpoints are labelled}}
			\\
			\If{\textsc{labs}$(b)\neq \varnothing, \forall b\in B_{L_l}$}
			{
				$\{(j, c_{j})\,:\,j\in \mathcal J\} \leftarrow \bigcup_{b\in B_{L_l}}$ \textsc{labs}$(b)$;
				\textcolor{gray}{\footnotesize{\# $\mathcal J\subseteq \llbracket 1,n\rrbracket$ is the set of codirections}}\\
				Define $C^l\in \R^n$ as
				\begin{itemize}
					\item $(C^l)_j = c_j$ if $j \in \mathcal J$
					\item $(C^l)_j = \min{
							\left\{
								(b_{l'})_j
								\,:\, l' \in L_l \cap L_B
							\right\}
						}$ otherwise
				\end{itemize}
				$C_B$.append$(C^l)$; \\ %
			}
			\textcolor{gray}{\footnotesize{ \# If the birthpoints are not all labeled,
					keep the birthpoints themselves as corners
			}}
			\\
			\Else{
				\For{$l'\in L_l\cap L_B$}{$C_B$.append$(b_{l'})$;}
			}
		}
		\textcolor{gray}{\footnotesize{\# If some birthpoints are not in $K$,
		they must correspond to infinite facets}} \\
		\Else{
			\ \textbf{Assert} $B_{L_l}\cap K^{\delta} \backslash K\neq\varnothing$; \\
			\textbf{Assert} \textsc{labs}$(b)\neq\varnothing$ for all $b\in B_{L_l}$; \\
			$\{(j, c_{j})\,:\,j\in \mathcal J\} \leftarrow \bigcup_{b\in B_{L_l}} \textsc{labs}(b)$; \textcolor{gray}{\footnotesize{\# The cardinality of $\mathcal J$ must be strictly less than $n$}}\\
			Define $C^l\in \R^n$ as:
			\begin{itemize}
				\item $(C^l)_j = c_j$ if $j \in \mathcal J$
				\item $(C^l)_j = -\infty$ otherwise
			\end{itemize}
			$C_B$.append$(C^l)$;
		}
	}
	{\bf Return} $C_B$;
\end{algorithm}

\paragraph*{Complexity.} Computing the $L$-fibered barcode $\fibered{\Mbb}_L$ on a simplicial complex, as well as assigning the corresponding bars to their associated
summands in the decomposition of $\Mbb$, can be done with the vineyard algorithm~\cite{cohen-steinerVinesVineyardsUpdating2006} as matching function with complexity
$O(N^3 + |L|\cdot N \cdot T)$, where $N$ is the number of simplices in the simplicial complex, and $T$ is the maximal number of transpositions required to update the
single-parameter filtrations corresponding to the consecutive lines in $L$. In the worst-case scenario, one has $T=N^2$.
Note that $T$ usually decreases to a fixed constant as $|L|$ increases, and that this computation can be easily parallelized in practice.

Now, adding the complexities
of Algorithms~\ref{algo:label_endpoints} and~\ref{algo:generate_corners},
the final complexity of Algorithm~\ref{algo:approx} is:
\begin{equation}\label{eq:MMA_complexity}
	O(N^3 + |L| \cdot N \cdot T  + |L|  \cdot n \cdot 2^{n-1}).
\end{equation}
Of importance, the dependence on $n$ is much better than the (exact) decomposition algorithm proposed in~\cite{deyGeneralizedPersistenceAlgorithm2022, deyDecomposingMultiparameterPersistence2025}
whose complexity is $O(N^{n(2\omega +1)})$, where $\omega < 2.373$ is the matrix
multiplication exponent.
It is also comparable
to \RIVET~\cite{lesnickInteractiveVisualization2D2015} (although \RIVET{} only works when $n=2$),
whose complexity is $O(N^3\kappa + (N+{\rm log}\kappa)\kappa^2)$,
where $\kappa=\kappa_x\kappa_y$ is the product of $x$ and $y$ coordinates used to evaluate the module
(note that $\kappa_x,\kappa_y$ are also user-dependent).
The elder-rule staircode~\cite{caiElderRuleStaircodesAugmentedMetric2021}
works only for point cloud data when $n=2$ and homology dimension $0$, but
has better complexity $O(m^2\log(m))$, where $m$ is the number of points.
Finally, note that our complexity
can be controlled by the number of lines, which is user-dependent. We illustrate this useful property in Section~\ref{sec:expe}. \\

\begin{remark}
	For the sake of simplicity and efficiency, the code that we provide at
	\url{https://github.com/DavidLapous/multipers} contains a simpler version of
	Algorithm~\ref{algo:generate_corners}, that does not compute and use labels,
	but rather gathers the birthpoints and deathpoints as corners directly. One can easily check that our approximation guarantees %
	(\cref{prop:tildeIcandidate} and~\cref{prop:approx}) carry over to that
	simpler algorithm, however the exactness result (\cref{prop:exact_recovery}) is only valid for corners computed with Algorithm~\ref{algo:generate_corners}.  \\
\end{remark}

\section{\rebuttal{Theoretical robustness of \MMA{}}}
\label{sec:approx}
\rebuttal{
In this section, our goal is to prove our first important result,
\cref{prop:tildeIcandidate}, which states that, if the matching function is {\em compatible}, then the \mmaout{} decompositions 
computed by \MMA{} are also {\em \mmaoutcompat{}} decompositions: they preserve the diagonal barcodes (up to $2\delta$) associated to {\em all} diagonal lines. We provide a proof in~\cref{subsec:approx_prop}. We also discuss the stability
of \MMA{} w.r.t. $\disti$ in~\cref{subsec:stab_prop}. %
}

\subsection{Approximation guarantee of \MMA{}}\label{subsec:approx_prop}

We first introduce so-called {\em compatible} matching functions, which are key elements for proving the approximation property satisfied by our \MMA{} algorithm. \\

\begin{definition}[Compatible matching function]\label{def:compat}
	Let $\Mbb$ be a {f.p.} $n$-parameter persistence module, and let $l_1,l_2\subseteq\R^n$ be two %
	diagonal lines
	that are at distance $\delta$ from each other. %
	Assume $\supp{\Mbb}\cap l_1$ and $\supp{\Mbb}\cap l_2$ are not empty,
	and let
	$[b_{1},d_{1})$  and $[b_2,d_{2})$ be bars in $\barcode{\restr{\Mbb}{l_1}}$ and $\barcode{\restr{\Mbb}{l_2}}$, characterized by their endpoints.
	These bars are \emph{compatible} if the rectangles %
	$\rectangle{b_{1}}{b_{2}}$, $\rectangle{b_{2}}{b_{1}}$, $\rectangle{d_{1}}{d_{2}}$ and $\rectangle{d_{2}}{d_{1}}$ are flat or empty. Equivalently, two bars are \emph{compatible} if their
	birthpoints (resp. deathpoints) are not strictly comparable, i.e., $b_1 \not <_n b_2$, and $b_1 \not >_n b_2$ (resp. $d_1 \not <_n d_2$, and $d_1 \not >_n d_2$).
	Moreover, we say that $[b^\Mbb_{l_1},d^\Mbb_{l_1})$ is \emph{compatible with
	the empty set} in $l_2$ if $ \norm{ b^\Mbb_{l_1}-d^\Mbb_{l_1}}_{\infty}\le 2 \delta$. %

	A {\em compatible matching function} is a matching function that only pairs bars that are compatible. \\
\end{definition}

\begin{remark}\label{rem:induced}
	Induced matching functions (see \cref{ex:exact_matching}) are compatible, as per~\cref{rem:flat_rect}. \\
\end{remark}

\begin{proposition}[Approximation result]\label{prop:tildeIcandidate}
	Let $\Mbb$ be a  %
	{f.p.}
	$n$-parameter persistence module and $\delta>0$. %
	Let $K$ be a rectangle in $\R^n$ %
	{that \compacityassumption{} $\Mbb$},
	and $L:=L_\delta(K^\delta)$ be the $\delta$-grid of lines of the offset $K^\delta$.
	Finally, let ${\tilde\Mbb^{\mMMA{}}_\delta :=\mMMA{}(\M,L,\matching)}$, where $\matching$ %
	is a compatible matching function.
	Then
	$\tilde \Mbb^{\mMMA{}}_\delta$ is a $\delta$-\mmaoutcompat{} decomposition of $\Mbb$. More precisely, given some diagonal line $l$, one has:
	\begin{enumerate}[label=\textnormal{(${\roman*}$)}]
		\item \label{enum:apprx_result_exact}$\distb\left(\barcode{\restr{\Mbb}{l}}, \barcode{\restr{\tilde\Mbb^{\emph{\MMA{}}}_\delta}{l}}\right)=0$ if $l\in L$,
			and
		\item \label{enum:apprx_result_bound}$\distb\left(\barcode{\restr{\Mbb}{l}}, \barcode{\restr{\tilde\Mbb^{\emph{\MMA{}}}_\delta}{l}}\right)\leq 2\delta$ otherwise. \\  %
	\end{enumerate}
\end{proposition}

In order for Proposition~\ref{prop:tildeIcandidate} to apply, one needs to find a compatible matching function $\matching$.
We discuss how to design such matching functions for interval decomposable modules in~\cref{app:compat_exact}
and for general $2$-parameter modules in~\cref{subsec:vineyards}.
We also hypothesize that compatible matching functions for general $n$-parameter persistence modules
could be constructed using representative cycles in a similar way than the
construction of the graphcode~\cite{russoldGraphcodeLearningMultiparameter2024, kerberRepresentingTwoparameterPersistence2025}, a conjecture that we leave for future work.

\begin{proof} %

	We first prove \ref{enum:apprx_result_exact}. Let $l\in L$, and let $\birthptapprox$ be the birthpoint of a bar $b$ in $\barcode{\restr{\Mbb}{l}}$.
	Let $B$ be the set of bars containing $b$ computed with
	Algorithm~\ref{algo:approx}, let $\birthLcp{B}$ and $\deathLcp{B}$ be the
	birth and death corners computed with Algorithm~\ref{algo:generate_corners}, and let $\tilde I$ be the interval computed with Algorithm~\ref{algo:approx_inter},
	i.e., one has:
	\begin{equation}\label{heuristic}
		 \tilde I =  \bigcup_{c \in \birthLcp{B}} \bigcup_{c' \in \deathLcp{B}} \rectangle{c}{c'} ,
	\end{equation}

	In order to show \ref{enum:apprx_result_exact},
	we first need to show that $\birthptapprox=b^{\tilde I}_l$ (and then the proof for deathpoints will follow by symmetry), where $b^{\tilde I}_l$ is defined as per \cref{def:birthpoints_deathpoints}.
	Note that $\birthptapprox$ and $b^{\tilde I}_l$ are comparable since they belong to the same diagonal line $l$. \\

	\noindent {\bf Strategy.} In order to show $\birthptapprox=b^{\tilde I}_l$, we
	are going to show that 1. $\birthptapprox\leqn b^{\tilde I}_l$ and 2. $b^{\tilde I}_l\leqn \birthptapprox$. %

	\begin{enumerate}
		\item In order to show $\birthptapprox\leqn b^{\tilde I}_l$,
			we are going to show that
			$c \not \sln \birthptapprox$ for any corner $c\in \birthLcp{B}$.
			Indeed, if one assumes $\birthptapprox \sgn b^{\tilde I}_l$ by contradiction, and since
			there always exists a birth corner $c\in \birthLcp{B}$ such that $c\leqn b^{\tilde I}_l$ by construction of $\tilde I$,
			one has $c\leqn b^{\tilde I}_l \sln \birthptapprox$.
		\item In order to show $b^{\tilde I}_l\leqn \birthptapprox$,
			we are going to show that
			there exists a corner $c \in \birthLcp{B}$ such that $c \leqn \birthptapprox$.
			Indeed, if there is such a birth corner, and if $b^{\tilde I}_l \sgn \birthptapprox$ by contradiction, then
			$c\leqn \birthptapprox \sln b^{\tilde I}_l$, and $\rectangle{c}{b^{\tilde I}_l}$ is not flat, contradicting Lemma~\ref{rem:flat_rect}.
	\end{enumerate}

	\noindent \textbf{Proof of (2).}
	By construction of $\tilde I$ with Algorithm~\ref{algo:generate_corners}, if $\birthptapprox$ is labelled, then there
	exists a line $l'$ and a corner $c^{l'}\in \birthLcp{B}$ that is smaller than $\birthptapprox$ so we can take $c:=c^{l'}$.
	If $\birthptapprox$ is not labelled, it belongs itself to $\birthLcp{B}$, and we can take $c:=\birthptapprox$. \\

	\noindent \textbf{Proof of (1).} Let $c \in \birthLcp{B}$ be a birth corner, and let $L_{l_0}$ be the associated surrounding set of lines for some $l_0 \in L$.
	Let $[c]_{l} :=\min \left[\left( c + \left( \R_+ \right)^n \right)\cap l\right]$ be the smallest element in
	the intersection between the positive cone on $c$ and $l$.
	Assume $[c]_{l}\geqn \birthptapprox$ and $c \sln \birthptapprox$. Then $\rectangle{c}{[c]_l}$ is not flat, contradicting the fact that $[c]_l$ is the smallest
	element. Thus, we only have to show $[c]_{l}\geqn \birthptapprox$.
	There are two cases.
	\begin{enumerate}
		\item Either some birthpoints of $L_{l_0}$ are not labelled by
			Algorithm~\ref{algo:label_endpoints}, and  $c$ is equal to the birthpoint $b_{l'}$ of another bar in $\barcode{\restr{\Mbb}{l'}}{\cap B}$ for some $l' \in L_{l_0}$.
			Now, assume $[c]_l \sln \birthptapprox$ by contradiction. Then $b_{l'}= c \leqn [c]_l \sln \birthptapprox$. Thus $b_{l'} \sln \birthptapprox$
			and $\rectangle{b_{l'}}{\birthptapprox}$ is not flat, contradicting
			{the fact that $\sigma$ is compatible}.
			Hence $[c]_l \geqn \birthptapprox$.

		\item Or all the birthpoints of $L_{l_0}$ are labelled by Algorithm~\ref{algo:label_endpoints}. Again, we study two separate cases. See Figure~\ref{fig:candidate_proof} for an illustration.
			\begin{enumerate}
				\item  Either $l\in L_{l_0}$. Then, $\exists i \in \llbracket 1,n\rrbracket$ such that $(\birthptapprox)_i=c_i$.
					This yields $(\birthptapprox)_i = c_i \le ([c]_l)_i$, and thus $[c]_l \geqn \birthptapprox$
					since they both belong to the same diagonal line $l$.

				\item Or the line $l$ does not belong to $L_{l_0}$. Since $[c]_l$ is on the boundary of the positive cone based on $c$, there exists
					$ i \in \llbracket 1,n \rrbracket$ such that
					$ ([c]_l)_i = c_i$. %
					Assume again by contradiction that $\birthptapprox \sgn [c]_l$,
					and write:
					\begin{equation*}
						[c]_l = c + \sum_{j\neq i} (\delta \alpha_j) e_j =: c +  \ora v \sln \birthptapprox,
					\end{equation*}
					with $ \alpha_j \ge 0$ for $j \in \llbracket 1,n \rrbracket \backslash \left\{ i\right\}$.
					Since $l \notin L_{l_0}$, there exists some $j_0$ such that $ \alpha_{j_0} >1$. %
					Let
					$\overrightarrow u := ((\overrightarrow v_j\, \mathrm{ mod}\, \delta)_{j\in \llbracket 1,n\rrbracket}) =
					( ([c]_l - c)_j\, \mathrm{ mod}\, \delta)_{j\in \llbracket 1,n\rrbracket} \in [0,\delta)^n \leqn \ora v$. %
					Let $l':=l_{c+\ora u}$ be the diagonal line passing through $c+\ora u$.
					Now, recall that the lines of $L$ are drawn on a grid, %
					so $l'\in L$
					since $l' = l + \ora u - \ora v$.
					Moreover, one has:
					by definition,
					$c\in \conv{L_{l_0}}$.
					Since the lines of $L$ are on a grid, one has:
					\begin{equation*}
						\forall l_1,l_2 \in L,\quad
					\Vert l_1 \cap H_n, \conv{L_{l_2}} \cap H_n)\Vert_\infty < \delta  \Longrightarrow l_1 \in L_{l_2},
				\end{equation*}
				where
				$H_n = \left\{ x\in \mathbb R^n : x_n = c_n \right\}$.
				Now,  note that $c+ \ora u$ and $c + \ora u -\ora u_n\cdot \textbf{1}$ both belong to $l'$, and that $c + \ora u -\ora u_n\cdot \textbf{1} \in H_n$. Moreover, since:
				$$\norm{ (c + (\ora u -\ora u_n\cdot \textbf{1})) - c}_\infty =  \norm{ \ora u -\ora u_n\cdot \textbf{1}}_{\infty} < \delta,$$
				one has $l'\in L_{l_0}$.
				Thus, {letting $b_{l'}$ be the birthpoint of the corresponding bar
				in $\barcode{\restr{\Mbb}{l'}} \cap B$}, there exists $ i'\in \llbracket 1,n\rrbracket$ such that $(b_{l'})_{i'} = c_{i'} \leq (c+\overrightarrow u)_{i'}$
				and thus $b_{l'} \leqn (c+\overrightarrow u)$ since $b_{l'}$ and $c+ \overrightarrow u$ are comparable on the diagonal line $l'$. Finally,
				$b_{l'} \leqn c+ \overrightarrow u \leqn c+ \overrightarrow v \sln \birthptapprox$,
				and $\rectangle{b_{l'}}{\birthptapprox}$ is not flat, contradicting
				the fact that $\sigma$ is compatible.
				Hence, $\birthptapprox\leqn [c]_l$.

		\end{enumerate}
		\begin{figure}[H]
			\centering{}
			\includesvg[width=.8\textwidth]{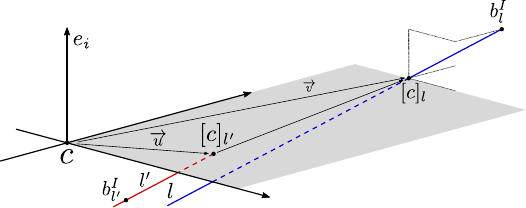}
			\caption{Illustration of $l,l',c,[c]_l, [c]_{l'}, \ora{u}, \ora{v}, \birthptapprox, b_{l'}$ when one assumes that $[c]_l \sln \birthptapprox$.}
			\label{fig:candidate_proof}
		\end{figure}
\end{enumerate}

The proof applies straightforwardly to deathpoints by symmetry.

Now, the proof of \ref{enum:apprx_result_bound} is then a simple consequence of
Lemma~\ref{lemma:compatible_close}.
Indeed, given a line $l\not\in L$, there
must be a line $l^*\in L$ such that $l^* = l+\ora{u}$ with
$\norm{\ora{u}}\leq \delta$ since $L$ fills $K$.
Then, one has:
\begin{equation*}
	\begin{array}{lcl}
		\distb\left(\barcode{\restr{\Mbb}{l}},
		\barcode{\restr{\tilde\Mbb}{l}}\right)
		&\leq&
		\distb\left(\barcode{\restr{\Mbb}{l}},
		\barcode{\restr{\Mbb}{l^*}}\right) + \distb\left(\barcode{\restr{\Mbb}{l^*}},
		\barcode{\restr{\tilde\Mbb}{l^*}}\right) +
		\distb\left(\barcode{\restr{\tilde\Mbb}{l^*}}, \barcode{\restr{\tilde\Mbb}{l}}\right)\\
		&\leq& \delta + 0 + \delta = 2\delta,
	\end{array}
\end{equation*}
by Lemma~\ref{lemma:compatible_close}.
\end{proof}

\paragraph*{Instability of \MMA{} w.r.t. interleaving distance $\disti$.} 

While using compatible matching functions helps controlling the diagonal fibered barcodes, it is unfortunately not sufficient for bounding the interleaving distance: 
outputs of \MMA{} can be very far in terms of $\disti$ while the modules they are computed from are not. 
We provide two
multi-parameter persistence modules in Figure~\ref{fig:non-stable-decomposition} that illustrate such lack of stability. In this figure, the two $2$-parameter filtrations only differ on the middle edge (in blue) of the simplicial complex. When the appearance of this edge is delayed (as is the case for the multi-parameter filtration displayed on top), the bars in the barcodes corresponding to the lower and upper cycles of the simplicial complex get paired by the compatible matching function, and create together the large red summand. On the other hand, this does not happen for the other multi-parameter filtration: the bars corresponding to these cycles are never paired and form distinct interval summands with same size.

\begin{figure}
	\centering
	\includegraphics[width=0.6\textwidth]{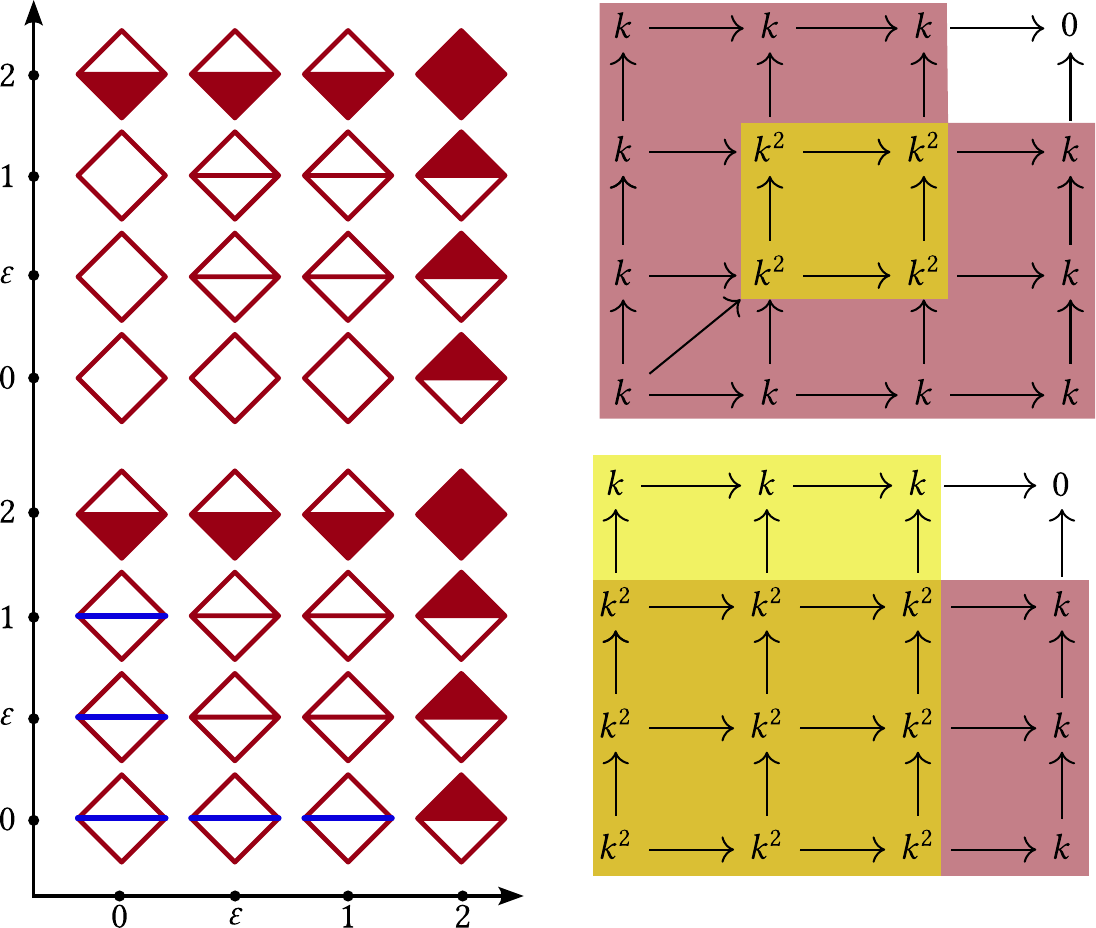}
	\caption{
		\textbf{(Left)} Two bi-filtrations whose corresponding multi-parameter persistence modules in homology dimension $1$ are $\varepsilon$-interleaved.  %
		\textbf{(Right)} The two corresponding, significantly different interval decompositions obtained with \MMA{} computed with a compatible matching. 
		Intervals in these decompositions are displayed with red and yellow colors.}
	\label{fig:non-stable-decomposition}
\end{figure}

Figure~\ref{fig:non-stable-decomposition} illustrates the
price to pay for designing interpretable decompositions (i.e., such that each summand corresponds to a cycle of the simplicial complex) when several choices are possible: there are several different ways to assign cycles to summands in the non-interval decomposable module displayed on top of Figure~\ref{fig:non-stable-decomposition}---the one computed by \MMA{} (shown in the figure) being one possibility. An important conjecture of this article, that is left for future work, is that stacking the candidate decompositions produced by our \MMA{} algorithm for all possible compatible matching functions induces a {\em complete} topological invariant of the module. 

\subsection{Stability property of \MMA{}}\label{subsec:stab_prop}

As it is not possible to control powerful distances such as $\disti$, 
we end this section by ensuring that \MMA{} can still remain stable w.r.t. to the data itself by appropriately 
choosing the matching functions.
Indeed, given two multi-parameter filtrations computed from the sublevel sets of functions $f,g$, Equation~(\ref{eq:stability}) ensures that the bottleneck distances between barcodes in the fibered barcodes of $\M(f)$ and $\M(g)$ are upper bounded by $\norm{f-g}_\infty$. This in turns means that we can fix an (arbitrary) compatible matching function $\matching_f$ (if it exists) for computing an \mmaoutcompat{} decomposition of $f$ with \MMA{}, and define another one $\matching_g$ that commutes with $\matching_f$ and the optimal partial matching $\nu$ given by those bottleneck distances. Doing so 
leads to the following proposition. \\ %

\begin{proposition}[Enforced stability]\label{prop:mma_stab}
	Let $f,g:\simpcomp\to\R^n$ be two continuous functions defined on a
	topological space $\simpcomp$, and let $\Mbb(f)$ and $\Mbb(g)$ be 
	the
	multi-parameter persistence modules associated to the homology groups of their
	sublevel sets. %
	Let $K$ be a rectangle in $\R^n$ %
	\rebuttal{that \compacityassumption{}  
	$\Mbb(f)$ and $\Mbb(g)$},
	and $L:=L_\delta(K^\delta)$ be the $\delta$-grid of lines of the offset $K^\delta$.
	Finally, let $\sigma_f$ be an arbitrary compatible matching function.
	Then, there exists a matching function $\sigma_g$,
	such that the following diagram commutes:
	\begin{equation}
		\begin{tikzcd}
			\barcode{\restr{\Mbb(f)}{l}}
			\arrow{r}{\sigma_{f}}
			\arrow{d}{}
			&
			\barcode{\restr{\Mbb(f)}{l'}}
			\arrow{d}{}
			\\
			\barcode{\restr{\Mbb(g)}{l}}
			\arrow{r}{\sigma_{g}}
			&
			\barcode{\restr{\Mbb(g)}{l'}}
		\end{tikzcd}.
	\end{equation}
	In particular, assuming that $\sigma_g$ is also compatible,
	if we let: %
	\begin{equation}
		\tilde\M^{\mMMA{}}_\delta(f):=\mMMA(\Mbb(f),L,\matching_f)
		\quad
		\textnormal{and}
		\quad
		\tilde\M^{\mMMA{}}_\delta(g):=\mMMA(\Mbb(g),L,\matching_g),
	\end{equation}
	we have the following stability inequality:
	\begin{equation}
		\disti(\tilde\M^{\mMMA{}}_\delta(f), \tilde\M^{\mMMA{}}_\delta(g))\leq
		\distb(\tilde\M^{\mMMA{}}_\delta(f), \tilde\M^{\mMMA{}}_\delta(g))\leq
		\norm{f-g}_\infty + \delta.
	\end{equation}
\end{proposition}

\begin{proof}
	Let $\matching_g$ be the matching function induced by the following commutative diagram:
	\begin{equation}
		\begin{tikzcd}
			\barcode{\restr{\Mbb(f)}{l}}
			\arrow{r}{\sigma_{f}}
			\arrow{d}{\nu_{l,\distb}}
			&
			\barcode{\restr{\Mbb(f)}{l'}}
			\arrow{d}{\nu_{l',\distb}}
			\\
			\barcode{\restr{\Mbb(g)}{l}}
			\arrow{r}{\sigma_{g}}
			&
			\barcode{\restr{\Mbb(g)}{l'}},
		\end{tikzcd}
	\end{equation}
	where $\nu_{l,\distb}$ denotes the optimal partial matching induced by
	$\distb(\barcode{\restr{\Mbb(f)}{l}}, \barcode{\restr{\Mbb(g)}{l}})$.

	Let $k^{I_f},k^{I_g}$ denote two interval summands of $\tilde\M^{\mMMA{}}_\delta(f)$ and $\tilde\M^{\mMMA{}}_\delta(g)$ that are paired by $\nu_{\cdot,\distb}$. Then, controlling the bottleneck distance between these outputs of \MMA{} simply amounts to controlling the Hausdorff distance between ${I_f}$ and ${I_g}$.
	In order to control this distance, let $B_f$ and $B_g$ denote the bars that induced these interval summands as per Algorithm~\ref{algo:approx_inter}.
	Then, one has $B_f\subseteq {I_f}$ and $B_g\subseteq B_f^\gamma$, where $\gamma=\norm{f-g}_\infty$.
	Thus, since $\supp{I_g}\subseteq B_g^\delta$ (by construction), one has $I_g\subseteq B_f^{\gamma+\delta}\subseteq \supp{I_f}^{\gamma+\delta}$. The result follows by symmetry of $f$ and $g$.
\end{proof}
\rebuttal{
Note that finding compatible matching functions can be weakened into a convex problem (as compatibility can be checked with a sequence of inequalities),
thus inducing an open question: is it possible to define
an optimization problem whose minimization would yield a matching function
that is both compatible \emph{and stable}
with respect to the input data?
The paragraph at the end of \cref{subsec:approx_prop} shows in particular
that adding another approximation term is necessary if one wants to
avoid decomposition instabilities.
}

\section{\rebuttal{The case of interval decomposable modules}}
\label{sec:exact_algo}

In this section, we 
\rebuttal{refine~\cref{prop:tildeIcandidate} to
interval decomposable modules. In particular, we show that, upon using
{\em induced} matching functions (see~\cref{ex:exact_matching}), the 
\mmaoutcompat{} decompositions computed by our \MMA{} algorithm become
{\em stable w.r.t. the interleaving and bottleneck distances} (\cref{prop:approx})
in \cref{subsec:proof_approx},
and that the input interval decomposable module can even be recovered exactly 
for a small yet positive $\delta$ 
(\cref{prop:exact_recovery}) in~\cref{subsec:exact}.}

\subsection{Stability w.r.t. interleaving and bottleneck distances
}\label{subsec:proof_approx}

The goal of this section is to show the following result: \\

\begin{proposition}[Stability result]\label{prop:approx}
	Let $\Mbb$ be a f.p. interval decomposable $n$-parameter persistence module.
	Let $K$ be a rectangle in $\R^n$
	that \compacityassumption{} $\Mbb$,
	and $L:=L_\delta(K^\delta)$ be the $\delta$-grid of lines of the offset $K^\delta$.
	Finally, let ${\tilde\Mbb^{\mMMA{}}_\delta :=\mMMA{}(\M,L,\matching)}$, where $\matching$ %
	is a matching function that commutes with the induced matching function $\sigma_\Mbb$.
	More precisely, denoting $\Mbb=\bigoplus_{i\in \mathcal I}\induced{I_i}$
	and $\tilde\Mbb^{\mMMA{}}_\delta=\bigoplus_{i\in \tilde{ \mathcal I}}\induced{\tilde I_i}$, this means that there exists a
	bijection $\nu \colon \mathcal I_L \to \tilde{ \mathcal I}$, where $\mathcal I_L=\{i\in\mathcal I : \interval_i\cap L \neq\varnothing\}$,
	such that, %
	for any two %
	lines $l,l' \in L$, the following diagram commutes:
	\[
		\begin{tikzcd}
			\barcode{\restr{\Mbb}{l}}
			\arrow{r}{\matching_{\Mbb}}
			\arrow{d}{\nu_{l}}
			&
			\barcode{\restr{\Mbb}{l'}}
			\arrow{d}{\nu_{l'}}
			\\
			\barcode{\restr{\tilde \Mbb^{\mMMA{}}_\delta}{l}}
			\arrow{r}{\matching}
			&
			\barcode{\restr{\tilde \Mbb^{\mMMA{}}_\delta}{l'}},
		\end{tikzcd}
	\]
	where
	$\nu_{l}\colon \restr {\interval_i} {l}  \in \barcode{\restr{\Mbb}{l}} \mapsto \restr{\tilde \interval_{\nu(i)}}{l} \in \barcode{\restr{\tilde \Mbb^{\mMMA{}}_\delta}{l}}$ (and similarly for $l'$).

	Then, %
	one has:
	$$
	\disti(
	\restr{\tilde \M^{\emph{\MMA{}}}_\delta} K, \restr \M K)
	\leq
	\distb (\restr{\tilde \M^{\emph{\MMA{}}}_\delta}  K, \restr \M  K) \leq \delta.$$
	\\
\end{proposition}

One might wonder how to construct matching functions that commute with
$\matching_{\Mbb}$ in practice. It turns out that any compatible matching function,
as well as the matching functions associated to the Wasserstein distances and
the vineyards algorithm, all commute with the induced matching $\matching_{\Mbb}$ for small enough $\delta$ and under some generic assumptions, as we show in~\cref{app:compat_exact}.

Note also that it is possible to generalize Proposition~\ref{prop:approx} to
modules that are not restricted to $K$ by constraining the parts of the
\mmaout{} decompositions that are outside of $K$ with Kan extensions, but we stick to our formulation for the sake of simplicity. We will
now prove a few technical results and lemmas in Section~\ref{subsec:lemmas}, and we finally prove Proposition~\ref{prop:approx} in Section~\ref{subsec:main_proof}.

\subsubsection{Additional lemmas}\label{subsec:lemmas}

In this section, we prove a few preliminary results about endpoints of interval modules, that will turn out useful for
proving Proposition~\ref{prop:approx}. %

\paragraph*{Endpoint location.}
The next definition and result show that endpoints of an interval module must
be located in the vicinity of the other endpoints of the module that are close
to it---more precisely, in their {\em rectangle hull}. This will be useful in
the proof of Proposition~\ref{prop:approx}; in particular, we will use this
result to characterize the positions of endpoints of any given diagonal line $l$ solely from the endpoints of the lines of the grid $L$ that are close to $l$. \\

\begin{definition}
	Let $S\subseteq \R^n$. The {\em rectangle hull} of $S$, denoted by $\mathrm{ recthull}[S]$, is defined as the smallest rectangle containing $S$:
	\begin{equation*}
		\mathrm{ recthull}[S] := \left\{ x \in \R^n : \forall i\in \llbracket 1,n\rrbracket, \min_{s\in S}s_i \le x_i \le \max_{s\in S} s_i\right\} = \rectangle{\wedge S}{\vee S},
	\end{equation*}
	where $(\wedge S)_i := \min_{s\in S}s_i$ and $(\vee S)_i := \max_{s\in S}s_i$. \\
\end{definition}

\begin{lemma}[Endpoints bound]\label{lemma:boundary_bound}
	Let $k^I$ be a {f.p.} interval module. Let $K$ be a rectangle in $\R^n$ and $L:=L_\delta(K^\delta)$ be the $\delta$-grid of lines of the offset $K^\delta$.
	Let $x\in K$, $l_x$ be the diagonal line passing through $x$, and let $
	L_{x,\delta}:=\{l\in L\,:\, d_\infty(x,l) \leq \delta \text{ and } l_x, l \text{ are $\delta$-comparable}\}
	$, which is non-empty since
	$L$ $\delta$-fills $K$.
	Assume that $l_x\cap I\neq \varnothing$, and $d^I_x \in \upperb{I}$ be the associated deathpoint, and
	assume that for any line $l$ in $L_{x,\delta}$, one has $I\cap l\neq \varnothing$, and
	let $D_{x,\delta}^I$ be the set of the associated deathpoints: $D_{x,\delta}^I=\{d_l^I\,:\, l\in L_{x,\delta}\}$.
	Then, $d^I_x$ belongs to the rectangle hull of a subset $\tilde
	D_{x,\delta}^I$ of $D_{x,\delta}^I$: one has $d^I_x\in \mathrm{recthull} [\tilde D_{x,\delta}^I]$ with $\tilde D_{x,\delta}^I \subseteq D_{x,\delta}^I$.

	Similarly, if $b^I_x\in \lowerb{I}$ is a birthpoint, then $b^I_x\in\mathrm{
	recthull}[\tilde B_{x,\delta}^I]$, where $\tilde B_{x,\delta}^I$ is a subset of $B_{x,\delta}^I=\{b_l^I\,:\, l\in L_{x,\delta}\}$, i.e., the set of birthpoints associated to $L_{x,\delta}$. \\
\end{lemma}

In other words, the endpoints of an interval module always belong to the rectangle hull of the endpoints associated to neighbouring lines.
See Figure~\ref{fig:boundary_bound} for an illustration.
\begin{figure}[H]
	\begin{center}
		\includesvg[width=.6\textwidth]{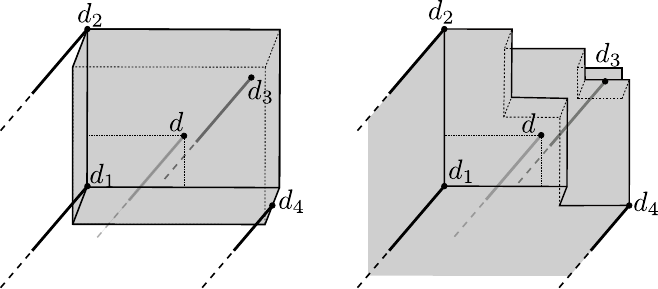}
	\end{center}
	\caption{ Example of deathpoint bound in $\R^3$, with $d\in \upperb{I}$, and $D_{x,\delta}^I= \left\{d_1,d_2,d_3,d_4 \right\}$.
		\textbf{(Left)} Rectangle hull of the deathpoints $D_{x,\delta}^I$.
	\textbf{(Right)} Upper-boundary $\upperb{I}$.}
	\label{fig:boundary_bound}
\end{figure}

\begin{proof}
	We first prove the result for deathpoints.
	Note that the result is trivially satisfied if $d^I_x$ and the deathpoints in $D^I_{x,\delta}$ are infinite,
	so we assume that they are finite in the following.
	To alleviate notations, we let $d:=d^I_x$.
	Let $j \in \llbracket 1,n \rrbracket $ be an arbitrary dimension.
	In order to prove the result, we will show that there exist two deathpoints $\underline{d}$ and $\overline{d}$ associated to consecutive lines of $L_{x,\delta}$ such that
	$
	\underline{d}_j\le d_j \le \overline{d}_j
	$. \\

	\noindent {\bf Construction of $\underline{d}, \overline{d}$. } Let $H_j$ be the hyperplane $H_j=d + e_j^{\perp}$. Since $L$ $\delta$-fills $K$, %
	there exists a diagonal line $l\in L$ such that $d_\infty(x,l)\le \delta/2$.
	Moreover, since $l$ and $l_x$ (the line passing through $x$ and $d$) are both diagonal, one has
	$d_\infty(d,l) = d_\infty(x,l)\le \delta/2$. Let $\pi_l(d)\in l$ be the projection of $d$ onto $l$ that achieves $d_\infty(d,l)$, and
	let $d^{j} := l \cap H_j$. See Figure~\ref{fig:proof} for an illustration of these objects.

	\svg{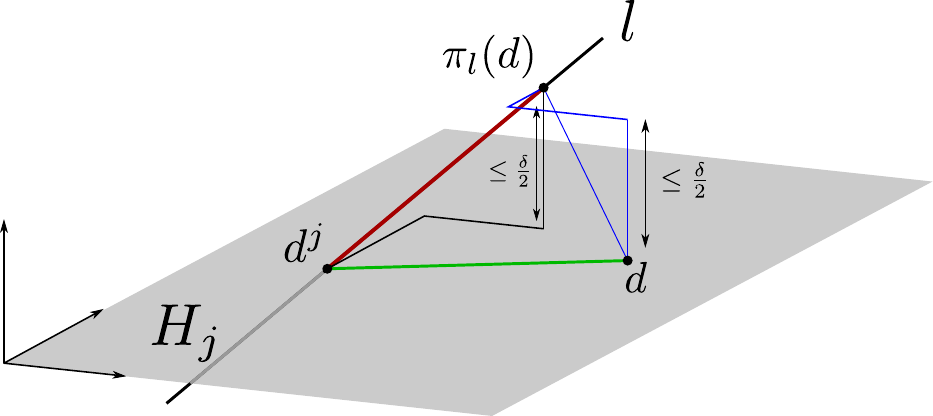}{0.5}{\label{fig:proof} Illustration of $H_j, d, l, d^j$. }

	Since $d^j$ and $d$ belong to $H_j$, they have the same $j$-th coordinate: $d^j_j=d_j$. Moreover, both $d^j$ and $\pi_l(d)$ belong to the diagonal line $l$,
	hence they are comparable, and $\Vert d^j-\pi_l(d)\Vert_\infty = |(d^j-\pi_l(d))_i|$ for any $i \in \llbracket 1, n\rrbracket$. Then, one has
	$\Vert d^{j} - d\Vert_{\infty} \le \Vert d^{j} - \pi_l(d)\Vert_{\infty} +
	\Vert \pi_l(d) - d\Vert_{\infty} = |(d^{j} - \pi_l(d))_j| + \Vert \pi_l(d) -
	d\Vert_{\infty} = |(d - \pi_l(d))_j| + \Vert \pi_l(d) - d\Vert_{\infty} \leq 2 \Vert \pi_l(d) - d\Vert_{\infty} \leq \delta$.
	Let
	$d^+ = d^j + \delta \sum_{j \in \mathcal J''} e_j \text{ and } d^- = d^j- \delta \sum_{j \in \mathcal J'} e_j$,
	where:
	\begin{equation*}
		\mathcal J' = \left\{ i \in \llbracket 1,n \rrbracket\backslash\{j\} : d_i < d^j_i \right\}
		\text{ and }
		\mathcal J'' = \left\{ i \in \llbracket 1,n \rrbracket\backslash\{j\} : d_i > d^j_i \right\}.
	\end{equation*}
	By construction, one has $d^- \le d \le d^+ \in H_j$. and
	\begin{equation}\label{eq:surr}
		\Vert d^+ - d\Vert_\infty, \Vert d^- - d\Vert_\infty \le \delta.
	\end{equation}
	Since $l$ and the diagonal lines $\underline l$ and $\overline l$ passing
	through $d^-$ and $d^+$ respectively are $\delta$-consecutive, and since $x\in
	K$, the projections of $x$ onto $\underline l$ and $\overline l$ are in
	$K^{\delta}$, and thus $\underline l, \overline l $ must belong to $L$, and
	thus to $L_{x,\delta}$, as by construction $l_x$ is $\delta$-comparable with the diagonal lines $\underline l$ and $\overline l$.
	Let $\overline{d}:=d_{\underline{l}}^I\in \underline{l}$ and $\underline{d}:=d_{\overline{l}}^I\in \overline{l}$ be their deathpoints (which exist by assumption).
	\\

	\noindent {\bf Proof of inequalities.} We now show that $\overline{d}_j \geqn
	d_j \geqn \underline{d}_j$. We start with the second inequality. Since $d^+$
	and $\underline{d}$ are one the same diagonal line, they are comparable.
	Furthermore, if one had $d^+ \sln \underline{d}$ by contradiction, then the induced rectangle $\rectangle{d}{\underline{d}}$ would not be flat since $d \leqn d^+ \sln \underline{d}$,
	which would contradict Lemma~\ref{rem:flat_rect}.
	As a consequence, $d^+ \geqn \underline{d}$. Taking the $j$-th coordinate yields $d_j = d^+_j \ge \underline{d}_j$. The first inequality holds
	using the same arguments. \\

	This proof applies straightforwardly to birthpoints by symmetry.
\end{proof}

Using Lemma~\ref{lemma:compatible_close}, one can generalize Lemma~\ref{lemma:boundary_bound} above to the case where
some lines in $L_{x,\delta}$ have an empty intersection with $I$, and then define a common location for all endpoints that belong to the
convex hull of the same $L$-surrounding set, as we do in the following proposition. \\

\begin{proposition}\label{prop:local_bound}
	Let ${k^{I}}$ be a {f.p.} interval module. Let $K$ be a rectangle in $\R^n$ and $L:=L_\delta(K^\delta)$ be the $\delta$-grid of lines of the offset $K^\delta$.
	Let $l\in L$ such that $|L_l| = 2^{n-1}$ and assume that $\conv{L_l} \cap \lowerb{I}$ (resp. $\conv{L_l} \cap \upperb{I}$) is not empty.
	Then, there exists a set $B_l$ (resp. $D_l$) such that for any $x\in \conv{L_l} \cap \lowerb{I}$ (resp. $\conv{L_l} \cap \upperb{I}$),
	one has either  $\Vert b_x^I -d_x^I \Vert_\infty \le \delta$, or
	$x \in B_l$ (resp. $D_l$), where $B_l$ (resp. $D_l$) is a rectangular set in
	$\R^n$ that can be constructed from the birthpoints $(b_{l'}^I)_{l'\in L_l}$  (resp. deathpoints $(d_{l'}^I)_{l'\in L_l}$). %
	Moreover, one has that:
	\begin{equation}\label{eq:diam}
		{\rm diam}(B_l)={\rm sup}_{x,x'\in B_l}\norm{x-x'}_\infty \leq \delta,
	\end{equation}
	and similarly for $D_l$. %

\end{proposition}

\begin{proof} We first construct $B_l$ and $D_l$, and then we will show items (1) and (2). \\

	{\bf Definition of $B_l,D_l$.}
	Let first assume that $x$ is in the interior of $ \conv{L_l}$, that we denote with $\interior{\conv{L_l}}$.
	Note that if there is a line $l_0$ that is $\delta$-comparable to $l_x$, and such that $\barcode{\restr{\induced{I}}{l_0}} = \varnothing$,  %
	then by Lemma~\ref{lemma:compatible_close} $(i)$, one immediately has $\Vert b_x^I - d_x^I \Vert_{\infty} \le \delta$.
	Hence, we now assume that the barcodes along any %
	line that is $\delta$-comparable to $l_x$ is not empty,
	which means that the hypotheses
	of Lemma~\ref{lemma:boundary_bound} are satisfied for $x$.
	Now, remark that since $L$ is a grid, if one is able to find a line $l'$ in $L$ whose intersections
	with hyperplanes associated to the canonical axes of $\R^n$ are $\delta$-close to $x$, then, since $x$
	is in the interior of an $L$-surrounding set $L_l$, $l'$ must belong to that surrounding set $L_l$ as well. More formally,
	one has that, for any line $l'\in L$:
	\begin{equation*}
		d_{\infty}(x, l'\cap H_i)\le \delta \quad\Longrightarrow\quad l' \in L_l,
		\quad
		\text{ where }
		H_i = \left\{ y \in \R^n : y_i = x_i\right\}.
	\end{equation*}

	This ensures (from Equation~\ref{eq:surr}) that the lines of $L$ associated to $\tilde B^I_{x,\delta}$ and  $\tilde D^I_{x,\delta}$
	are all included in $L_l$ for any $x\in\interior{\conv{L_l}}$, and thus that we can safely define:
	\begin{equation*}
		D_l := \overline{\bigcup_{x \in \interior{\conv{L_l}}}\mathrm{ recthull}[\tilde D_{x,\delta}^I]}
		\quad \text{and} \quad
		B_l := \overline{\bigcup_{x \in \interior{\conv{L_l}}}\mathrm{ recthull}[\tilde B_{x,\delta}^I] }.
	\end{equation*}
	Note that $B_l$ and $D_l$ depend only on the endpoints of the lines in $L_l$ %
	and that $d_x^I \in D_l$ and $b_x^I \in B_l$ for any $x\in \interior{\conv{L_l}}$ by Lemma~\ref{lemma:boundary_bound}.
	Furthermore, if $x$ is in the closure of $\conv{L_l}$, the previous statements still hold
	since $D_l$ and $B_l$ are closed sets.
	We now show that $B_l$ and $D_l$ satisfy Equation~(\ref{eq:diam}). \\

	{\bf Proof of Equation~(\ref{eq:diam}). } By applying
	Lemma~\ref{lemma:boundary_bound} and its proof for arbitrary dimension $j$ to all $x\in\conv{L_l}\cap U[I]$, %
	there exist deathpoints $\overline d^j$ and $\underline d_j$
	that satisfy $(\overline d^j)_j = \sup_{d\in D_l} d_j$ and $(\underline d_j)_j = \inf_{d\in D_l} d_j$ and $(\underline d_j)_j \le x_j \le (\overline d^j)_j$ for all $x\in\conv{L_l}\cap U[I]$.
	Moreover, these points are located on lines in $L_l$
	that are $\delta$-consecutive by definition.
	Thus, applying Lemma~\ref{lemma:compatible_close} $(ii)$ repeatedly on these
	pairs of lines for all dimensions, we end up with $D_l$ having a diagonal smaller than $\delta$. The same goes for birthpoints. %
\end{proof}

\subsubsection{Proof of Proposition~\ref{prop:approx}}\label{subsec:main_proof}
\begin{proof}
	Let $\Mbb= \bigoplus _{i \in \mathcal I}{\field^{I_i}}$ and
	$\tilde \Mbb^{\MMA{}}_\delta = \bigoplus_{i \in \tilde{ \mathcal I}} {\field^{\tilde I_i}}$
	be the interval decompositions of $\Mbb$ and $\tilde \Mbb^{\MMA{}}_\delta$,
	with induced matching functions $\matching_{\M}$ and $\matching$ respectively.
	In order to upper bound the bottleneck distance $\distb(\M,\tilde \M^{\MMA{}}_\delta)$,
	one can upper bound the interleaving distance $\disti({k^{I_i}}, {k^{\tilde I_{\nu(i)}}})$ for any index $i\in \mathcal I$.
	Let $I$ and $\tilde I$ be two such intervals (we drop the index $i$ to alleviate notations).
	We need to show that the
	morphisms
	$f^{ (\bound\delta)} \colon {\field^{I}} \rightarrow {\field^{\tilde I}}(\bound\boldsymbol \delta)$  and
	$g^{ (\bound\delta)} \colon {\field^{\tilde I}} \rightarrow {\field^{I}}(\bound\boldsymbol \delta )
	$
	exist and commute, i.e., that they induce a $\bound\delta$-interleaving.
	Hence, we first show that:
	\begin{equation}\label{to_show_approxRn}
		g^{ (\bound\delta)}_{x+\bound\boldsymbol \delta}
		\circ
		f^{(\bound\delta )}_x =
		\varphi^{x+2\boldsymbol \delta}_x,
	\end{equation}
	for any $x\in K$, where $\varphi_{\cdot}^{\cdot}$ denote the transition maps of $\induced{I}$. \\

	If $x\in l$ for some line $l\in L$, Equation~(\ref{to_show_approxRn}) is satisfied from ${I}\cap l={\tilde I}\cap l$,
	which itself comes from the fact that $\tilde \M^{\MMA{}}_\delta$ has the same $L$-fibered barcode than $\M$ (see~\cref{prop:tildeIcandidate} $(i)$). Hence, we assume in the following
	that $x \not \in \cup_{l\in L}l$.
	Furthermore, if $x\not\in {I}$ or $x+\twobound\boldsymbol \delta \not\in {I}$, then %
	Equation~(\ref{to_show_approxRn}) is trivially satisfied. %
	Hence, we also assume  $x, x+\twobound\boldsymbol \delta \in {I}$.
	This means that $b^I_x$ and $d^I_x$ are well-defined, and that
	$
	\varphi^{x+2\boldsymbol \delta}_x
	\cong \mathrm{ id}_{\field\to\field}
	$. Thus we only have to show that ${(k^{\tilde I})}_{x+ \bound\boldsymbol \delta} \cong \field$, i.e., $x+\bound\boldsymbol \delta\in {\tilde I}$. \\

	As $L=L_\delta(K^\delta)$ is the $\delta$-grid of $K^\delta$ and $x\in K$, let
	$l\in L$ be a line such that $x\in \conv{L_l}$ and let $l_x \subseteq \conv{L_{l}}$ be the diagonal line passing through $x$.
	Now, as $\rectangle{x}{x + \boldsymbol \delta} \subseteq {I}$, Lemma~\ref{lemma:compatible_close} $(1)$ ensures that $\barcode{\restr{{\induced{I}}}{l}}\neq \varnothing$
	for any line $l\in L$ that is $\delta$-comparable to $l_x$; and the same holds
	for $\tilde I$ since $\fibered{{\field^I}}_L=\fibered{{\induced{\tilde I}}}_L$. %
	Using Proposition~\ref{prop:local_bound} on both $I$ and $\tilde I$, %
	there exist two sets $B_l$ and $D_l$ such that
	$d_x^I, d_x^{\tilde I} \in D_l$ and $b_x^I, b_x^{\tilde I} \in B_l$,
	with the segments $l_x \cap B_l$ and $l_x \cap D_l$ having length at most $\delta$. %
	Since one also has: %
	\begin{equation*}
		b_x^I \le x \le x+ 2\boldsymbol \delta \le d_x^I,
	\end{equation*}
	and $ \norm{ d_x^I - d_x^{\tilde I}}_\infty, \norm{ b_x^I - b_x^{\tilde I}}_\infty \le \delta$, one finally has
	$
	b_x^{\tilde I} \le x+ \boldsymbol \delta \le d_{x}^{\tilde I},
	$
	which concludes that $x + \boldsymbol \delta \in {\tilde I}$.
	Since $I$ and $\tilde I$ are interchangeable in the arguments above, the result follows.
\end{proof}

\subsection{Exact reconstruction}\label{subsec:exact}

In this section, we identify the interval decomposable multi-parameter persistence modules that can be recovered
exactly with our \MMA{} algorithm.
Given a precision parameter $\delta > 0$, they correspond to modules that can decomposed into interval summands %
that form a subclass of the family of discretely presented interval modules,
that we call the {\em $\delta$-discretely presented interval modules}. \\

\begin{definition}[$\delta$-discretely presented interval module]\label{def:delta_discretely_presented}
	Let $K \subseteq \R^n$ be a rectangle in $ \R^n$,  and let $\rebuttal{\induced{I}}$ be a discretely presented interval module.
	Given $\delta > 0$, we say that $I$ is $\delta$-\emph{discretely presented} in $K$ if: %
	\begin{enumerate}
		\item (Large facets) for each point $x \in \lowerb{I}$ (resp. $\upperb{I}$) there exists, for each facet $F$ containing $x$,
			an $(n-1)$-hypercube $Q^x_F$ of side length $2\delta$ such that $x\in Q^x_F$ and $Q^x_F \subseteq F$; %
		\item (Large holes) if there exists a diagonal line $l$ such that $l\cap I = \varnothing$, then there exists an $n$-hypercube $R$ of side length $\delta$ containing ${\bf 0}$
			such that for any line $l'$ in $l+R$, %
			one has $l' \cap I = \varnothing$;
		\item (Locally small complexity) any $\infty$-ball of radius $\delta$, i.e.,
			any set $B_\delta(x):=\{y\in\R^n : d_\infty(x,y)\le\delta\}$ for some $x\in\R^n$, intersects at most one facet in $\lowerb{I}$ (resp. $\upperb{I}$) of any fixed codirection;
		\item (Compact description) each facet of $I$ has a non-empty intersection with $K$. \\
	\end{enumerate}
\end{definition}

Assumptions 1 and 2 ensures that the parts of $I$ are large enough w.r.t. $\delta$, %
while Assumptions 3 and 4 ensure that surrounding sets of lines can detect at most one facet associated
to a given codirection at a time, and that critical points of $I$ are all included in the rectangle $K$, respectively. \\

\begin{remark}
	One might wonder whether Assumption 2 and Assumption 3 are redundant with Assumption 1.
	In other words, one might wonder whether it is actually possible to define an interval module with large facets and small holes, or with large facets that can
	share the same codirection and lie close to each other at the same time.
	Even though this
	seems to be impossible %
	when $n=2$ (indicating that Assumption 2 and Assumption 3 might indeed be redundant with Assumption 1 in that case),
	it can happen when $n \ge 3$, as Figure~\ref{fig:interval_small_hole} shows.

	\svg{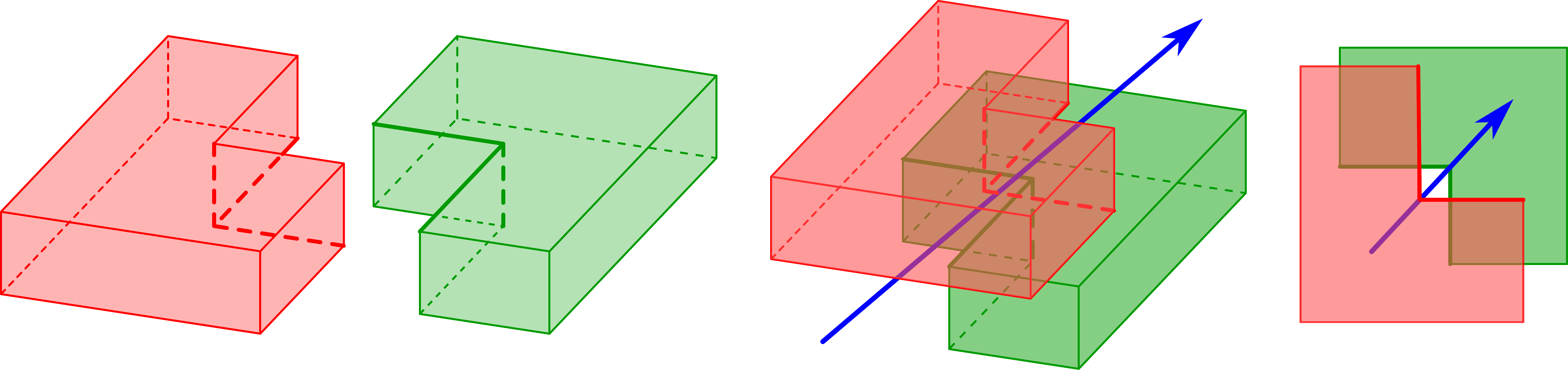}{1}{
		\label{fig:interval_small_hole}
		Example of interval module in dimension $n=3$ with large facets, small holes and some facets with the same codirection close to each other.
		The support of the module can be constructed by taking the (closed) red and (open) green $L$-shaped sets on \textbf{(Left)}, and glue them together
		as shown in \textbf{(Middle)}.
		While arbitrarily large facets can be created using this construction, the resulting interval always contains a small hole and large facets of same codirection that are close
		to each other. Because of this, it is possible to find a (blue) diagonal line that goes through the support without intersecting it, while lines in its surrounding set will
		detect some facets.
		\textbf{(Right)} View of the interval from the top showing the hole and the
		spatially close facets (showed in bold font). %
		This is an example where Assumptions 1 and 4 of Definition~\ref{def:delta_discretely_presented} are satisfied, %
	while Assumptions 2 and 3 are not.}
\end{remark}

The main advantage of $\delta$-discretely presented modules is that they ensure that Algorithm~\ref{algo:label_endpoints} can identify
every single facet with a corresponding label. \\

\begin{lemma}[{Labels are exact}]\label{lemma:label_exact}
	Let $\delta > 0$ and $K$ be a rectangle in $\R^n$. Let
	${\induced I}$ be a $ \delta$-discretely presented interval module in $K$, and let
	$L:=L_\delta(K^{2\delta})$ be the $\delta$-grid of lines of the offset $K^{2\delta}$.
	Then, there exists a bijection between the facets of $I$ and the labels identified by
	Algorithm~\ref{algo:label_endpoints}.
\end{lemma}

\begin{proof}
	We first prove the result for birthpoints and facets of $\lowerb{I}$.

	Let $F$ be a facet of $\lowerb{I}$.
	Let $l_F\in L$ be a diagonal line intersecting $F$, and $b_F\in\R^n$ be the associated birthpoint.
	By Definition~\ref{def:delta_discretely_presented}, item (1), there exists an $(n-1)$-hypercube $Q_F^{b_F}\subseteq F$ of side length $2\delta$
	such that $b_F\in Q_F^{b_F}$. %
	This ensures that for any dimension $i$ that is not in the codirection: $i\in \llbracket 1,n\rrbracket\backslash\codir{F}$, one has
	either $b_F + \delta e_i\in Q_F^{b_F}$ or $b_F - \delta e_i \in Q_F^{b_F}$.
	Since $L$ is the $\delta$-grid of lines associated to $K^{2\delta}$, and since $Q_F^{b_F}$ is an $(n-1)$-hypercube,
	there exists a line $l_0\in L$ such that $l_F$ belongs to the surrounding set $L_{l_0}$,
	and such that the birthpoints corresponding to the lines in $L_{l_0}$ are all in $Q_F^{b_F}$. %
	This means that $\codir{F}$ is detected as a label of $b_F$ by Algorithm \ref{algo:label_endpoints}. %

	Reciprocally, assume there exists a line $l_0\in L$ such that all birthpoints associated to the lines in the
	surrounding set $L_{l_0}$
	share a coordinate along dimension $i\in\llbracket 1,n \rrbracket$, so that $i$ is a label detected by Algorithm \ref{algo:label_endpoints}.
	Then, the set of birthpoints $B_{L_{l_0}}$ has a minimal element, and thus its convex hull $\conv{B_{L_{l_0}}}$ is in $\lowerb{I}$.
	Since $\conv{B_{L_{l_0}}}$ is an $(n-1)$-hypercube of codirection $i$, it must be associated to a facet of $\lowerb{I}$ of codirection $i$ as well.

	The proof extends straightforwardly for deathpoints. \\
\end{proof}

Now that we have proved that all facets can be detected with $\delta$-grids of lines and $\delta$-discretely presented modules, we can state our following result,
which claims that it is possible to {\em exactly} recover the underlying module under the same assumptions. \\

\begin{lemma}[Exact recovery of intervals]\label{lemma:exact_recov_int}
	Let $\delta > 0$ and $K=\rectangle{\alpha}{\beta}$ be a rectangle in $\R^n$, where $\alpha \leqn \beta$. Let
	${k^{I}}$ be a $ \delta$-discretely presented interval module in $K$, and let $L:=L_\delta(K^{2\delta})$ be the $\delta$-grid of lines of the offset $K^{2\delta}$.
	Let {$B=L\cap I$}, and $\birthLcp{B}$ and $\deathLcp{B}$ be the $L$-birth and death corners of $I$ computed by Algorithm~\ref{algo:generate_corners}, and let
	$\tilde I=\bigcup_{c \in \birthLcp{B}}\bigcup_{c' \in \deathLcp{B}} R_{c,c'}$ be the approximation computed by Algorithm~\ref{algo:approx_inter}.
	Then,
	one has:
	\begin{equation}
		\disti\left({k^{I}},{k^{\tilde I}}\right)=\distb\left({k^{I}},{k^{\tilde I}}\right)=0.
	\end{equation}

\end{lemma}

\begin{proof}
	As interval modules are characterized by their support, it is enough to show that $ \overline{I} = \overline{\tilde I}$.
	In the following, we thus assume that $I$ is closed in $\overline{\R}^n$.
	We will also use an additional definition. Let $b$ be an infinite corner computed by Algorithm~\ref{algo:generate_corners}. We say that $b'$ is a \emph{pseudo birth corner} for $b$ if:
	\begin{enumerate}
		\item $b'_i=b_i$ for all $i\in\llbracket 1,n \rrbracket \backslash \mathcal J$,
			and for each
			dimension $j \in \mathcal J$, %
			there exists a hyperplane of codirection $j$ intersecting $K$ such that $\bigcap_j H_j \ni b'$.
			The set $\mathcal J$ is called the \emph{codirection} of $b'$ and denoted
			with $\codir{b'}$, and the set $\llbracket 1,n\rrbracket \backslash \mathcal J$ is called the {\em direction} of $b'$ and is denoted with $\dir{b'}$.
		\item there exists a line $l_0 \in L$ such that:
			\begin{enumerate}
				\item $b' \in \conv{L_{l_0}} \cap K^{2 \delta} \backslash K$,
				\item for each line $l \in L_{l_0}$, the endpoint $b_l^I$ is non trivial,
				\item for each dimension $j\in\mathcal J$, %
					there exists $l_j \in L_{l_0}$ such that $b_{l_j}^I\in H_j$.
			\end{enumerate}
	\end{enumerate}

	Note that codirections can be extended to any finite corner (i.e., that is
	potentially not the pseudo corner of an infinite corner) straightforwardly, and that pseudo death corners can be defined by symmetry.
	We now prove Proposition~\ref{prop:exact_recovery}. \\

	\textbf{We first show the inclusion ${\tilde I}\subseteq {I}$.} More specifically, we have to prove that the
	corners computed by Algorithm~\ref{algo:generate_corners}
	all belong to ${I}$. %
	A key argument that we will use several times comes from the following lemma,
	which allows for a local control of the boundary of ${I}$ using the hyperplanes associated to specific corners.

	\begin{lemma}\label{lemma:local_control}
		Let $b$ be a birthpoint (resp. deathpoint) of $I$ in $K^ \delta$,
		and $l_0\in L$ be the line such that $b \in \conv{L_{l_0}}$ (this line exists since $L$ fills $K^{\delta}$).
		Then, one has the following:
		\begin{enumerate}
			\item for any facet $F$ of $\lowerb{I}$ (resp. $\upperb{I}$) containing $b$,
				there exists a line $l_F \in L_{l_0}$ such that $b_{l_F}^I \in F$ (resp. $d_{l_F}^I \in F$).
			\item for any dimension $i$, there exists at most one facet of codirection $i$ intersecting the
				set of birthpoints (resp. deathpoints) $ \left\{ b_l^I : l \in L_{l_0}\right\}$ (resp. $ \left\{ d_l^I : l \in L_{l_0}\right\}$.
				\item let $b'_{L_{l_0}}$ (resp. $d'_{L_{l_0}}$) be the finite corner generated by $L_{l_0}$. Then, one has:
					\begin{align*}
						& \conv{L_{l_0}}\cap \lowerb{I} \cap K^{2\delta} \subseteq \bigcup_{i \in \codir{b'}} \left\{ x \in \R^n: x_i  = b'_i \right\}         \\
						\text{ (resp.           } & \conv{L_{l_0}}\cap \upperb{I} \cap K^{2\delta} \subseteq \bigcup_{i \in \codir{d'}} \left\{ x \in\R^n: x_i  = d'_i \right\}\text{)}. \\
					\end{align*}
			\end{enumerate}
		\end{lemma}

		\begin{proof} We only show the result for birthpoints since the arguments for deathpoints are the same. Let $b\in\lowerb{I}$ be a birthpoint in $K^ \delta$.\\

			\textbf{Proof of (1).} Let $F$ be a facet containing $b$.
			According to Definition~\ref{def:delta_discretely_presented}, item (1), there exists an $(n-1)$-hypercube $Q^b_F$ of side length $2 \delta$
			such that $Q^b_F\subseteq F$ and $b\in Q^b_F$.
			Since $L$ is a grid, there exists a line $l\in L$ with $d_{\infty}(b,l) < \delta$ intersecting $Q^b_F$.
			Now, since $b \in \conv{L_{l_0}}$, one has $d_{\infty}(l \cap H_F,L_{l_0} \cap H_F) < \delta$, where $H_F$ is
			the hyperplane containing $F$; thus, $l \in L_{l_0}$ (the argument is the same than in the proof of Proposition~\ref{prop:local_bound}, first paragraph).\\

			\textbf{Proof of (2).} By Proposition~\ref{prop:local_bound}, item (2),
			the birthpoints associated to lines of $L_{l_0}$ are all contained in a ball of radius $\delta$.
			Thus, the unicity of the facets with given codirection comes straightforwardly from Definition \ref{def:delta_discretely_presented}, item (3). \\

			\textbf{Proof of (3).} Note that the birthpoint $b$ is obviously included in
			the facets of $\lowerb{I}$ that contain it, which is a subset of the facets associated to the birthpoints of the lines in $L_{l_0}$.
			Now, as Lemma~\ref{lemma:label_exact} ensures that the birthpoints
			associated to lines in $L_{l_0}$ are correctly labelled, the corner generated by $L_{l_0}$ must be on the intersection of the facets containing $b$.
			This ensures that:
			\begin{equation*}
				b\in \bigcup_{i \in \mathrm{ codir}(b')} \left\{ x \in\R^n : x_i  = b'_i \right\}.
			\end{equation*}

			Since these arguments do not depend on $b\in\conv{L_{l_0}}$, the result follows.
		\end{proof}

		Now that we have Lemma~\ref{lemma:local_control}, we can prove that finite
		and infinite corners belong to $\supp{I}$. We will prove the results for birth corners, but the arguments for death corners are symmetric.\\

		\textbf{Finite corners.} %
		Let $b$ be a finite birth corner, associated to a set of consecutive lines $L_{l_0}$ for some line $l_0\in L$. %
		By assumption, each birthpoint $b_l^I$, for $l \in L_{l_0}$, is nontrivial; and thus any birthpoint in $\conv{L_{l_0}}$ is nontrivial as well, using
		Definition~\ref{def:delta_discretely_presented}, item (2). Let $l \in \conv{L_{l_0}}$ be the diagonal line passing through $b$.
		\\
		Using Lemma~\ref{lemma:local_control}, one has:
		\begin{equation*}
			b_{l}^I \in \conv{L_{l_0}} \cap \lowerb{I} \cap K^{2\delta} \subseteq \left(\bigcup_{i\in \codir{b}} \left\{ x \in \R^n :x_i=b_i\right\}\right) \cap l = \left\{ b\right\}.
		\end{equation*}
		Thus $b=b_l^I$ and $b \in {I}$.
		\\

		\textbf{Infinite corners.} %
		Let $b$ be an infinite birth corner, and %
		let $b'$ be the minimal (w.r.t. $\leqn$) pseudo birth corner for $b$, which is well defined by construction of $b$ (see Algorithm~\ref{algo:generate_corners}).
		Let $L_{l_0}$ be the associated set of consecutive lines $L_{l_0}$, for some line $l_0 \in L$.
		We will show that, if $j$ is a free coordinate of $b'$, i.e., if $j\in\dir{b'}$, then $b'_j < \alpha_j$ (recall that $K$ is the rectangle $\rectangle{\alpha}{\beta}$).
		The reason we want to prove such inequalities is that they directly lead to
		the result. Indeed, if $b'_j < \alpha_j$ for any $j\in\dir{b'}$, then $b' - t \sum_{j\in\dir{b'}} e_j$ belongs to $\lowerb{I}$ for any $t>0$,
		since otherwise the line $\{b' - t \sum_{j\in\dir{b'}} e_j : t >0\}$ would have to intersect a facet $F\subseteq \lowerb{I}$ of codirection $j$ for some $j\in \dir{b'}$,
		which would not intersect $K$, contradicting Definition~\ref{def:delta_discretely_presented}, item (4).

		Let $j\in \dir{b'}$ be a free coordinate. By contradiction, assume that %
		$b'_j \ge \alpha_j$, and let $b^j$ denote the pseudo birth corner generated by $L_{l_0- \delta e_j}$.
		In particular, this means that, for any $l\in L_{l_0}$,
		$l- \delta e_j \in L$ and $L_{l- \delta e_j} \subseteq L$ since $L$ fills $K^{ \delta}$.
		Now, if for every line $l \in L_{l_0}$ such that $l=l_0+\ora v$ with $\ora v_j=0$, one has that $b_l^I$ and $b_{l- \delta e_j}^I$
		are on the same facets, then one has $b_{l- \delta e_j}^I = b_l^I - \delta e_j$, and
		the pseudo corner $b^j$ %
		is equal to $b' - \delta e_j$ by construction, as per Algorithm \ref{algo:generate_corners}.
		Moreover, one has $b^j = b' - \delta e_j \leqn b'$, contradicting the fact that $b'$ is minimal.
		Hence, there is at least one line $l\in L_{l_0}$, $l=l_0+\ora v$ with $\ora v_j=0$, such that $b_l^I$ and $b_{l- \delta e_j}^I$ are not
		on the same facets, in other words, there exists a facet
		$F_j$ of $\lowerb{I}$ of codirection $j$ that intersects the (half-open) segment $[b_l^I - \delta e_j, b_l^I)$.
		In order to locate that facet more precisely, we will prove the following lemma.

		\begin{lemma}\label{lemma:decreasing_coord}
			For any $i\in \llbracket 1,n\rrbracket$ and $s,t\in\R$ such that $s<t$, one has $(b^I_{l-t e_i})_i \leq (b^I_{l-s e_i})_i$.
		\end{lemma}

		\begin{proof}
			Without loss of generality, assume $s=0$.
			Since $b^I_l - t e_i \in l-te_i$, it follows that $b^I_l-te_i$ and $b^I_{l-te_i}$ are comparable. Moreover, one must have $b^I_l-te_i \leqn b^I_{l-te_i}$, otherwise one would have
			$b^I_l \sgn b^I_l-te_i \sgn b^I_{l-te_i}$, contradicting Lemma~\ref{rem:flat_rect}.
			If the points are equal, i.e., $b^I_l-te_i = b^I_{l-te_i}$, then one has $(b^I_l)_i \geq (b^I_{l-te_i})_i$. Otherwise, if $b^I_l-te_i \sln b^I_{l-te_i}$, then:
			\begin{equation*}
				\forall k\neq i, \, (b^I_{l-te_i})_k > (b^I_{l})_k.
			\end{equation*}
			Moreover, since $b^I_l$ and $b^I_{l-te_i}$ cannot be comparable as per Lemma~\ref{rem:flat_rect}
			one must have $(b^I_{l -te_i})_i\le (b^I_l)_i$. %

		\end{proof}
		Let $H_j = \left\{ x\in\R^n : x_j = c_j \right\}$ be the hyperplane associated to $F_j$.
		Then, by Lemma~\ref{lemma:decreasing_coord}, one has:
		\begin{equation*}
			(b_{l- \delta e_j}^I)_j \le c_j < (b_l^I)_j.
		\end{equation*}

		Since the lines $l$ and $l-\delta e_j$ both belong to the surrounding set $L_{l_0-\delta e_j}$,
		it follows from Lemmas~\ref{lemma:label_exact}, and~\ref{lemma:local_control}, item (3), that
		$\codir{b^j} \supseteq \codir{b'}\cup \left\{ j \right\}$. Moreover, since the facets of $\lowerb{I}$ associated to $\codir{b^j}$ are unique in a $\delta$-ball around $b^j$,
		as per Definition~\ref{def:delta_discretely_presented}, item (3), they all have a unique associated value $c_i$ (corresponding to their associated hyperplanes).

		Finally, we will show that $b^j\leqn b'$.
		Let $i\in\llbracket 1,n \rrbracket$ be an arbitrary dimension.
		\begin{itemize}
			\item If $i \in \codir{b'}$, then $b^j_i = b'_i$.
			\item If $i \in \codir{b^j} \backslash \codir{b'}$, then $b^j_i \in \left\{
				c_i, \min_{l \in L_{l_0-\delta e_j}}(b_l^I)_i\right\} \le \min_{l \in L_{l_0}}(b_l^I)_i = b'_i$, with a strict inequality for $i=j$.
			\item If $i \in \dir{b^j} \subseteq \dir{b'}$, then
				$b^j_i =  \min_{l \in L_{l_0- \delta e_j}}(b_l^I)_i \le  \min_{l \in L_{l_0}}(b_l^I)_i = b'_i$.
		\end{itemize}
		Hence, one always has $b^j_i\leq b'_i$, and thus $b^j \sln b'$, which contradicts the fact that $b'$ is minimal. Thus, one must have $b'_j < \alpha_j$.
		\\

		\textbf{We now show that $ {I} \subseteq {\tilde I}$.} %
		Let $x \in {I}$. We will show that there exists a birth corner $c$ such that $c \leqn x$.
		Let $\mathcal H$ be the family of hyperplanes associated to the facets of $\lowerb{I}$. %
		The corner $c$ will be defined as the limit of a sequence of points $\{x^{(k)}\}_{k\in\mathbb{N}^*}$in $ \overline{\R}^n$, defined by induction with:
		\begin{enumerate}
			\item $x^{(1)} = \inf \left\{ x - t\cdot \mathbf{1} : t\ge 0\right\} \cap \supp{I}$. Then, one has the two following possibilities:
				\begin{itemize}
					\item either $x^{(1)} = -\boldsymbol \infty$,
						and
						we let $c:= x^{(1)}$.
					\item or there exists a maximal subset of hyperplanes $\mathcal H^1 \subset \mathcal H$, $\mathcal H^1 \neq \varnothing$,
						such that $x^{(1)} \in \cap_{H \in \mathcal H^1}H=:H_1$.
						Let $\mathcal  J^1 \subseteq \llbracket 1,n \rrbracket$ be the set of free coordinates in $H_1$, i.e., those dimensions such that $j\in \mathcal J^1 \Longleftrightarrow x^{(1)}- e_j \in H_1$.
				\end{itemize}
			\item $x^{(2)} = \inf \left\{ x^{(1)} - t \cdot \sum_{j\in \mathcal J^1}e_j %
				: t\ge 0\right\} \cap \supp{I}$. %
				Then,  one has the two following possibilities:
				\begin{itemize}
					\item either $x^{(2)}$ is at infinity in $H_1$, i.e., $x^{(2)}_j = -\infty$ if $j \in \mathcal J^1$ and $x^{(2)}_j = x^{(1)}_j$ otherwise, and we let $c:=x^{(2)}$.
					\item or there exists a maximal subset of hyperplanes $\mathcal H^2\supsetneq \mathcal H^1$ %
						such that $x^{(2)} \in \cap_{H \in \mathcal H^2}H=:H_2$. %
						Let $\mathcal J^2 \subseteq \llbracket 1,n \rrbracket$ be the set of free coordinates in $H_2$, i.e., those dimensions such that $j\in \mathcal J^2 \Longleftrightarrow x^{(2)}- e_j \in H_2$.
				\end{itemize}
			\item For $k\ge 3$,  $x^{(k+1)} = \inf \left\{ x^{(k)} - t \cdot \sum_{j\in \mathcal J^k}e_j : t\ge 0\right\}\cap \supp{I}$. %
				Then,  one has the two following possibilities:
				\begin{itemize}
					\item either $x^{(k+1)}$ is at infinity in $H_k$, i.e., $x^{(k+1)}_j = -\infty$ if $j \in \mathcal J^k$ and $x^{(k+1)}_j = x^{(k)}_j$ otherwise, and we let $c:=x^{(k+1)}$. %
					\item or there exists a maximal subset of hyperplanes $\mathcal H^{k+1}\supsetneq \mathcal H^k$ %
						such that $x^{(k+1)} \in \cap_{H \in \mathcal H^{k+1}}H=:H_{k+1}$. %
						Let $\mathcal J^{k+1} \subseteq \llbracket 1,n \rrbracket$ be the set of
						free coordinates in $H_{k+1}$, i.e., those dimensions such that $j\in \mathcal J^{k+1} \Longleftrightarrow x^{(k+1)}- e_j \in H_{k+1}$.
				\end{itemize}
		\end{enumerate}

		If this sequence stops at step one, i.e., $c=x^{(1)}=-\boldsymbol \infty$,
		then every birthpoint of $I$ is at $-\boldsymbol\infty$, the only birth corner is $c=-\boldsymbol\infty$, and one trivially has
		$c\leqn x$. %
		Hence,
		we assume in the following that $c$ is obtained after at least one iteration of the sequence.
		Note that this sequence of points has length at most $n$.
		Let $c^-$ and $c$ be the penultimate and last elements of the sequence respectively,
		and let $\mathcal J^-$ be the set of free coordinates associated to $c^-$. %
		By construction, one has:
		\begin{equation*}
			c \leqn c^- \leqn \cdots\leqn x^{(2)}\leqn x^{(1)} \leqn x.
		\end{equation*}
		We now show that $c$ is indeed a birth corner.
		If $c$ is finite,
		then it must belong to the intersection of $n$ hyperplanes, %
		and it is thus a finite birth corner.
		Hence, we assume now that $c$ is not finite. We will construct a minimal pseudo birth corner from $c^-$,
		and show that $c$ is its associated infinite birth corner. %
		We will consider two different cases, depending on whether $c^-$ is close to $K=\rectangle{\alpha}{\beta}$ or not.
		If $c^- \in K^\delta$, the filling property of $L$ and the size of the facets of $\lowerb{I}$ ensure that $c^-$ is itself %
		a minimal pseudo birth corner, associated to $c$, which is thus an infinite birth corner. %
		If $c^- \notin K^{\delta}$, then let $ \overrightarrow v \in \mathbb R^n$ be a vector that pushes back $c^-$ into $K^\delta$, i.e.,
		such that, for any dimension $i\in \mathcal J^- $, one has: %
		$$\alpha_i - \delta \leq (c^-+\overrightarrow v)_i < \alpha_i,
		$$
		and $ \ora v_i = 0$ if $i \notin \mathcal J^-$.
		Let $S$ be the segment $[c^-, c^- + \ora v]$.
		We have the two following cases:
		\begin{enumerate}
			\item Assume $S \subseteq \lowerb{I}$. Then $c^- + \ora v \in \supp{I} \cap K^\delta$, and there exists a line $l \in L$
				such that $c^-+\ora v \in \conv{L_l}$. Let $c^l$ be the pseudo birth corner associated to $L_l$. %
				Since one has $c^l_j < \alpha_j$ for any dimension $j \in \mathcal J^-$, %
				it follows that $\mathcal J^- \subseteq \dir{c^l}$. %
				Furthermore, since $c^-+\ora v$ %
				belongs to the same facets than $c$ and $c^-$, and since  $c^-+\ora v \in \conv{L_l}$
				one has $ \codir{c^l} \supseteq \codir{c}$ and $\dir{c} = \mathcal J^- $. %
				Thus, $c$ is an infinite birth
				corner associated to the minimal pseudo birth corner $c^l$.

			\item Assume $S \not\subseteq \lowerb{I}$. In that case, there must be a facet of codirection $j$, for some $j \in \mathcal J^-$, that intersects $S$.
				Since
				one has $c^-_j \leq (c^-+\ora v)_j  < \alpha_j$ for any $j \in \mathcal J^-$, this means that the facet would not intersect $K$,
				which yields to a contradiction as per Definition~\ref{def:delta_discretely_presented}, item (4).
		\end{enumerate}
		This concludes that $ {I} \subseteq {\tilde I}$, and the equality between these supports holds. %
	\end{proof}

	\cref{lemma:exact_recov_int} extends to the following proposition, whose proof is immediate from the definition of {induced} matchings. \\ %

	\begin{proposition}[Exact recovery]\label{prop:exact_recovery}%
		Let $\Mbb$ be a {f.p.} interval decomposable $n$-parameter persistence module.
		Let $K$ be a rectangle in $\R^n$ %
		{that \compacityassumption{} $\Mbb$},
		and $L:=L_\delta(K^{2\delta})$ be the $\delta$-grid of lines of the offset $K^{2\delta}$.
		Assume that all interval summands of $\Mbb$ are $\delta$-discretely presented, and
		let ${\tilde\Mbb^{\mMMA{}}_\delta :=\mMMA{}(\M,L,\matching)}$, where $\matching$ %
		is a matching function that commutes with the induced matching function $\sigma_\Mbb$.
		Then, one has:
		$$\disti(\M,\tilde \M^{\mMMA{}}_\delta)=\distb(\M,\tilde \M^{\mMMA{}}_\delta)=0.$$
	\end{proposition}

	Note that f.p. interval decomposable modules are always made of
	$\delta$-discretely presented interval summands, for small enough yet positive
	$\delta$ (one can take for instance the smallest distance $\delta_{\rm exact} > 0$ between two distinct graded Betti numbers).

	One might wonder whether the usual distances between barcodes,
	such as the bottleneck or Wasserstein distances,
	could be used to define matching functions that commute with induced matching functions.
	Indeed, a major advantage of, e.g., Wasserstein distances,
	is that their associated matching functions are usually unique.
	However, when the space $\delta$ between two lines is too large,
	the matching functions induced by Wasserstein distances can still fail to be induced, %
	as shown in Figure~\ref{fig:wasserstein-not-working}. In the next section, we discuss how to design such matching functions from compatible matching functions.

	\begin{figure}[h]
		\centering
		\includegraphics[width=0.4\textwidth]{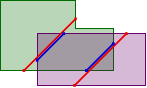}
		\caption{\label{fig:wasserstein-not-working}
			{Example of interval decomposable module with two interval summands (green and purple), and its barcodes
				along two lines (here the two couples of red-blue bars).
				Any matching function induced by, %
				e.g., Wasserstein distances between the barcodes,
				will match the first red bar with the second red bar and the first blue bar with the second blue bar;
			however, this matching is not induced.} %
		}
	\end{figure}

\section{\rebuttal{Finding compatible and induced matching functions}}
\label{sec:matching_mma}

\rebuttal{In this section, we discuss how to design matching functions that are compatible or that commute with induced matching functions,
	so as to satisfy the assumptions of~\cref{prop:tildeIcandidate}, \cref{prop:approx} and \cref{prop:exact_recovery}. In particular, in~\cref{app:compat_exact} we restrict to interval decomposable modules, and we show that compatible and usual matching functions always commute with induced matching functions for small enough $\delta$ and generic assumptions. Then, in~\cref{subsec:vineyards}, we show that the vineyards algorithm induces compatible matching functions when $n=2$.
	}

\subsection{Induced matching functions for interval decomposable modules}\label{app:compat_exact}

In this section, we show, given some interval decomposable module $\Mbb$, that compatible matching functions commute with the induced matching function $\matching_{\M}$  under some specific conditions. In the rest of this section, given a matching function $\matching$ that commutes with $\matching_\M$, we will also call $\matching$ \rebuttal{{\em induced}} for the sake of simplicity. Before stating the main result, we first show a technical lemma that relates the locations of endpoints of compatible bars with each other. \\

\begin{lemma}\label{lemma:consecutive_segment}
	Let $l_1$ and $l_2$ be two $\delta$-consecutive lines, and let $[b_1,d_1):=\barcode{\restr{\rebuttal{k^{I}}}{l_1}}$ be the bar of a \rebuttal{f.p.} interval module along $l_1$. %
	Let $[b_2,d_2)$ be a bar along $l_2$ that is compatible with $[b_1,d_1)$. Then, $d_2$ (resp. $b_2$) is included in a segment of size $ \delta$ in $l_2$. %
\end{lemma}
\begin{proof}
	Applying Lemma~\ref{lemma:compatible_close},
	one has:
	\begin{equation*}
		d_2 \in C :=
		\big[B_{ \delta }(d_1)\cap l_2\big]
		\backslash
		\big[ \left\{ z\in \mathbb R^n : z>d_1 \right\} \cup \left\{ z\in \mathbb R^n : z<d_1 \right\}\big].
	\end{equation*}
	Since $C$ is a nonempty, totally ordered set, we can define $y := \min C$. By construction, there exists a dimension $i$ such that
	$y_i\ge (d_1)_i$, and thus $C$ must be included in the segment $[y, y+\delta\cdot\mathbf{1}]$ along $l_2$. %

	The proof applies straightforwardly to $b_2$ by symmetry.
\end{proof}

Since bars that are matched under an \rebuttal{induced} matching function are always compatible,
one way to construct an \rebuttal{induced} matching function between two barcodes is therefore to isolate, among all possible matching functions, the ones %
such that matched bars are compatible. If this family contains a single element, it must be the \rebuttal{induced} matching we are looking for.
This typically happens for interval decomposable multi-parameter persistence modules whose summands are sufficiently separared, as we show in the proposition below. \\

\begin{proposition}\label{prop:general_exact_matching}
	Let $\M= \bigoplus_{I \in \mathcal I}\rebuttal{k^{I}}$ be a \rebuttal{f.p.} interval decomposable $n$-parameter persistence module.
	Let $\delta > 0$, and $\rebuttal{k^{I}},\rebuttal{k^{I'}}$ be two interval summands in the decomposition of $\M$. %
	Assume that the two following properties are satisfied:
	\begin{enumerate}
		\item Let $l \subset \R^n$  be a diagonal line such that ${I}\cap l\neq \varnothing$ and ${I'}\cap l\neq \varnothing$.

		      Then, one has either $\norm{ b_l^I - b_l^{I'}}_\infty > \delta$ or $ \norm{d_l^I - d_l^{I'}}_\infty > \delta$.
		      In other words, the endpoints of the bar in $\barcode{\rebuttal{ \restr{\induced{I}}{l} }}$ and of the bar in $\barcode{\rebuttal{\restr{\induced{I'}}{l}}}$
		      are at distance at least $\delta$.

		\item %
		      The bars of length at most $2\delta$ in $I$ and $I'$ are at distance at least $\delta$, i.e.,
		      if we let:
		      $$S^I:=\left\{ l:  l\cap {I} \neq \varnothing,\norm{ b_l^{I} - d_l^{I}}_{\infty} \le 2 \delta\right\},$$
		      (and similarly for $I'$),
		      one has

		      $d_{\infty}
			      \left(
			      S^I,
			      S^{I'}
			      \right)
			      > \delta/2.$

		      In other words, a small bar in $I$ cannot be too close to a small bar in $I'$.
	\end{enumerate}
	Then, the matching function $\matching_{\rm comp}$,
	induced by matching bars that are compatible together, is well-defined and \rebuttal{induced from $\Mbb$}. \\
\end{proposition}

\rebuttal{Note that, upon using chunk reduction~\cite{fugacciChunkReductionMultiParameter2019} and infinitesimal perturbations, or whenever the graded Betti numbers of $\Mbb$ are independent and identically distributed from a non-singular distribution, it is always possible to ensure that Assumptions (1) and (2) are satisfied for a given f.p. interval decomposable module $\Mbb$ and small enough $\delta$ (see also the paragraph after~\cref{prop:exact_recovery}).}
See Figure~\ref{fig:compatibility} for an illustration of Assumptions (1) and (2).

\begin{figure}
	\centering
	\includesvg[width=.9\textwidth]{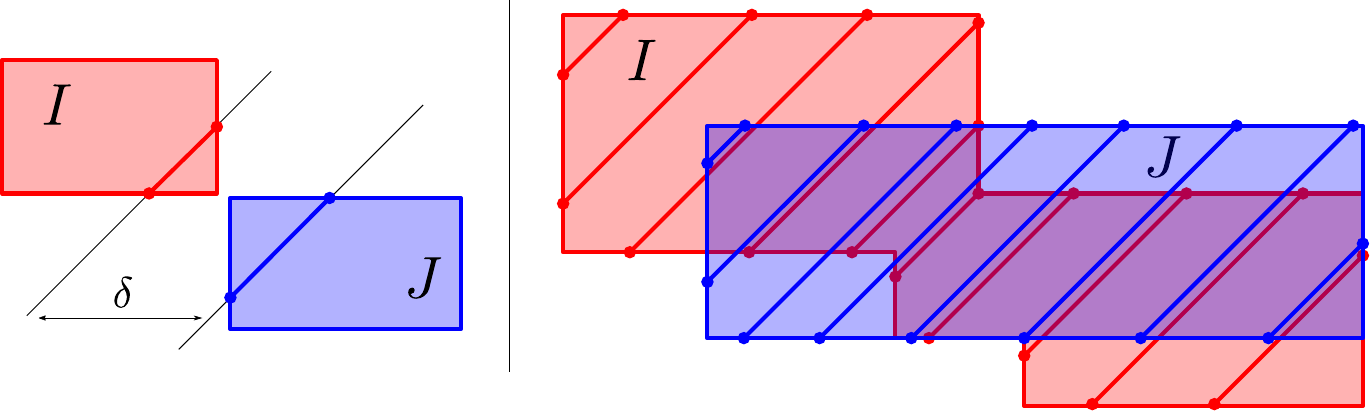}
	\caption{
		\textbf{(Left)} Example of module whose interval summands do not satisfy Assumption (2).
		\textbf{(Right)} Example of module whose interval summands do satisfy Assumptions (1) and (2).
		Bars corresponding to consecutive lines can only be matched if they are compatible, which, in this figure, means that they have the same color, i.e., that they
		are associated to the same interval summand.
	}\label{fig:compatibility}
\end{figure}

\begin{proof}
	Let $\induced I$ and $\induced{I'}$ be two interval summands in the decomposition of $\M$.
	Let $l_1$ and $l_2$ be two $\delta$-consecutive lines of $L$, and let $b:=\barcode{\restr{\induced{I}}{l_1}}$ be the bar corresponding to  $I$  along $l_1$.
	We will show that $\matching_{\rm comp}$ must match $b$ to either $b':=\barcode{\restr{\induced{I}}{l_2}}$ if ${I}\cap l_2\neq\varnothing$, or the empty set if ${I}\cap l_2 = \varnothing$.

	\begin{itemize}
		\item[$\bullet$] If $ \supp{I}\cap l_2 = \varnothing$, then
		      by Lemma \ref{lemma:compatible_close},
		      the length of $b$ is at most $\delta$, i.e., $\norm{b^I_{l_1}-d^I_{l_1}}_\infty \le \delta$.

		      It is thus compatible with the empty set.
		      Now,
		      since $d_\infty(l_1,l_2)=\delta/2$ and since $l_1\in S^I$,
		      Assumption (2) ensures that the bar $b'':= \barcode{\restr{\induced{I'}}{l_2}}$ (if it exists) must be of length at least $2\delta$. %
		      In particular, it is not compatible with $b$, hence $\matching_{\rm comp}$ cannot match $b$ to $b''$, and must match $b$ to the empty set.

		\item[$\bullet$] If $ \supp{I} \cap l_2 \neq \varnothing$, then
		      the bar $b'=[b^I_{l_2},d^I_{l_2})$ in $\barcode{\restr{\induced{I}}{l_2}}$

		      is compatible with $b$, as per Lemma~\ref{rem:flat_rect}.
		      According to Lemma~\ref{lemma:consecutive_segment}, it follows that the birthpoint and deathpoint of any
		      bar along $l_2$ that is compatible to $l_1$ must belong to segments $s_b, s_d$ of length $\delta$ that contain
		      $b^I_{l_2}$ and $d^I_{l_2}$ respectively.
		      Let $b'':=[b^{I'}_{l_2}, d^{I'}_{l_2})$ be the bar in $\barcode{\restr{\induced{I'}}{l_2}}$ (if it exists).

		      According to Assumption (1), we either have $ \norm{ b_{l_2}^I - b_{l_2}^{I'}}_{\infty}>\delta$ or $ \norm{ d_{l_2}^{I'} - d_{l_2}^{I'}}_{\infty} > \delta$.
		      In particular this means that either $b^{I'}_{l_2}\not\in s_b$ or $d^{I'}_{l_2}\not\in s_d$.
		      Hence $b''$ is not compatible with $b$, and $\matching_{\rm comp}$ must match $b$ to $b'$.

	\end{itemize}
	In both cases, $\matching_{\rm comp}$ is well-defined and \rebuttal{induced from $\Mbb$}.
\end{proof}

\rebuttal{One can check that the proof of \cref{prop:general_exact_matching} extends easily to the matching functions associated to the Wasserstein distances and the vineyards algorithm. Indeed, their associated matching functions are unique when $\delta$ becomes small enough, and thus must correspond to the only compatible matching $\matching_{\rm comp}$ identified in~\cref{prop:general_exact_matching}.}

\subsection{The vineyards algorithm for general $2$-parameter modules}\label{subsec:vineyards}
In this section, we show that the matching function associated to the vineyards
algorithm for simplicial complexes is compatible, for small enough $\delta$ and $n=2$. %
This section is quite technical and can be skipped by readers who are most interested in the general exposition.
Since vineyards are heavily based on simplicial homology,
we first recall the basics of persistent homology from simplicial complexes in Section~\ref{subsubsec:simplicial_homology}.
Then, we provide an analysis of the
vineyards algorithm in Section~\ref{subsubsec:vineyard}.

\subsubsection{Persistent homology of simplicial complexes}\label{subsubsec:simplicial_homology}

We assume in the following that the reader is familiar with simplicial complexes, boundary operators and homology groups, and we refer
the interested reader to~\cite[Chapter 1]{munkresElementsAlgebraicTopology1984} for a thorough treatment of these notions.
The first important definition is the one of {\em filtered simplicial chain complexes}. \\

\begin{definition}
	Let $S$ be a simplicial complex, and $f\colon S \to \R$ be a {\em filtration function}, i.e., $f$ satisfies
	$f(\sigma) \le f( \tau)$ when $\sigma \subseteq \tau$.
	Then, the {\em filtered simplicial chain complex} $(S,f)$ is defined as
	$(S,f) = ((C_t)_{t\in \R}, \iota)$, where:
	\begin{enumerate}
		\item $C_t = \left< \sigma_0,\dots, \sigma_i\right>$ is the vector space over a field $k$ whose basis elements are the simplices that have filtration values smaller than $t$, i.e.,
			$ \left\{ \sigma_0,\dots, \sigma_i\right\} = \left\{ \sigma\in S: f(\sigma) \le t\right\}$, and
		\item for any $s\le t$, the map $\iota = \iota_s^t \colon C_s \hookrightarrow C_t$ is the canonical injection. \\
	\end{enumerate}
\end{definition}

Note that $f$ can be used to define an order on the simplices of $S = \left\{ \sigma_i\right\}_{i=0}^N$, by using the ordering induced by the filtration values.
In other words, we assume in the following that $f(\sigma_0) \le f(\sigma_1) \le \dots \le f(\sigma_N)$.
We also slightly abuse notations and define $C_i := \left< \sigma_0, \dots,
\sigma_i \right>$ for any $i\in\llbracket 0,N \rrbracket$, and
\begin{equation}
	(S,f) = \left(C_0 \overset {\iota_0} \hookrightarrow C_1 \overset {\iota_1}\hookrightarrow \dots \overset {\iota_{N-1}}\hookrightarrow C_N = \left<S\right>\right).
\end{equation}

Then, applying the homology functor $H_{*}$ on this filtered simplicial chain
complex yields the following (single-parameter) persistence module:
\begin{equation*}
	H_*(S,f) = 0
	\rightarrow
	H_*(C_0)
	\rightarrow
	H_*(C_1)
	\rightarrow
	\dots
	\rightarrow
	H_* (C_N).
\end{equation*}

An important theorem of single-parameter persistent homology states that, up to a change of basis, it is possible to pair some chains together in order
to define the so-called {\em one-dimensional persistence barcode} associated to the filtered simplicial chain complex.  \\

\begin{theorem}[Persistence pairing, {\cite[Theorem 2.6]{desilvaDualitiesPersistentCo2011}}]
	\label{th:pers_pairing}
	Given a filtered simplicial chain complex $(S,f) %
	= C_1\hookrightarrow C_2\hookrightarrow\dots\hookrightarrow C_N$ and associated persistence module $H_*(S,f)$,
	there exists a partition $\llbracket 1,N\rrbracket = E\sqcup B\sqcup D$,
	a bijective map $\pairing:D\rightarrow B$, and a new basis $\hat\sigma_1,\dots,\hat\sigma_N$ of $C$, called {\em reduced basis}, such that:
	\begin{enumerate}
		\item $C_i=\langle \hat\sigma_1,\dots,\hat\sigma_i\rangle$,
		\item $\partial\hat\sigma_e=0$ for any $e\in E$,
		\item for any $d\in D$, one has $\partial\hat\sigma_{\pairing(d)}=0$, and
			$\partial\hat\sigma_{d}$ is equal to $\hat\sigma_{\pairing(d)}$
			{\em up to simplification}, i.e.,
			there exists a set of indices $\reduced(d)$ such that
			$(i)$ $j<\pairing(d)\leq d$ for any $j\in \reduced(d)$, and
			$(ii)$ $\partial\hat\sigma_{d}=\hat\sigma_{\pairing(d)} + \sum_{j \in \reduced(d)}\hat\sigma_j$.
	\end{enumerate}
	In particular, the chains $\{\hat\sigma_j:j\in E \cap \llbracket 1,i\rrbracket\} \cup \{\hat\sigma_j:j\in B \cap \llbracket 1,i\rrbracket \text{ and } \exists d > i\text{ s.t. }\pairing(d)=j\}$
	form a basis of the simplicial homology groups $H_*(C_i)$.
	Moreover, the chains $\{\hat\sigma_j:j\in B\sqcup E\}$ are called {\em positive chains} while the chains $\{\hat\sigma_j:j\in D\}$ are called {\em negative chains}.

	The multiset of bars $\mathcal
	B(f):=\{[f(\hat\sigma_b),f(\hat\sigma_d)]:b=\pairing(d)\}\cup \{[f(\hat\sigma_e),+\infty):e\in E\}$ is called the {\em persistence barcode} of the filtered simplicial chain complex $(S,f)$
	and of the single-parameter persistence module $H_*(S,f)$. \\
\end{theorem}

Note that while
the reduced basis $\{\hat\sigma_1,\dots,\hat\sigma_N\}$ does not need to be unique,
the pairing map $\pairing$ is actually independent of that reduced basis, see \cite[VII.1, Pairing Lemma]{edelsbrunnerComputationalTopologyIntroduction2010}.

\subsubsection{Vineyards algorithm and matching}\label{subsubsec:vineyard}

The vineyards algorithm \cite{cohen-steinerVinesVineyardsUpdating2006} is a
method that allows to find reduced chain bases for filtered simplicial complexes whose simplex orderings only differ by a single transposition of consecutive simplices,
that we denote by $(i\, i+1)$.
This algorithm was later generalized to the setup of zigzag persistence modules in \cite{mariaZigzagPersistenceReflections2015}. This is the setup that we use
in our context as follows.
We start with some $n$-parameter persistence module $\Mbb$, such that there
exists a finite dimensional, generic, $n$-filtered one-critical simplicial chain complex $(\simpcomp,F)$, satisfying $\Mbb=H_*(F)$.\footnote{We drop the dependence on $\simpcomp$ for simplicity.}
We also fix an order $(\sigma_1, \dots, \sigma_N)$ on the simplices of $\simpcomp$, and use the following notation (in this section only):
\begin{equation}\label{eq:vineyards_x_def}
	(x_1, \dots, x_N):= \left( F(\sigma_1), \dots, F(\sigma_N) \right), \quad \textnormal{and }\quad   (x_1^l,\dots,x_N^l) :=\left( \restr{F}{l}(\sigma_1), \dots, \restr{F}{l}(\sigma_N) \right),
\end{equation}
the filtration values (in $\R^n$) of the simplices of $F$ and $\restr F l$ respectively, for any given positive line $l$.
Note that we always have, for any line $l$:
\begin{equation}\label{eq:vineyards_line_proj}
	\forall 1\le j\le N,\quad x_j \leqn x_j^l\quad \textnormal{and} \quad\exists 1\le k \le n, \quad (x_j)_k = (x_j^l)_k.
\end{equation}
Finally, we consider the map $\ordonly \colon \restr F l \mapsto \ord{\restr F l }\in\mathfrak S_N$
that gives the partial order of the simplices of $F$ according to the
filtration $\restr F l$, completed to a total order by the initial order if necessary, i.e.,
$\ord{\restr F l}$ is the only permutation $\gamma\in \mathfrak S_N$  satisfying
\begin{equation}\label{eq:vineyards_ord_def}
	\forall 1\le i<j\le N,
	\textnormal{ we have }
	\begin{cases}
		F \left(\sigma_{\gamma_i}\right)
		\leqn
		F  \left(\sigma_{\gamma_j}\right)
		& \textnormal{ if they are comparable, and}                     \\
		\restr F l \left(\sigma_{\gamma_i}\right)
		\leqn
		\restr F l \left(\sigma_{\gamma_j}\right)
		& \textnormal{  otherwise,}
		\textnormal{ with }
		\\
		\gamma_i
		< \gamma_j
		& \textnormal{ if } \restr{F}{l} \left(\sigma_{\gamma_i}\right)
		=
		\restr F l \left(\sigma_{\gamma_j}\right).
	\end{cases}
\end{equation}
\def\vine{\ensuremath{\mathrm{vine}_i}}
\begin{proposition}\label{prop:vine_update_comp}
	Consider $l_1,l_2$ two consecutive diagonal lines such that
	for any diagonal line $l_1 \le l \le l_2 $, we have either
	$\gamma = \gamma_1$ or $\gamma = \gamma_2$, where
	$\gamma :=\ord{\restr F l}, \, \gamma_1 := \ord{\restr F {l_1}} $ and $\gamma_2:=\ord{\restr F {l_2}}$.
	Then, either $\gamma_1 = \gamma_2$ or there exists
	an integer $1\le i \le n$ such that $\gamma_1 = \gamma_2 \circ (i \, i+1)$,
	and the $i$th vineyard update
	induces a compatible matching between $\barcode{\restr{\Mbb}{l_1}}$ and
	$\barcode{\restr{\Mbb}{l_2}}$. \\
\end{proposition}

\begin{remark}[More details about the vineyard matching]
	More specifically,
	the matching function given by the vineyards algorithm corresponds to the following.
	If $(\tau_1, \dots, \tau_N)$ is a reduced basis on the first line,
	then this algorithm provides an updated basis
	$(\vine(\tau_1),\dots, \vine(\tau_N))$ that is reduced on the second line.
	The matching function is then given by the following relation:
	if $b$ (resp. $d$) is the birthpoint (resp. deathpoint) of a bar in the barcode induced by $\restr \Mbb {l_1}$, and generated by some $\tau_j$ (resp. $\tau_k$),
	then the endpoint $b$ is matched to the endpoint generated by $\vine(\tau_j)$ (resp. $\vine(\tau_k)$).

	Note that a non-trivial bar $\left[\restr F {l_1} (\tau_j), \restr F {l_1} (\tau_k) \right)$
	may be matched to an empty bar if $\restr F {l_2} (\vine((\tau_j)) =  \restr
		F {l_2} (\vine(\tau_k))$.
		In particular, such continuous matching functions lead to indicator summands, which,
		once split up, lead to the desired \mmaout{} decomposition with interval summands. \\
	\end{remark}

	\begin{proof}
		First note that if $\gamma_1 = \gamma_2$, there is nothing to show, hence we
		assume $ \gamma_1 \neq \gamma_2$ in the following.
		Without loss of generality, we will first assume that the initial order
		$(\sigma_1,\dots,\sigma_N)$ on the first line is compatible with the partial
		order defined on the first line, i.e., $\ord{\restr{F}{l_1}} = \mathrm{id}$.

		First, note that on the set of diagonal lines in $\R^n$, the map $l\mapsto x^l$ is
		continuous, and the same goes for the map $l\mapsto (x_{\pi_1}^l,
		\dots, x _{\pi_N}^l)$, where $\pi = \ord{\restr F {l_2}}$.
		In particular, if $\gamma_1\neq \gamma_2$ then there exists an integer $i$ such that
		$x _{\pi_i}^l = x _{\pi _{i+1 }}^l$ for some line $l_1\le l\le l_2$.
		By the genericity assumption, this $i$ is unique which concludes that
		$\gamma_1 = \gamma_2 \circ (i \, i+1)$. \\

		Now, as  $\gamma_1 \neq \gamma_2$, we have $x_i^{l_1}\le x_{i+1}^{l_1}$ and $x_i^{l_2}> x_{i+1}^{l_2}$.
		Note that in particular that this implies that  both
		$x_i$, $x_{i+1}$ and %
		$x_i^{l_1}$, $x_{i+1}^{l_2}$ are strictly incomparable:
		\begin{itemize}
			\item as $x_i^{l_2}> x_{i+1}^{l_2}$, there exists a dimension $j$ (obtained with Equation~(\ref{eq:vineyards_line_proj})) such that ${
				\left( x_i \right)_j = \left( x_i^{l_2} \right)_j> \left( x_{i+1}^{l_2} \right)_j \ge \left( x_{i+1} \right)_j}$, and
			\item If $x_i$ and $x_{i+1}$ were comparable, we would have (by previous point)
				$x_{i+1} \leqn x_i$ with $x_i \neq x_{i+1}$, which would contradict
				the initial assumption, i.e., the fact that the initial order $ \left( \sigma_1,\dots,\sigma_N \right)$ is given by a completion
				of the original poset.
				Hence, there exists another dimension $j'\neq j$ (obtained again with Equation~(\ref{eq:vineyards_line_proj})) such that
				$ \left( x_i \right)_{j'} \le \left( x_i^{l_1} \right)_{j'}< \left( x_{i+1}^{l_1} \right)_{j'} = \left( x_{i+1} \right)_{j'}$.
		\end{itemize}
		Combining everything,
		we have:
		\begin{equation}\label{eq:vineyards_uncomp_idx_i}
			\exists 1 \le j \neq j' \le n, \quad  \left( x_i^{l_1} \right)_{j'}< \left( x_{i+1} \right)_{j'} \leq \left( x_{i+1}^{l_2} \right)_{j'}
			\quad \textnormal{ and } \quad
			\left( x_{i+1}^{l_2} \right)_j < \left(x_i\right)_j \leq \left( x_i^{l_1} \right)_j,
		\end{equation}
		which guarantees that %
		$x_i^{l_1}$ and $x_{i+1}^{l_2}$
		are strictly incomparable.

		Consider for any index $1\le j \le N$ the complex $K_j := \left< \sigma_1, \dots \sigma_j \right>$, and the following diamond
		\begin{equation}
			\begin{tikzcd}
				&                                           & H_*\left(K_{i-1}\cup\left\{\sigma_{i+1}\right\}\right) \arrow[rd, "d"]      &                       &        \\
				\cdots \arrow[r] & H_*(K_{i-1}) \arrow[r, "a"] \arrow[ru, "b"] & H_*\left(K_{i-1}\cup\left\{\sigma_i\right\}\right) \arrow[r] \arrow[r, "c"] &  H_*(K_{i+1}) \arrow[r] & \cdots,
				\label{eq:vineyard_diamond}
			\end{tikzcd}
		\end{equation}
		where the maps $a,b,c,d$ are induced by the inclusion.
		Note that as simplices are only added one by one, the maps $a,b,c,d$ are
		either surjerctive of nullity one or injective of corank one, depending on the positivity or negativity of the simplices $\sigma_i$ and $\sigma_{i+1}$.
		Furthermore, the Mayer-Vietoris theorem ensures that the following sequence is exact:
		\begin{equation*}
			H_*(K_{i-1})
		\xrightarrow{x\mapsto (x,x)}{H_*(K_{i-1}\cup \{\sigma_{i+1}\})\oplus K_{i-1}\cup \{\sigma_{i}\})}
		\xrightarrow{(x,y)\mapsto y-x}{H_*(K_{i+1})}.
	\end{equation*}

	Note that in the continous case, \cref{eq:vineyard_diamond}
	corresponds to the following diagram:
	\begin{equation}
		\begin{tikzcd}[column sep=3.5em, row sep=3em]
			H_*\left(F_{x_{i+1}^{l_2}}\right) \arrow[r, "d"] &  H_*\left(F_{x_{i}^{l_2}}\right) & \\
			H_*\left(F_{x_{i-1}^{l_2}}\right)\arrow[u, "b"]&  H_*\left(F_{\bigvee_{j\le i+1} x_j}\right) \arrow[r, "\sim"] \arrow[u, "\sim"] & H_*\left(F_{x_{i+1}^{l_1}}\right)  \\
			H_*\left(F_{\bigvee_{j\le i} x_j}\right) \arrow[u, "\sim"] \arrow[r,"\sim"]  & H_*\left(F_{x_{i-1}^{l_1}}\right)\arrow[r, "a"]& H_*\left(F_{x_i^{l_1}}\right),\arrow[u, "c"]
		\end{tikzcd}
	\end{equation}
	where $\bigvee$ denotes the coordinate-wise maximum, and the isometries are guaranteed by \cref{eq:vineyards_line_proj,eq:vineyards_ord_def}.
	Such diamonds are called {\em transposition diamonds}, on which \cite[Theorem
	2.4]{mariaZigzagPersistenceReflections2015} applies and states that
	if $\left\{ \tau_1, \dots, \tau_N \right\}$ is a reduced chain basis of
	the persistence module generated by applying the homology functor $H_*$ on the filtration $ \left( K_j \right)_{1\le j \le N}$,
	then, there is an explicit updated basis
	$\left\{ \vine(\sigma_1),\dots,\vine(\sigma_N) \right\}$,
	that is reduced for the filtered chain complex $( \hat K_j)_{1\le j
	\le N}$, where, given an index $j$, the complex $\hat K_j$ is defined as
	$\hat K_j := \left< \sigma_{(i\, i+1)1}, \dots, \sigma_{(i\, i+1) j}
	\right>$.
	This vine update follows a case study, that we follow below to show that the corresponding  matching function is compatible.
	\begin{enumerate}
			\def\v{{\ensuremath{\color{red} v}}}
			\def\u{{\ensuremath{\color{blue} u}}}
		\item The maps $a$ and $c$ are surjective of nullity 1, i.e., the added simplices $\sigma_i$ and $\sigma_{i+1}$ are negative.
			Let $\u,\v\in \left\{ \tau_1, \dots, \tau_N \right\}$ the chains
			generating these intervals, i.e., $\u =\partial \tau_i$ and $\v=\partial
			\tau_{i+1}$.
			\begin{enumerate}
				\item {}{} \label{vine:case1i}
					Assume that $\v\in \ker(b)$.
					See \cref{fig:vine1.i} for an illustration.
					In that case,
					\cite[Theorem 2.4]{mariaZigzagPersistenceReflections2015} guarantees that
					\begin{equation}
						\forall 1\le j \le N, \quad \tau_j \mapsto \vine \left( \tau_j \right) := \tau_j
					\end{equation}
					is a reduced basis of the filtered chain complex
					$(\hat K_j)_{1\le j\le N}$
					and hence of the filtered chain complex
					$\restr F {l_2}$.
					Hence, \cref{eq:vineyards_line_proj} guarantees that for any
					index $j$, the matched bar endpoints, $\tau_j \mapsto \vine(\tau_j)$, with filtration values $x_{j}^{l_1}$ and
					$x_j^{l_2}$ are not strictly comparable, and thus the induced matching function is compatible.
				\item{}{} \label{vine:case1ii} If $\v\notin \ker(b)$, then
					\cite[Theorem 2.4]{mariaZigzagPersistenceReflections2015}
					guarantees that there exists $\alpha\in k\setminus \left\{ 0
					\right\} $ such that
					$\u+\alpha \v \in \ker(b)$.
					See \cref{fig:vine1.ii} for an illustration.
					In other words, in the quotient space ${H_*(K_{i-1})}$, we
					have  $\u = u_1 \partial \sigma_i$
					and $\v = v_1\partial \sigma_{i} + v_2 \partial \sigma_{i+1}$ for some invertible
					constants $u_1, v_1, v_2 \in k\setminus \left\{ 0 \right\} $, with $ \alpha u_1 = -
					v_1$.
					Hence, we have:
					\begin{equation}\label{vine:case1ii_eq}
						\left\{  \vine(\u), \vine(\v) \right\} =
						\begin{cases}
							\left\{ \u, \u + \alpha \v \right\} \quad \textnormal{if } \restr
							F {l_2}(\u) \le \restr F {l_2}(\v), \\
							\left\{ \v, \u +\alpha\v\right\}\quad \textnormal{otherwise,}
						\end{cases}
					\end{equation}
					and the identity for the non-impacted chains:
					\begin{equation}
						\forall \tau_j \notin{ \left\{ \u,\v, \tau_i,
						\tau_{i+1}\right\} },
						\quad
						\vine \left( \tau_j \right) : =
						\tau_j.
					\end{equation}
					The chains $\tau_i$ and $\tau_{i+1}$ are then updated such that
					the vine update commutes with the boundary, i.e., such
					that:
					\begin{equation}\label{eq:vine_coboundary_update}
						\vine (\partial \tau_i) = \vine(\u) = \partial
						\vine(\tau_i)
						\quad \textnormal{and} \quad
						\vine (\partial \tau_{i+1}) = \vine(\v) = \partial
						\vine(\tau_{i+1}).
					\end{equation}
					In particular, \cref{eq:vineyards_line_proj} guarantees once
					again that (at least) all but four  endpoint (two bars, with two endpoints each) matching are compatible.
					Furthermore, note that since the simplices' order only differ from the
					permutation $(i\, i+1)$ and the chains $\u,\v$ belong to $ K_{i-1}$, there exist two indices $j,k < i $ such that the inequality $x_j^{l_1} = \restr F {l_1}(\u)
					\le \restr F {l_1}(\v) = x_k^{l_1}$ is equivalent to $x_j^{l_2}=\restr F {l_2}(\u)
					\le \restr F {l_2}(\v) = x_k^{l_2}$.
					We fix such indices $j$ and $k$.
					\\
					We follow the two different cases.
					\begin{enumerate}
						\item {}\label{vine:negok}
							In the first case, we have  $\vine(\u) = \u$ (with
							$\vine(\tau_i) = \tau_i$)
							and
							$\vine(\v) =\u+\alpha\v$
							(with $\vine(\tau _{i+1}) = \tau _{i} + \alpha \tau _{i+1}$).
							Furthermore,  since $\restr F {l_2}(\u)
							\le \restr F {l_2}(\v)$, we also have $\restr F
							{l_2} \left( \u+ \alpha\v \right)
							= \restr F {l_2} (\sigma_k)
							=\restr F
							{l_2} \left(\v \right) =  x_{k}^{l_2}$.
							Hence, \cref{eq:vineyards_line_proj} guarantees once again that the birthpoint matching is compatible.
							The same goes for the deathpoints matching, since for $j\in \left\{ i,i+1 \right\}$, we match the endpoints
							$x_{j}^{l_1} = \restr F {l_1}(\tau_j)$ and
							$x_{j}^{l_2} = \restr F {l_2}(\vine(\tau_j))$,
							which are not strictly comparable.
							Thus, the matching function is compatible.

						\item \label{vine:negnok} In the second case, we have
							$\vine(\u) = \u+\alpha\v$ (with $\vine(\tau_i) =
							\tau_i + \alpha \tau_{i+1}$)
							and $\vine(\v) = \v$ (with $\vine(\tau_{i+1}) =
							\tau_{i+1}$).
							In this case, we have $\restr F {l_2} (\u) > \restr F {l_2}(\v)$, and we match the birthpoints
							\begin{equation*}
								x_j^{l_1} = \restr F {l_1} (\u) \textnormal{ with }
								\restr F {l_2}(\vine(\u)) =
								\restr F {l_2}(\u + \alpha\v )  = \restr F {l_2} (\u) = x_j^{l_2},
								\textnormal{ and,}
							\end{equation*}
							\begin{equation*}
								x_k^{l_1} = \restr F {l_1} (\v) \textnormal{ with }
								x_{k}^{l_2} = \restr F {l_2}(\vine(\v)) =
								\restr F {l_2}(\v),
							\end{equation*}
							and the deathpoints
							\begin{equation*}
								x_i^{l_1} = \restr F {l_1} (\tau_i) \textnormal{ with }
								\restr F {l_2}(\vine(\tau_i)) =
								\restr F {l_2}(\tau_i + \alpha\tau_{i+1}) =
								\restr F {l_2} (\sigma_{i+1}) =
								x_{i+1}^{l_2},
								\textnormal{ and }
							\end{equation*}
							\begin{equation*}
								x_{i+1}^{l_1} = \restr F {l_1} (\tau_{i+1}) \textnormal{ with }
								x_{i}^{l_2} = \restr F {l_2}(\vine(\tau_{i+1})) =
								\restr F {l_2}(\tau_{i+1}) = \restr F {l_2}(\sigma_i).
							\end{equation*}
							Finally, the birthpoints are not strictly comparable hence compatible thanks to
							\cref{eq:vineyards_line_proj}
							and the deathpoints thanks to
							\cref{eq:vineyards_uncomp_idx_i}.
							The matched bars are thus compatible.
					\end{enumerate}
			\end{enumerate}
			\begin{figure}[H]
				\centering
				\begin{subfigure}[b]{.45\textwidth}
					\includesvg[width=\textwidth]{images/vineyards_diams/vine_1.i.svg}
					\caption{Illustration of case \ref{vine:case1i}.}
					\label{fig:vine1.i}
				\end{subfigure}
				\hfill
				\begin{subfigure}[b]{.45\textwidth}
					\includesvg[width=\textwidth]{images/vineyards_diams/vine_1.ii.svg}
					\caption{Illustration of case \ref{vine:case1ii}.}\label{fig:vine1.ii}
				\end{subfigure}
				\caption{Vineyard case: $\sigma_i$ and $\sigma _{i+1 }$ negative.}
			\end{figure}
		\item The map $a$ and $c$ are injective of corank $1$, i.e., the
			simplices $\sigma_i$ and $\sigma_{i+1}$ are positive.
			See \cref{fig:vine2.ii} for an illustration.
			\begin{enumerate}
				\item \label{vine:case2i} Assume that $\u\in \mathrm{im}(d)$.
					Then, {\cite[Theorem~2.4]{mariaZigzagPersistenceReflections2015}}
					guarantees that the updated basis given by
					$\tau_j\mapsto \vine \left( \tau_j \right) := \tau_j$
					is a reduced basis of $\restr F {l_2}$ as well.
					Using a similar argumentation as
					Case \ref{vine:case1i}, the matching function is compatible.
				\item \label{vine:case2ii} If there exists an $\alpha \in k
					\setminus \left\{ 0 \right\} $ such that $\u + \alpha \v\in
					\mathrm{im}(d)$, i.e., $\u + \alpha\v \in H_*(\hat K_i)$, i.e., in
					$H_*(K_{i-1})$, we have $\u = u_1 \sigma_i$ and $\v = v_1 \sigma_i + v_2
					\sigma_{i+1}$ , for invertible constants $u_1,v_1, v_2 \in \field \setminus \left\{ 0  \right\}$ satisfying $\alpha u_1 = -v_1$.
					We also have two cases
					\begin{equation*}
						\left\{ \vine(\u), \vine(\v) \right\} =
						\begin{cases}
							\left\{  \u+\alpha\v, \u \right\} \quad \textnormal{if }\restr F {l_2}(\delta\u) \le \restr F {l_2} (\delta\v), \\
							\left\{  \u+\alpha\v,\v \right\} \quad \textnormal{otherwise,}
						\end{cases}
					\end{equation*}
					where $\restr F {l_2} (\delta \u)$ (resp. $\restr F {l_2} (\delta \v)$)
					is the first time in which a coboundary of $\u$ (resp. $\v$) appears where, given a chain cycle $\tau$, we define
					$\restr F {l_2} (\delta \tau):= \inf \left\{ t\in l_2 : \tau
					= 0 \textnormal{ in } H_*(F_t)\right\}\in \mathbb R \cup
					\left\{ +\infty \right\} $.
					When they exist, we consider the first coboundary $\tau_j$
					(resp. $\tau_k$) in
					$\restr F {l_1}$ of $\u$ (resp. $\v$) i.e., when the cycle $\u$
					(resp. $\v$) is not essential, we consider the index $j>i+1$
					(resp. $k>i+1$) such that
					$\partial \tau_j = \u$ (resp. $\partial \tau_k = \v$),
					and hence $\restr F {l_1}(\delta\u) = \restr F {l_1} (\tau_j)$
					\begin{equation}\label{eq:vine:coboundary_stuff}
						\restr F {l_1}(\delta\u) = \restr F {l_1} (\tau_j)
						\quad
						\textnormal{ and respectively }
						\quad
						\restr F {l_1}(\delta\v) = \restr F {l_1} (\tau_k)
					\end{equation}
					Note that since the simplices' order on $l_1$ and $l_2$
					only differ from the permutation $(i\, i+1)$,
					\cref{eq:vine:coboundary_stuff} is also satisfied when the line
					$l_1$ replaced by $l_2$.
					\\
					The other simplices being updated by the identity
					i.e. by $\vine(\tau_j): = \tau_j$, unless their boundary
					is $\u$ or $\v$, in which case there are updated as in
					previous Cases \ref{vine:negok} and \ref{vine:negnok}, by
					\cref{eq:vine_coboundary_update}.
					\begin{enumerate}
						\item In the first case ($\restr F {l_2}(\delta \u) \le \restr F {l_2}(\delta\v)$), simlilarly to Case \ref{vine:negok}, we have
							$\vine(\u) = \u$ (with $\vine(\tau_j) =\tau_j$), and $\vine(\v) = \u + \alpha \v$ (with $\vine(\tau_{j+1}) = \tau_j + \alpha\tau_{j+1}$); which ensures,
							since $\restr F {l_2} \left( \u + \alpha\v \right) = x_{i+1}^{l_2}$ and  $\restr F {l_1}(\v) = x_{i+1}^{l_1}$,
							that this case also induces a compatible matching function.
							If $\tau_j$ exists, then $\restr F {l_2}(\vine
							(\tau_j)) = \restr F {l_2} (\tau_j)$ which induces a
							compatible endpoint matching.
							If $\tau_j$ and $\tau_k$ exist, then, as $\restr F
							{l_2} (\tau_j) < \restr F {l_2} (\tau_l)$, we
							have $\restr F {l_2} (\vine(\tau_k))  = \restr F {l_2}(\tau_j + \alpha\tau_k) = \restr F {l_2} (\tau_k)$ which also induces a compatible endpoint matching.
							Hence, the vineyard barcode matching is compatible in this case.
						\item The second case ($\restr F {l_2}(\delta \v) < \restr F {l_2}(\delta\u)$) is similar to Case \ref{vine:negnok}, we pick
							$\vine(\u) = \u + \alpha\v$ and $\vine(\v) = \v$.
							We match the birthpoints
							\begin{equation*}
								x_i^{l_1} = \restr F {l_1} (\u) \textnormal{ with }
								\restr F {l_2}(\vine(\u)) =
								\restr F {l_2}(\u + \alpha\v )  					                  = x_{i+1}^{l_2},
								\textnormal{ and,}
							\end{equation*}
							\begin{equation*}
								x_{i+1}^{l_1} = \restr F {l_1} (\v) \textnormal{ with }
								\restr F {l_2}(\vine(\v)) =
								\restr F {l_2}(\v )  = \restr F {l_2} (\sigma_i) = x_{i}^{l_2}.
							\end{equation*}
							\cref{eq:vineyards_uncomp_idx_i} hence guarantees that these endpoints are strictly uncomparable, hence compatible.
							Furthermore, when $\tau_k$ exists, we match the deathpoints
							\begin{equation*}
								x_{k}^{l_1} = \restr F {l_1} (\tau_k) \textnormal{ with }
								\restr F {l_2}(\vine(\tau_k)) =
								\restr F {l_2}(\tau_k ) = x_{k}^{l_2}.
							\end{equation*}
							and when $\tau_j$ exist as well:
							\begin{equation*}
								x_{j}^{l_1} = \restr F {l_1} (\tau_j) \textnormal{ with }
								\restr F {l_2}(\vine(\tau_j)) =
								\restr F {l_2}(\tau_j + \alpha\tau_k )  = \restr F {l_2} (\tau_j) = x_{j}^{l_2},
							\end{equation*}

							In both cases, these deathpoint matching are compatible thanks to
							\cref{eq:vineyards_line_proj}.
							The matching function is therefore also compatible.
					\end{enumerate}
			\end{enumerate}
			\begin{figure}[H]
				\centering
				\begin{subfigure}[b]{.45\textwidth}
					\includesvg[width=\textwidth]{images/vineyards_diams/vine_2.i.svg}
					\caption{Illustration of case \ref{vine:case2i}.}
					\label{fig:vine2.i}
				\end{subfigure}
				\hfill
				\begin{subfigure}[b]{.45\textwidth}
					\includesvg[width=\textwidth]{images/vineyards_diams/vine_2.ii.svg}
					\caption{Illustration of case \ref{vine:case2ii}.}
					\label{fig:vine2.ii}
				\end{subfigure}
				\caption{Vineyard case: $\sigma_i$ and $\sigma _{i+1 }$ positive.}
			\end{figure}
		\item \label{vine:case3} $a$ is injective of corank 1 and $c$ is surjerctive of nullity 1, i.e., $\sigma_i$ is positive and $\sigma_{i+1}$ is negative.
			See \cref{fig:vine3} for an illustration.
			In this case, since these two chains are not of the same
			dimension, they do not interact together. Hence, we have
			$\vine(\u)=\u$ and $\vine(\v) = \v$ and
			\cref{eq:vineyards_line_proj} guarantees that the matching
			function is compatible.
		\item \label{vine:case4}  $a$ is surjective of nullity 1 and $c$ is
			injective of corank 1, i.e., $\sigma_i$ is negative and
			$\sigma_{i+1}$ is positive.
			See \cref{fig:vine4} for an illustration.
			This case is symmetrical to Case
			\ref{vine:case3}.
			\begin{figure}[H]
				\centering
				\begin{subfigure}[b]{.35\textwidth}
					\includesvg[width=\textwidth]{images/vineyards_diams/vine_3.svg}
					\caption{Illustration of case \ref{vine:case3}.}
					\label{fig:vine3}
				\end{subfigure}
				\hfill
				\begin{subfigure}[b]{.35\textwidth}
					\includesvg[width=\textwidth]{images/vineyards_diams/vine_4.svg}
					\caption{Illustration of case \ref{vine:case4}.}
					\label{fig:vine4}
				\end{subfigure}
				\caption{Vineyard case: $\sigma_i$ and $\sigma _{i+1 }$ of different signs.}
			\end{figure}
	\end{enumerate}
\end{proof}

\begin{remark}
	Note that, in this proof, and in the specific cases for which the vineyards
	algorithm makes a matching choice (which are Cases \ref{vine:case2ii} and \ref{vine:case1ii}), these corresponding choices are made so that one cycle is left unchanged.
	Furthermore, as permuted simplices are guaranteed to be strictly incomparable
	(\cref{eq:vineyards_uncomp_idx_i}),
	both choices (i.e., permuting $i$ and $i+1$ or not) induce a compatible
	matching.
	Also notice that, in the generic case, this case study allows to span the set of all
	possible matchings between two lines whose induced orders only differ from a single $(i\, i+1)$ transposition. \\
\end{remark}

\begin{remark}[Extension to free presentations]\label{rem:free_pres}
	In this proof, we only used the framework of simplicial complexes
	in order to ensure that
	simplices can be added one by one, which in turn guarantees that the
	corresponding linear maps are
	either injective of corank 1, or surjective of nullity 1.
	However, given a multi-parameter persistence module $\Mbb$, such an assumption
	on linear maps can also be guaranteed
	using a finite free presentation of $\Mbb$ (see
	\cite[Section 7]{botnanIntroductionMultiparameterPersistence2023}).
	Hence, computing a minimal presentation can be seen as a pre-processing step
	of our \MMA{} algorithm. \\
\end{remark}

\begin{remark}[Coxeter decompositions]
	Given a permutation of simplices between two lines, one might wonder how to
	decompose it into a product of transpositions $(i\, i+1)$ in order to apply the
	vineyards algorithm.
	This can be achieved in the three following steps.
	A permutation $\sigma\in
	\mathfrak S_n$ can always be decomposed into a product of cyclic
	permutations, i.e., permutations $\sigma_{i_1,\dots,i_k}\in \mathfrak S_n$
	satisfying:
	\begin{equation*}
		\forall 1\le i \le n, %
		\sigma_{i_1,\dots,i_k} (i) =
		\begin{cases}
			i_{j+1 \, \mathrm{mod}\, k} \textnormal{ if } i = i_j, \textnormal{ for some }1\leq j \leq k, \textnormal{ and } \\
			i \textnormal{ otherwise.}
		\end{cases}
	\end{equation*}
	Then, using the fact that any cyclic permutation
	$\sigma_{i_1,\dots,i_k} = (i_1\, i_{2})\circ (i_2\, i_3)\circ \cdots
	(i_{k-1}\, i_k)$ is a product of transpositions,
	it follows that every permutation is a product of transpositions as well.
	Finally, if $(i,j)$ is a transposition with $1\le i<j\le n$, using the following relation:
	\begin{equation*}
		(i\, j) = (i\, i+1)\circ \cdots \circ (j-1\, j) \circ (j-2\, j-1) \circ \cdots \circ(i \, i+1),
	\end{equation*}
	one can see that every permutation is in fact the product of adjacent transpositions.
	In practice, this product can be retrieved using a sorting algorithm, such as
	the insertion sort, or the bubble sort (see, e.g., \cite{knuthArtComputerProgramming1995}). \\
\end{remark}

\begin{remark}[Lazy vineyard update]
	Note that the $(i \, i+1)$ swaps can be directly inferred with
	the approach provided in
	\cite[Section 4, Lazy
	minimazation]{kerberFastMinimalPresentations2021}.
	More formally, using the same notations, the swaps $(i\, i+1)$ only occur
	precisely when two incomparable filtration values satisfy $x_i^l = x_{i+1}^l$
	for some line $l$,
	which, in the two-parameter case, corresponds to the presence of a line crossing the point
	$(\max \left\{ (x_i)_1,(x_{i+1})_1 \right\}, \max \left\{ (x_i)_2,(x_{i+1})_2
	\right\})$.
	This suggests that our \MMA{} algorithm can be trivially extended to fibered
	barcodes involving only such lines (instead of $\delta$-grids of lines), thus reducing its running time. We stick to grids in this article for the sake of clarity. \\
\end{remark}

\begin{proposition}[Vineyards is compatible for $n=2$]\label{prop:vine_compat}
	Let $\Mbb$ be a $2$-parameter persistence module, and $L$ be an ordered set of diagonal
	lines $L:=(l_i)_{1\le i\le N}$ with
	increasing basepoints.\footnote{Ordering $L$ in such a way is possible precisely because $n=2$.} %
	For each index $1\le i\le N$, let $\sigma_i$ be an arbitrary compatible matching function between $\barcode{\restr{\Mbb}{l_i}}$ and
	$\barcode{\restr{\Mbb}{l_{i+1}}}$ (obtained with, e.g.,~\cref{prop:vine_update_comp}).
	Then, the matching function $\sigma = (\sigma_i)_{1\le i \le N}$ is a compatible
	matching function on $L$.
\end{proposition}

\begin{proof}
	First, note that, given a compatible matching function
	$\sigma$ between two diagonal lines $l,l'\in L$, and letting
	$b\in\barcode{\restr{\Mbb}{l}},b'\in\barcode{\restr{\Mbb}{l'}}$ be any
	non-trivial pair of matched bars,
	then, assuming that, e.g., $(x,0)\le (x',0)$ are the basepoints of
	$l$ and $l'$ respectively, one has $\min (b)_1\le \min(b')_1$ and $\max(b)_2
	\ge \max(b')_2$.
	See \cref{fig:2d_compatible}. %
	\begin{figure}[H]
		\centering
		\includesvg[width=.3\textwidth]{vineyards_diams/2d_compatible}
		\caption{Let $l_1 < l_2$ be two diagonal lines of $\mathbb{R}^2$, and $x$,
			$y$ be two matched points in $l_1$ and $l_2$ respectively.
		Then, if the matching is compatible, one must have $x_1 \le y_1$ and $x_2 \ge y_2$.}
		\label{fig:2d_compatible}
	\end{figure}

	Now, let $x\in l_1$ be the grade of a given bar endpoint. A point $y$ is strictly incomparable with $x$ if either:
	\begin{enumerate}
		\item $y_1 \le x_1$ and $y_2 \ge x_2$, or
		\item $y_1 \ge x_1$ and $y_2 \le x_2$.
	\end{enumerate}
	Now, assuming that $y\in l_2$ and $x$ and $y$ are strictly incomparable, Case 1 can be excluded. Indeed, in that case,
	there exist constants $x_0,y_0,s,t\in \R$ such that
	$x= (x_0+s,s)$ and $y=(y_0+t,t)$, and hence:
	\begin{equation*}
		y_1\le x_1 \textnormal{ and } x_0<y_0 \,\implies\, t<s, \textnormal{ i.e., } y_2\le x_2 \textnormal{ and } y\leqn x\textnormal{, which is a contradiction}.
	\end{equation*}
	This shows that two bar endpoints matched by a compatible matching function
	between two consecutive diagonal lines with increasing basepoints must satisfy Case 2, and thus, that a sequence of compatible matching functions is still compatible.
\end{proof}

\section{Numerical experiments}
\label{sec:expe}

In this section, we showcase the performances of \MMA{}.
More precisely, we empirically study how the output quality and running time depend on the precision $\delta$ and the number of lines $|L|$. %
Then, we compare the running times of \MMA{} with those of
\RIVET~\cite{lesnickInteractiveVisualization2D2015} and the elder-rule
staircode (ERS)~\cite{caiElderRuleStaircodesAugmentedMetric2021},
which are our closest competitors in terms of producing visual and interpretable descriptors of persistence modules. %
Finally, we investigate how running time is affected by the number of filtrations. %
All experiments were done on a laptop with AMD Ryzen 4800 CPU and 16GB of RAM.
Our code is part of the \texttt{multipers} library
\cite{loiseauxMultipersMultiparameterPersistence2024} and is publicly available at \url{https://github.com/DavidLapous/multipers}. %

\paragraph*{Interpretation.}\label{expe:interpretability}
We first qualitatively show how to interpret our \mmaout{} decompositions on a toy
dataset in \cref{fig:mma_vs_pointed_filtration}, using two different
two-parameter filtrations (Čech and (edge-collapsed) Rips with sublevel sets of codensity) on a point cloud, computed using \cite{alonsoDelaunayBifiltrationsFunctions2024}.
In these examples,
we consider a coordinate $x\in \R^2$ that is included in the support of some specific summands of the \mmaout{} decomposition, i.e., in some specific colored shapes in the plot.
Then, we look back at the filtration at this coordinate $x$ and we identify the corresponding cycles.
Note that, at each of these points $x$, the cycle representatives (from
\cref{th:pers_pairing}) are already calculated when running \MMA{} and can thus be used for interpretation without additional computational cost.

\begin{figure}
	\centering{}
	\begin{subfigure}[b]{\textwidth}
		\includegraphics[width=.9\textwidth]{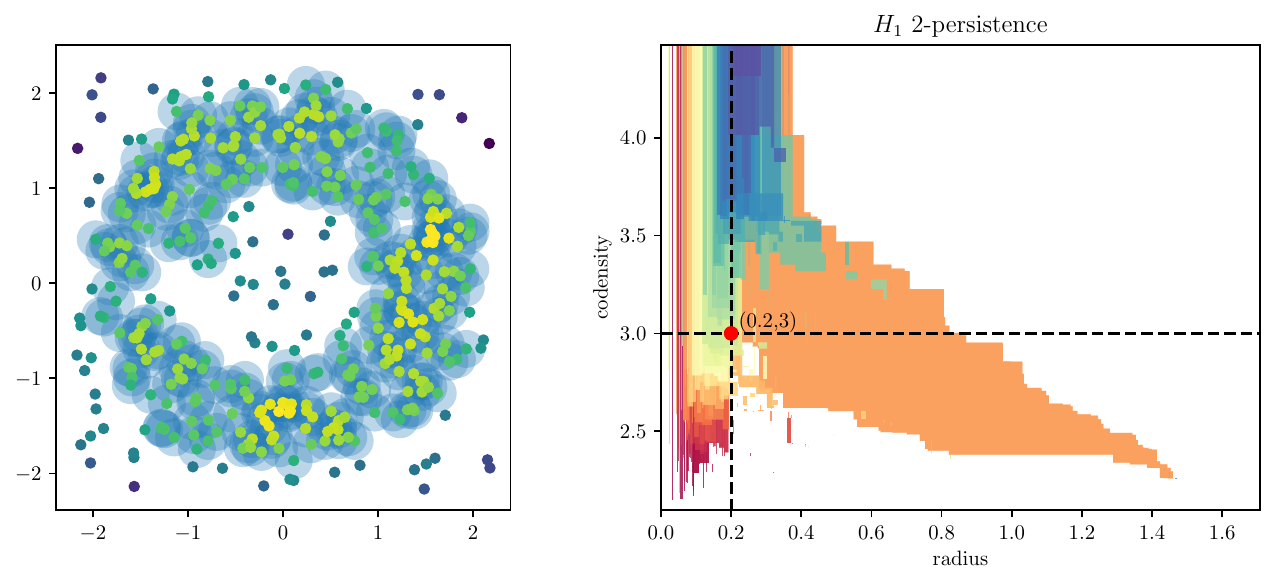}
		\caption{
			\textbf{(Left)} The Čech complex, at radius $0.2$, of the points with codensity values larger than $3$.
			\textbf{(Right)} The corresponding \mmaout{} decomposition, with a red dot at
			the coordinates fixed by the radius and codensity values used on the left.
		}
	\end{subfigure}
	\begin{subfigure}[b]{\textwidth}
		\includegraphics[width=.9\textwidth]{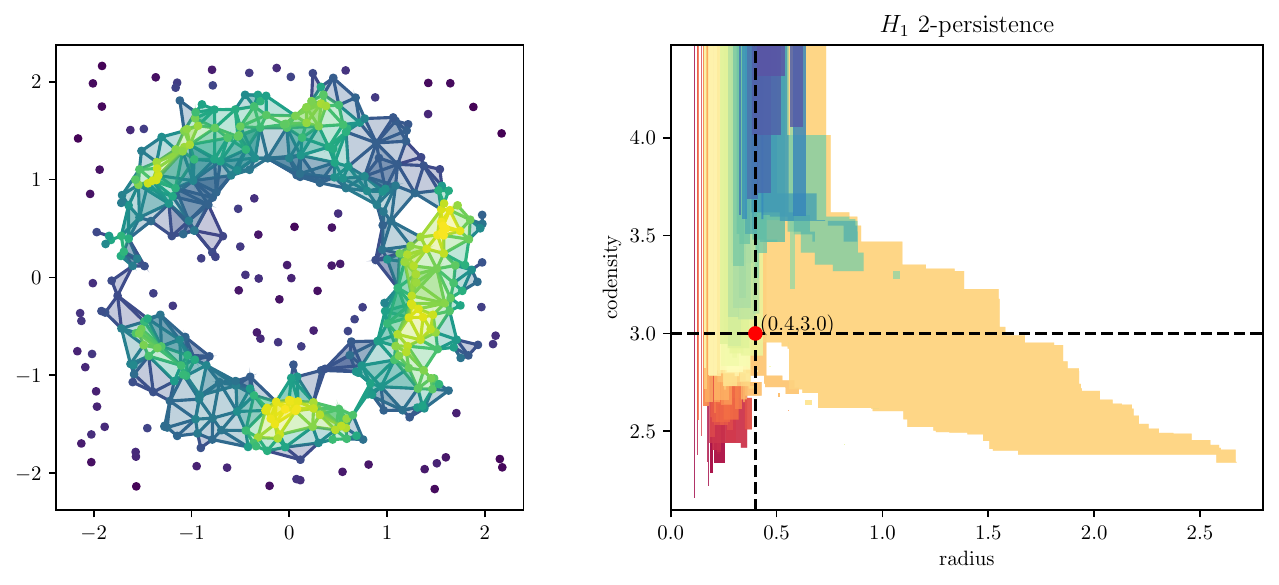}\\
		\caption{
			\textbf{(Left)} The (edge-collapsed) Rips complex, at radius $0.4$, of the points
			with codensity values larger than $3$.
			\textbf{(Right)} The corresponding \mmaout{} decomposition, with a red dot at
			the coordinates fixed by the radius and codensity values used on the left.
		}
	\end{subfigure}
	\caption{Interpretation of \mmaout{} decompositions computed by \MMA{}
	on a toy dataset.}
	\label{fig:mma_vs_pointed_filtration}
\end{figure}

We also provide another example, containing $25,000$ points uniformly sampled on the
unit square (noise), and $25,000$ points sampled on three distinct annuli
with different sizes concentration levels (signal). See \cref{fig:3_circle}.
As before, we then compute the Čech-codensity filtration on this point cloud.
The first cycle is very dense, so it should appear quickly w.r.t. the codensity
filtration (i.e., the lower part of the two-parameter filtration), and it is also small,
so it is expected to die quickly w.r.t. the Čech radius parameter.
The second cycle is bigger, and slightly less dense, so we can expect
it to appear later and survive more on the right side, and the same
goes for the third one.
The fourth one appears when the Čech radius is large enough (in order to connect the
three previous cycles together), and the condensity parameter is large enough as well (such that the three previous cycles are visible).

\begin{figure}
	\centering
	\includesvg[width=.9\textwidth]{images/3circle_7.svg}
	\caption{
		\textbf{(Left)} The point cloud dataset, colored with
		the (estimated) density of the sampling.
		\textbf{(Right)} The \mmaout{} decomposition produced by \MMA{}.
		One can see that four interval summands clearly stand out, and can be interpreted as described in the text.
		The summands that are induced by noise are all located on the rainbow strip
		on the left side.
	}\label{fig:3_circle}
\end{figure}

\paragraph*{Data sets and filtrations.}
In %
our next two experiments,
we focus on two real-world data sets of point clouds.
The first ones, called \texttt{LargeHypoxicRegion}, were obtained from immunohistochemistry in~\cite{vipond2021multiparameter}. These
are made of a few thousand points,  each representing a single cell.
The others were obtained from applying time-delay embedding in $\R^2$ on time series taken from a few
data sets (\texttt{Wine}, \texttt{Plane}, \texttt{OliveOil}, \texttt{Coffee}) from the \texttt{UCR} archive~\cite{UCRArchive}.
On all of these data sets, we computed bi-filtrations using the standard
Vietoris-Rips filtration, and the superlevel
sets of a Gaussian kernel density estimation (with bandwidth parameter $0.1d$
where $d$ is the diameter of the dataset), and we applied \MMA{} and its
competitors on the corresponding multi-parameter persistence modules in
homology dimensions $0$ and $1$ (note that the ERS can only be computed in degree
$0$). A typical example of an \MMA{} representation is given in
Figure~\ref{fig:immuno_mma_h01}.

In %
our third experiment,
we measure the dependence on the number of filtrations using a synthetic data set
obtained by sampling $300$ points in the unit square $[0,1]^2$, computing their Alpha simplicial complex, and generating
$n$ random filtration values on the points. %

Finally, it is worth noting that
we used multi-parameter edge collapses~\cite{filtration-domination} in order to
simplify the multi-parameter persistence modules (without losing information)
as much as possible before applying \RIVET{} and \MMA{}. The timing of this
simplifications are not taken into account, but they are not the computational
bottleneck of our computations. Furthermore, for the
\texttt{LargeHypoxicRegion}, we thresholded the Rips edges at $0.02$, leading
to simplicial complexes of $\sim\!85k$ and $\sim\!125k$ simplices respectively,
after simplifications.

\begin{figure}[H]
	\centering
	\includegraphics[width=\textwidth]{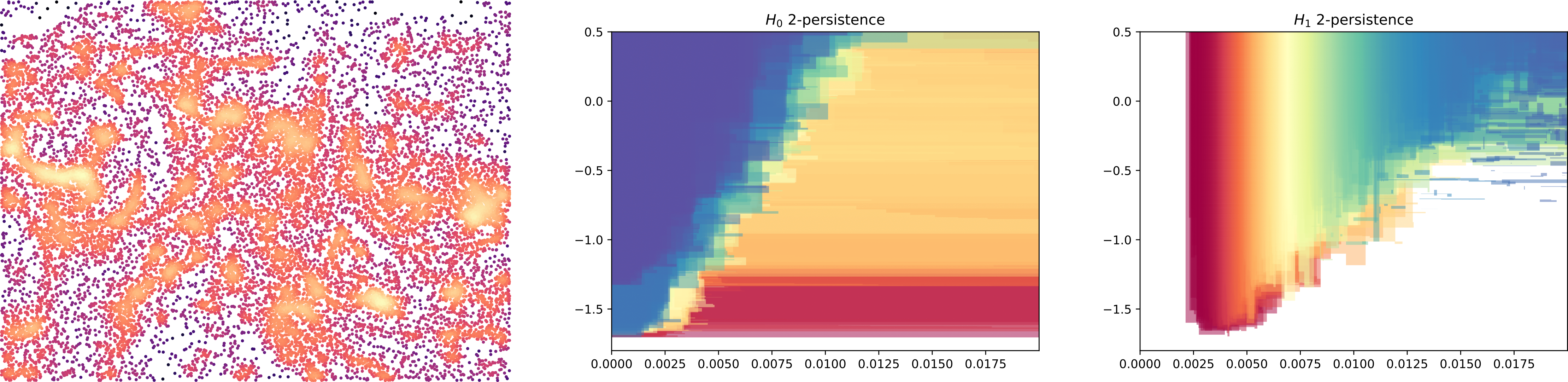}
	\caption{
		\textbf{(left)} \texttt{LargeHypoxicRegion2}, colored  with a kernel density estimation of bandwidth $0.01$,
		\textbf{(middle, resp. right)} Interval decomposition of degree 0 (resp. 1) homology given by \MMA{}. There are $\approx 28k$ non-trivial intervals, each one having a unique color.
	}
	\label{fig:immuno_mma_h01}
\end{figure}

\paragraph*{Convergence of \MMA{}.} %
We first empirically validate Propositions~\ref{prop:tildeIcandidate},
\ref{prop:mma_stab} and~\ref{prop:approx} by measuring how far is the output of
\MMA{} from the data when the precision parameter $\delta$ decreases and the
number of lines in $L$ increases. For this, we used the bottleneck distances
between the fibered barcodes (on $100$ random diagonal lines) of the outputs of
\MMA{} and the ones of the underlying modules as a proxy for the interleaving distances (since they are practically very difficult to evaluate). We call this the {\em estimated matching distance}.
Results are displayed in Figure~\ref{fig:conv_line}. One can see that the
convergence is empirically linear with the number of lines $|L|$, which is in
line with Propositions~\ref{prop:tildeIcandidate}, \ref{prop:mma_stab}
and~\ref{prop:approx} (since $|L|$ increases linearly as $\delta$ decreases for a fixed $n$). Note that the distance even reaches $0$ on a few cases.

\begin{figure}[h]
	\centering
	\includegraphics[width=0.45\textwidth]{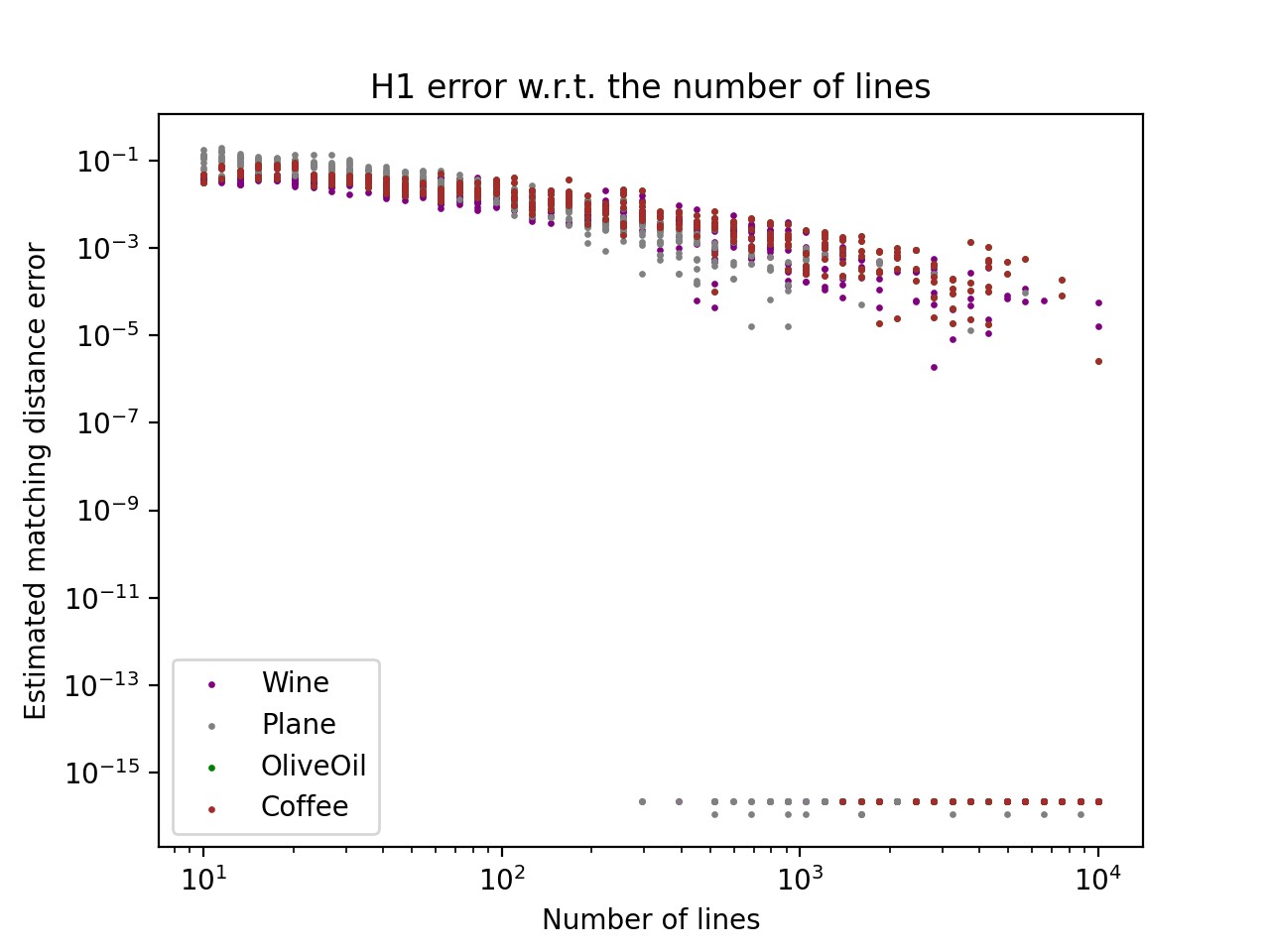}
	\includegraphics[width=0.45\textwidth]{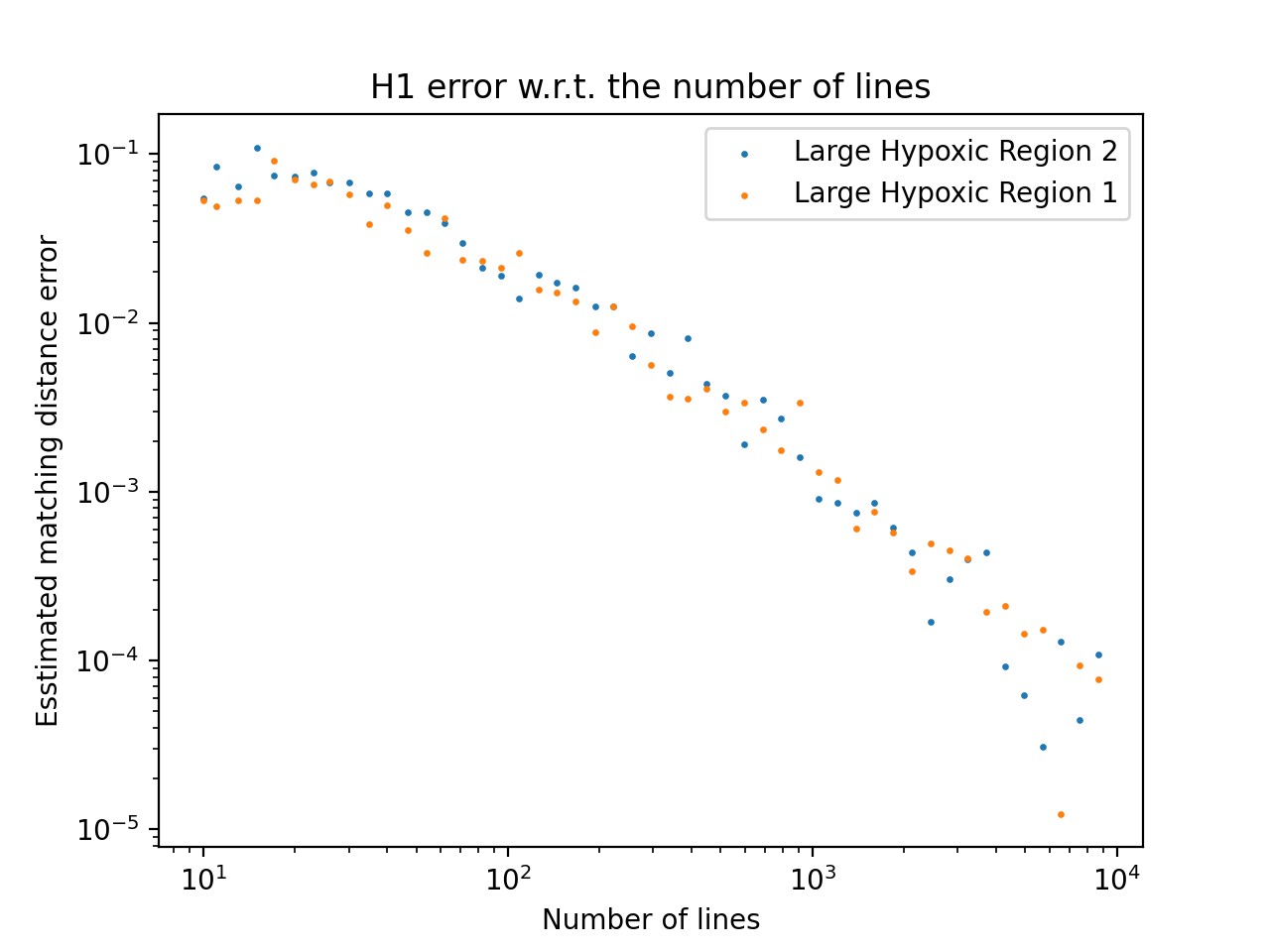}
	\caption{Convergence on \texttt{UCR} data sets {\bf (left)} and immunohistochemistry data sets {\bf (right)}.}
	\label{fig:conv_line}
\end{figure}

\paragraph*{Running times.} %
We now compare the running times of \MMA{} with those of \RIVET{} and the ERS.
Results can be found in Table~\ref{tab:TSres-small}. It is worth noting that on
several occasions, \RIVET{} and the ERS could not produce outputs in reasonable
time, due (among other things) to large memory consumption. On the other hands,
the fact that \MMA{} produces discretely presented intervals allows to encode them in a sparse manner with their corners.
Note that computations with \RIVET{} can be also approximated by coarsening the
filtration values, and thus the module. In practice, this corresponds to
restricting the 2-module to a $\kappa = \kappa_x\times \kappa_y$ grid, where
$\kappa_x, \kappa_y \in \mathbb N$ are the resolutions in both axes, see \cite[Section 8.2]{lesnickInteractiveVisualization2D2015}.
The parameters $\kappa$ and $|L|$ have (roughly) the same role, and can be related  with $\sqrt \kappa \simeq |L|$ (for a given, prescribed precision). %
Overall, we find that the running times of \MMA{} outperforms its competitors, except when
$\kappa$ is very small.
However, in this case, %
the corresponding output of \RIVET{} is very crude as the module is restricted on just a few points,
whereas \MMA{} produces intervals that are accurate along whole straight lines:
we find that for large data sets, \MMA{} is the only method that can produce accurate representations.

\begin{table}[h!]
	\resizebox{\columnwidth}{!}{%
		\begin{tabular}{|c | c c c | c | c c c |}
			\hline
			& \multicolumn{3}{c|}{\RIVET{}} & ERS & \multicolumn{3}{c|}{\MMA{}} \\
			& $\kappa=10^2$ & $\kappa=50^2$ & $\kappa=100^2$ & & $|L|=100$ & $|L|=1000$ & $|L|=10,000$\\
			\hline
			\texttt{Coffee} &
			$0.21 \pm 0.01, 0.18 \pm 0.01$ &
			$9.75 \pm 5.92, 0.35\pm 0.12$ &
			$--, 0.95\pm 0.56$
			&
			$0.34 \pm 0.04$
			&
			$0.0093\pm 0.001$ &
			$0.024 \pm 0.001$ &
			$0.16  \pm 0.008$ \\
			\texttt{Plane} &
			$0.19 \pm 0.005, 0.18 \pm 0.03$ &
			$4.36 \pm 2.24, 0.28\pm 0.04$ &
			$33.3 \pm 17.5, 0.56\pm 0.17$
			&
			$0.09\pm 0.03$
			&
			$0.004\pm 0.0$ &
			$0.012 \pm 0.001$ &
			$0.095 \pm 0.004$ \\
			\texttt{Wine} &
			$0.21 \pm 0.003, 0.19 \pm 0.007$ &
			$8.50 \pm 2.00, 0.22\pm 0.01$ &
			$--, 0.28\pm 0.023$
			&
			$0.22 \pm 0.04$
			&
			$0.004\pm 0.0$ &
			$0.016 \pm 0.0$ &
			$0.129 \pm 0.002$ \\

			\texttt{OliveOil} &
			$0.21 \pm 0.004, 0.19 \pm 0.002$ &
			$5.55 \pm 1.20, 0.31\pm 0.016$ &
			$--, 0.82\pm 0.17$
			&
			$1.39 \pm 0.03$
			&
			$0.026 \pm 0.001$ &
			$0.058 \pm 0.001$ &
			$0.37 \pm 0.006$
			\\
			\texttt{Worms} &
			$0.29 \pm 0.082, 0.23 \pm 0.23$ &
			$19.9 \pm 14.4, 4.60 \pm 5.0$ &
			$--, 31.36 \pm 36.24$
			&
			$3.85 \pm 0.1$
			&
			$0.22 \pm 0.11$ &
			$0.34 \pm 0.15$ &
			$1.35 \pm 0.41$
			\\
			\hline
			\texttt{LargeHypoxicRegion1} &
			$1.73,2.88$ &
			$--, 234$ &
			$--,--$
			&
			$--$
			&
			$26.4$ &
			$26.6$ &
			$59.4$
			\\
			\texttt{LargeHypoxicRegion2} &
			$2.39,6.04$ &
			$--,--$ &
			$--,--$
			&
			$--$
			&
			$57.3$ &
			$54.3$ &
			$102.9$
			\\
			\hline
		\end{tabular}
	}
	\vspace{3mm}
	\caption{Mean and variances of the running times (s) for \RIVET{}, the ERS and
		\MMA{}. We provide both  degree 0 (left) and 1 (right) homology timings for \RIVET{}, whereas the timings of \MMA{} include both.
		The double dashes correspond to out of memory errors, i.e., a memory usage that is over $12\mathrm{GB} $.
	}
	\label{tab:TSres-small}
\end{table}

Interestingly, computing 0-dimensional homology is sometimes slower than
1-dimensional homology for \RIVET{}; as it relies on computing minimal free
presentations, we think that this comes from the fact that minimal
presentations in homology dimension 0 can be more complex than their counterparts in homology dimension 1, i.e., they have much     %
more generators. %
We also investigate how running times of \MMA{} depend on the number of lines.
Unsurprisingly, one can see from Figure~\ref{fig:UCR_3F_H0_time} that running
time increases with the number of lines. However, the dependency looks
empirically sublinear, which could come from the fact that even though there
are more lines, these lines are closer to each other, and thus matching them
with vineyards requires fewer computation steps. This is also highlighted by
the running times of \texttt{LargeHypoxicRegion2}, Table \ref{tab:TSres-small} which are smaller when computing it over $1~000$ lines than $100$ lines.

\begin{figure}
	\centering
	\begin{minipage}{.45\textwidth}
		\centering
		\includegraphics[width=0.99\textwidth]{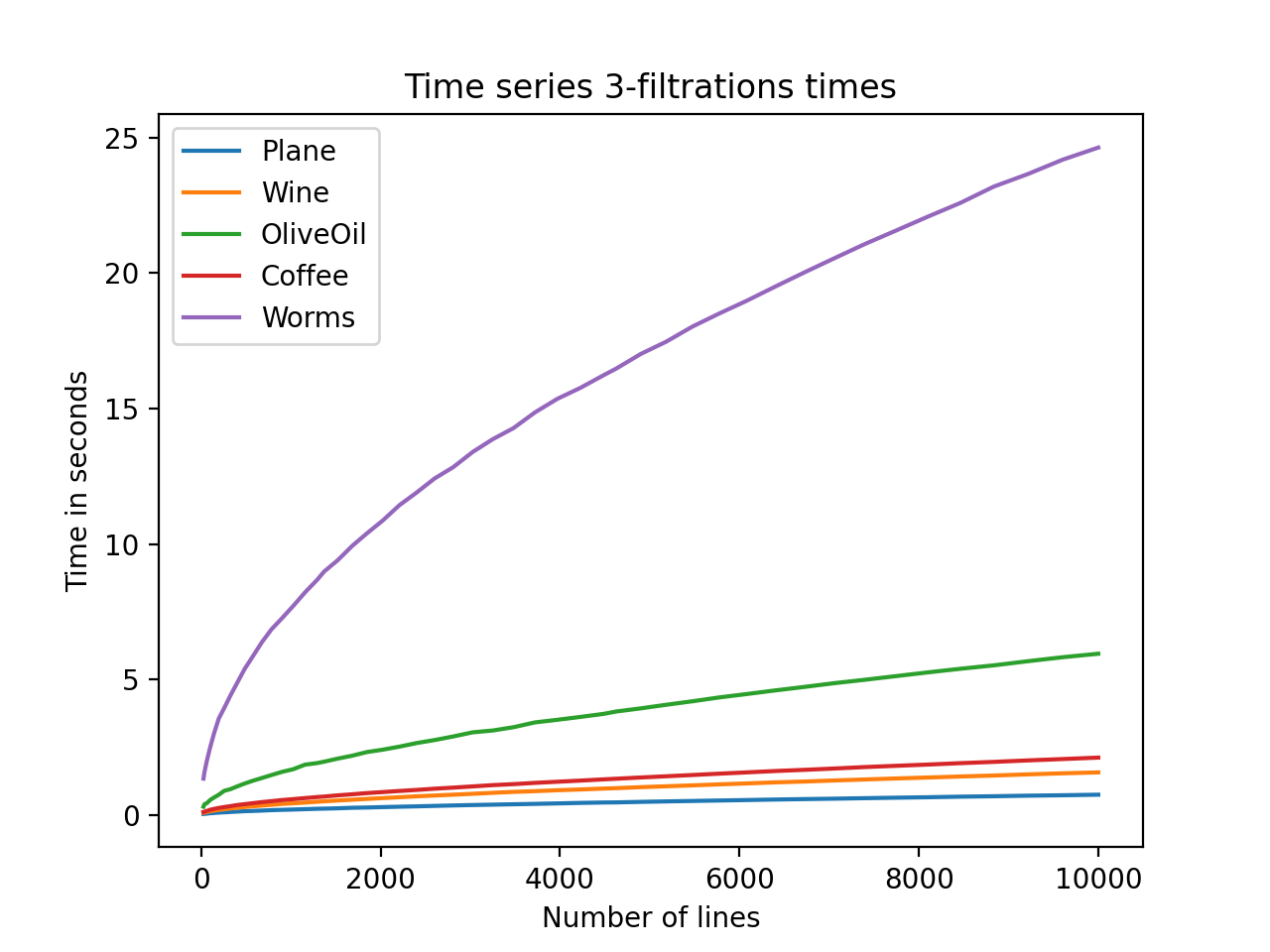}
		\captionof{figure}{Running time (s) needed to run \MMA{} on the \texttt{UCR} datasets.
		}
		\label{fig:UCR_3F_H0_time}
	\end{minipage}%
	\ \ \ \ \ \ \ \ \ \ \
	\begin{minipage}{.45\textwidth}
		\centering
		\includegraphics[width=0.99\textwidth]{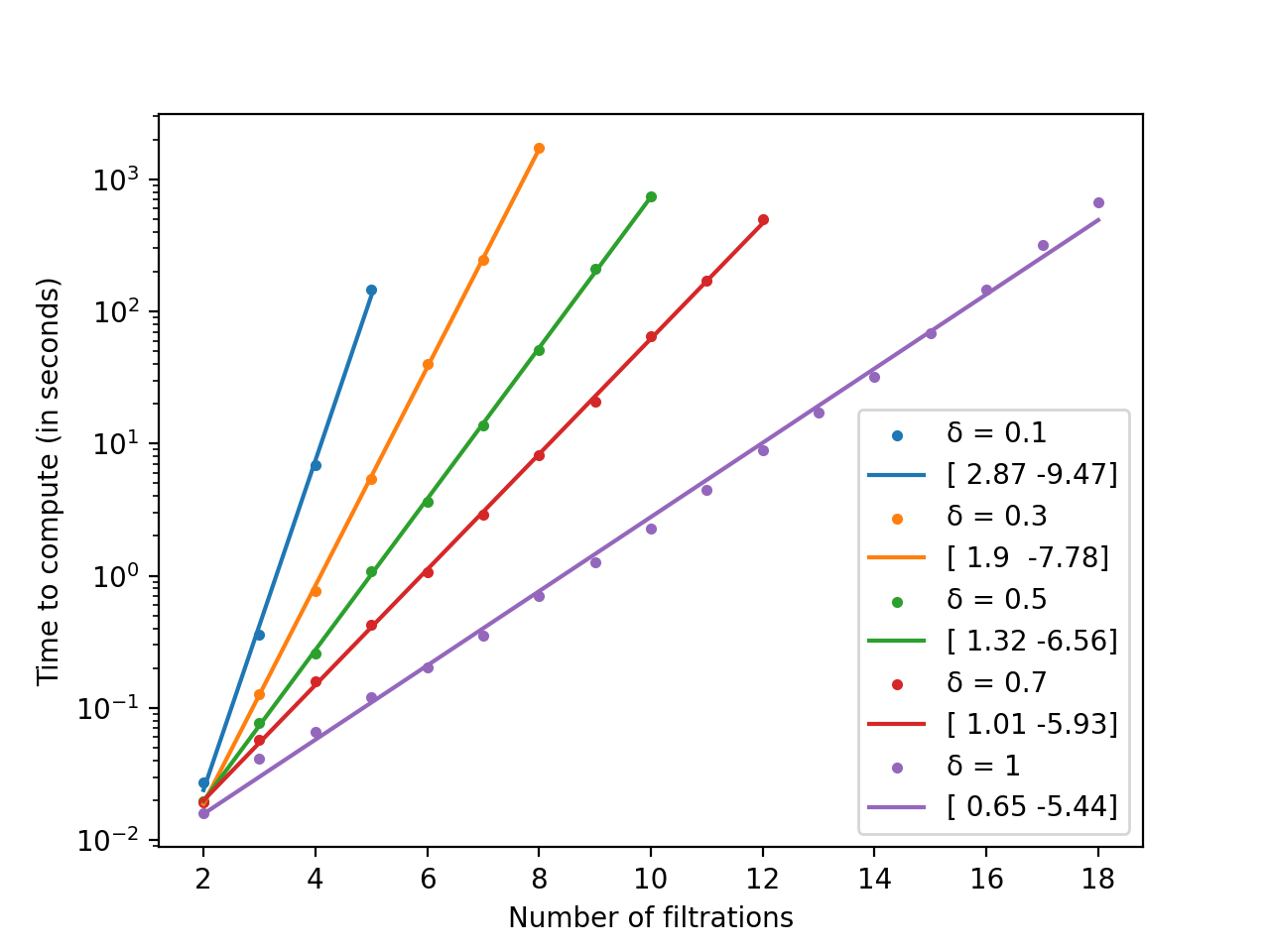}
		\captionof{figure}{Running time (s) of \MMA{} w.r.t. the number $n$ of filtrations.
		}
		\label{fig:time_vs_nfiltration}
	\end{minipage}
\end{figure}

\paragraph*{Dependence on number of filtrations.} %
Finally, we investigate how the running times of \MMA{} depend on the number
$n$ of filtrations.  Although most of the approaches in the literature are
limited to $n=2$ parameters, one can see from Figure~\ref{fig:time_vs_nfiltration} that \MMA{} can produce outputs in a reasonable amount of time for up to $n\simeq 10$ parameter filtrations.
As expected from the complexity of \MMA{} in Equation~(\ref{eq:MMA_complexity}), the running times increase exponentially with $n$.

\section{Conclusion}
\label{sec:conclusion}

\rebuttal{In this article, we present \MMA: a new algorithm for computing topological \rebuttal{descriptors} for multi-parameter persistence modules taking the form of \mmaout{} decompositions. Our algorithm %
	has complexity and running time
	that are controlled by user-defined parameters,
	and enjoy several approximation and stability properties.
	We also showcased the performances of \MMA{} on
	synthetic and real data sets.
	Our code is \rebuttal{
		part of the \texttt{multipers} library \cite{loiseauxMultipersMultiparameterPersistence2024} and is} publicly available at \url{https://github.com/DavidLapous/multipers}.} \\

\rebuttal{
	Along the way, we identified
	several directions
	for future work.
	\begin{enumerate}
		\item  %
		      While the outputs of \MMA{} satisfy approximation guarantees when computed with compatible matching functions, they remain arbitrary to some extent, as discussed at the end of related work and \cref{fig:non-stable-decomposition}. %

		      We hypothesize that, in the general case, the family $\mathcal F$ of \mmaoutcompat{} decompositions obtained from \MMA{} by varying $\matching$ across an appropriate family of compatible matching functions $\Sigma$, i.e., $\mathcal F=\{\MMA{}(\M,L,\sigma)\}_{\sigma\in \Sigma, \delta >0}$,
		      is a {\em complete} topological invariant of the module $\Mbb$.

		\item Practically speaking, the existence and construction of compatible matching functions for $n$-parameter persistence modules with $n>2$ is an open question. We hypothesize that matching functions computed from tracking representative cycles, in a similar way than
		      the construction of the graphcode~\cite{russoldGraphcodeLearningMultiparameter2024, kerberRepresentingTwoparameterPersistence2025}, could provide a step towards that direction.
		      Another possibility includes designing a convex optimization problem that would converge to compatible (and potentially stable) matching functions.

	\end{enumerate}
}

\section*{Statements and Declarations}
\paragraph{Ethical Approval}
Not applicable.
\paragraph{Competing interests}
No competing interests.
\paragraph{Authors' contributions}
D.L. M.C. and A.B worked out the proofs and wrote the manuscript, D.L. did the numerical experiments. All authors reviewed the manuscript. 
\paragraph{Funding}
D.L. was funded by the French government through the 3IA Côte d’Azur Investments, ANR-19-P3IA-0002.
M.C. was supported by Agence Nationale de la Recherche through ANR JCJC TopModel ANR-23-CE23-0014, and by the French government, through the 3IA Cote d’Azur Investments in the project
managed by the National Research Agency (ANR) with the reference number ANR-23-IACL-0001.
\paragraph{Availability of data and materials.}
All of the non-synthetic data sets used are publicly available, and listed below.
\begin{enumerate}
    \item The immunohistochemistry datasets can be retrieved from
    \\
    \url{https://github.com/MultiparameterTDAHistology/SpatialPatterningOfImmuneCells}, and
    \item The time series datasets can be retrieved from \url{http://www.timeseriesclassification.com/dataset.php}.
\end{enumerate}
All of the code used is publicly available as well.
\begin{enumerate}
    \item Our code, in \texttt{multipers}, at \url{https://github.com/DavidLapous/multipers},
    \item \texttt{Rivet}, at \url{https://github.com/rivetTDA/rivet/},
    \item ERS, at \url{https://github.com/Chen-Cai-OSU/ER-staircode}.
\end{enumerate}

\bibliography{biblio}

\end{document}